\documentclass[12pt,reqno]{amsart} 
\usepackage{a4wide} 
\usepackage{amsmath,amsthm,amssymb,bbm,bm}
\usepackage{amsthm,amssymb,mathrsfs,mathtools,empheq}

\usepackage{xcolor}
\usepackage{verbatim}
\usepackage{todonotes}
\usepackage{mathrsfs}
\usepackage{ulem}
\usepackage{cancel}

\newcommand{\fcal}{\mathcal{F}}
\newcommand{\p}{\mathbb{P}}
\newcommand{\tp}{\tilde{\mathbb{P}}}
\newcommand{\loc}{\mathrm{loc}}

\newcommand{\tOmega}{\tilde{\Omega}}
\newcommand{\tfcal}{\tilde{\fcal}}
\newcommand{\tW}{\tilde{W}}
\newcommand{\tu}{\tilde{u}}

\newcommand\law[2]{\mathrm{Law}_{{#1}}(#2)}

\mathtoolsset{showonlyrefs}

\usepackage{thmtools}
\usepackage[pdfdisplaydoctitle,colorlinks,breaklinks,urlcolor=blue,linkcolor=blue,citecolor=blue]{hyperref}
\usepackage{enumitem}
\newcommand\newW{\mathbf{W}}
\newcommand\Yg{Y_2}
\newcommand\Yb{Y_1}

\usepackage{color}

\newcommand\dela[1]{}

\newcommand\embed{\hookrightarrow}

\newcommand{\newr}{r}
\newcommand{\news}{s}

\newtheorem{Satz}{Theorem}[section]
\newtheorem{theorem}[Satz]{Theorem}
\newtheorem{lemma}[Satz]{Lemma}		
\newtheorem{corollary}[Satz]{Corollary}	
\newtheorem{proposition}[Satz]{Proposition}	
\numberwithin{equation}{section}
\theoremstyle{definition}
\newtheorem{definition}[Satz]{Definition}
\newtheorem{remark}[Satz]{Remark}
\newtheorem{assumption}[Satz]{Assumption}
\newtheorem{example}[Satz]{Example}

\newcommand{\C}{\mathbb{C}} 
\newcommand{\R}{\mathbb{R}} 
\newcommand{\Rd}{{\mathbb{R}^d}} 
\newcommand{\N}{\mathbb{N}} 

\newcommand{\E}{\mathbb{E}}
\newcommand{\Prob}{\mathbb{P}}
\newcommand{\F}{\mathscr{F}}
\newcommand{\Filtration}{\mathbb{F}}
\newcommand{\tildeProb}{\tilde{\Prob}}

\newcommand{\energy}{{\mathscr E}}
\newcommand{\EA}{{V}}
\newcommand{\EAdual}{{V^\ast}}

\newcommand{\LzweiTimeSum}{{L^2([s,t]\times \N)}}

\newcommand{\LalphaPlusEins}{{L^{\alpha+1}}}

\newcommand{\LalphaPlusEinsDual}{{L^{\frac{\alpha+1}{\alpha}}}}

\newcommand{\D}{\mathscr{D}}
\newcommand{\df }{\,\mathrm{d}}
\newcommand{\im }{\mathrm{i}}
\newcommand{\sumM }{\sum_{m=1}^{\infty}}
\newcommand{\Real}{\operatorname{Re}}
\newcommand{\skpH}[2]{\big(#1,#2\big)_{H}}
\newcommand{\skp}[2]{\big(#1,#2\big)}
\newcommand{\skpHReal}[2]{\Real \big(#1,#2\big)_{H}}
\newcommand{\sqrtA}{A^{\frac{1}{2}}}

\newcommand{\norm}[1]{\Vert #1 \Vert}
\newcommand{\quadVar}[1]{\langle \langle #1 \rangle \rangle}
\newcommand{\duality}[2]{\langle #1, #2 \rangle}
\newcommand{\dualityReal}[2]{\Real \langle #1, #2 \rangle}
\newcommand{\Fhat}{\hat{F}}

\newcommand{\Addresses}{{
		\bigskip
		\footnotesize
		
		Zdzis{\l}aw ~Brze{\'z}niak, Department of Mathematics, University of York,
			Heslington, York, YO105DD, UK\par\nopagebreak
		\textit{E-mail address}: zdzislaw.brzezniak@york.ac.uk
		\medskip
		
		Benedetta Ferrario, Dipartimento di Scienze Economiche e Aziendali, Universit\`a di Pavia, 27100 Pavia, Italy  \par\nopagebreak
		\textit{E-mail address}: benedetta.ferrario@unipv.it
		
		\medskip
		
Margherita Zanella, Dipartimento di Matematica "Francesco Brioschi", Politecnico di Milano, Via Bonardi 13, 20133 Milano, Italy \par\nopagebreak
		\textit{E-mail address}: margherita.zanella@polimi.it	
	}}

\title[Invariant measures for stochastic damped  Schr\"odinger equation]{Invariant measures for a stochastic nonlinear and damped  2D Schr\"odinger equation}
\author[Z. Brze{\'z}niak, B. Ferrario and M. Zanella ]{Zdzis{\l}aw Brze{\'z}niak, Benedetta Ferrario and Margherita Zanella	}

\date{\today}
\begin{document}

\begin{abstract}
We consider a stochastic nonlinear defocusing Schr\"{o}dinger equation  with zero-order linear damping,
where the stochastic forcing term is given by  a combination of a  linear multiplicative noise in the Stratonovich form and a
nonlinear noise in the It\^o form.
We work at the same time on compact Riemannian manifolds without boundary and on relatively compact  smooth domains
with either the Dirichlet or the Neumann boundary conditions,  always in dimension two. We construct a martingale solution using a
modified Faedo-Galerkin's method, following \cite{Brz+H+W-2019}. Then by means of the Strichartz estimates deduced  from
\cite{Blair+Sogge_2008} but modified for our stochastic setting we show the pathwise
uniqueness of solutions.
Finally, we prove  the existence of an invariant measure by means of a version of the Krylov-Bogoliubov method,
which involves the weak topology, as proposed by Maslowski and Seidler \cite{Mas-Seid}. This is the first result of this type
for stochastic NLS on compact Riemannian manifolds without boundary and on relatively compact  smooth domains even for an additive noise.
Some remarks on the uniqueness in a particular case are provided as well.
\end{abstract}

\maketitle


\tableofcontents

\section{Introduction}

Let us consider the following nonlinear damped stochastic Schr\"{o}dinger equation
\begin{align}
\label{eqn-ProblemStratonovich}
\df u(t)&= -\left[\im A u(t)+\im  F(u(t)) +\beta u(t) \right] dt  -\im B u(t) \circ \df W(t)
\\
\nonumber
&\hspace{6.2truecm}-\im G(u(t)) \,\df \newW(t) ,\qquad t> 0
\end{align}
Here $A$ is a linear self-adjoint  non-negative operator,
$\beta$  is a damping constant (usually considered not negative),
$B$ is a linear bounded operator; $F$  and $G$  are nonlinear terms. Moreover
$W$ and $\newW$  are two independent  Wiener processes;
the first  stochastic differential is in the Stratonovich form and the other one is  in the It\^o form.

A basic example of the operator  $A$ is the negative Laplace-Beltrami operator $-\Delta_g$ on a compact Riemannian
manifold $(M,g)$ without boundary; this appears in some previous papers on the nonlinear 
Schr\"{o}dinger
equation, see for instance \cite{Burq+G+T_2004}, \cite{BrzezniakStrichartz}, \cite{Brz+H+W-2019}.
However, we can deal as well with  the negative Laplacian $-\Delta$ on a
relatively compact  smooth domain in  $\Rd$ with Neumann or Dirichlet boundary conditions.
We will consider $F$ to be a  power-type defocusing nonlinearity.

The nonlinear Schr\"odinger equation  occurs as a basic model in many areas of physics:
hydrodinamics, plasma physics, nonlinear optics, molecular biology. It describes the propagation of waves media with both
nonlinear and dispersive responses, see, e.g.,  \cite{SulSul} where many physical models are discussed.
A lot of attention has been paid
recently to the influence of a noise on the dynamics described by the equation:
the noise acts as a random potential to incorporate spatial and temporal fluctuations of certain parameters in a physical
model. Typically the noise depends on the solution itself and, according to the physical situation one aims at describing,
the It\^o or the Stratonovich forms of the noise can be taken into account.
These different types of stochastic differential lead in fact to different properties of solutions.
For instance in the multiplicative noise model the mass is preserved only
when the noise is taken in the Stratonovich sense of a particular type.
This provides the same property as for  the deterministic equation.
However, with  both  a Stratonovich and  an It\^o noise, the energy is not conserved anymore in general.
For this reason, when addressing the problem of the existence of an invariant measure, some dissipative terms (e.g.: our model with $\beta>0$) are usually involved in the stochastic case.

The question of the existence and/or the uniqueness of solutions for nonlinear Schr\"{o}dinger equations  with additive or linear
multiplicative noise was previously addressed
in the $\Rd$-case by De Bouard and Debussche \cite{BouardLzwei, BouardHeins},
Barbu, R\"{o}ckner and Zhang \cite{BarbuL2, BarbuH1, Barbu17},
Cui, Hong and Sun in \cite{Cui}
and Hornung  \cite{Hornung-F_PhD} and \cite{Hornung_2018}. In the case of compact two dimensional Riemannian manifolds there are results
 by Brze{\'{z}}niak and Millet \cite{BrzezniakStrichartz} and by  Brze{\'{z}}niak, Hornung and Weis
\cite{Brz+H+W-2019}.

So far  the existence  of an invariant measure has been obtained for this equation with  a damping term
and an additive It\^o  noise, i.e.  $\beta>0$, $B=0$ and $G$ constant in equation \eqref{eqn-ProblemStratonovich}.
In this framework, in the papers by Kim \cite{Kim_2006} and by Ekren,  Kukavica and  Ziane \cite{Ekren_2017} the result is proved in the full space $\mathbb{R}^d$, $d\ge 1$. Debussche and Odasso \cite{Odasso} obtain instead the result on a bounded one-dimensional domain (dealing with the cubic
focusing Schr\"odinger equation)  and solved  the corresponding uniqueness problem too.
The recent paper \cite{noi2} proves the uniqueness of the invariant measure in the large damping regime in $\mathbb{R}^d$ for $d=2,3$.
Some numerical approximations of invariant measures can be found in the book by Hong and Wang \cite{Hong+Wang_2019_invariant}.
The aim of our paper is to generalize the previously cited papers by considering a more general stochastic forcing term: we consider
 a linear multiplicative Stratonovich noise $B$, which conserves the $H$-norm,  and a nonlinear It\^o noise $G$.
When we reduce to the case of a pure additive noise we get the existence of an invariant measure when $\beta>0$. This finding is in line with \cite{Kim_2006} and \cite{Ekren_2017} that obtain the same conclusions working on the full space. We emphasize here that, as far as we know, our result is the first one providing the existence of an invariant measure in the case of a two-dimensional compact Riemannian manifold without boundary and of a relatively compact  smooth domain
with either the Dirichlet or the Neumann boundary conditions.

The paper is structured as follows. In the first part of the paper we shall construct a martingale solution of problem $\eqref{eqn-ProblemStratonovich}$
in the energy space $\EA=\D(\sqrtA)$ by a modified
Faedo-Galerkin's approximation, following the lines
of a previous paper by the first author and collaborators \cite{Brz+H+W-2019}.  However here we generalize that  setting by   dealing with a random initial data and more general diffusion terms. One should mention here that a very recent  paper
\cite{Hornung_2020} provides another generalization \cite{Brz+H+W-2019} in the direction of stochastic NSL equations on unbounded domains and non-compact manifolds.

In the second part of the paper we shall prove the pathwise uniqueness of the solutions.
Hence the existence and the uniqueness of a strong solution will follow.
This result is based on further regularity properties of the martingale solutions,
which are obtained by means of the Strichartz estimates in dimension two. Although the proof of our existence and uniqueness result follows the lines of the proof of the analogous
 results in \cite{Brz+H+W-2019}, we emphasize here that we allow the initial data to be random.

 As far as the existence of an invariant measure is concerned, in the last  part of the paper  we proceed differently from  \cite{Kim_2006}  and
\cite{Ekren_2017}. Following the proof of the existence of a martingale solution  we prove that the corresponding Markov semigroup is
sequential weak Feller in the energy space  $\EA$.
Moreover we show a tightness result in the space $\EA$ equipped with the weak topology, when the damping coefficient is
sufficiently large. In this case  a new condition involving the strengthes  of the two noises will appear.
With these two latter properties we prove the existence of at least one invariant measure, by means of the method introduced
 by Maslowski and Seidler \cite{Mas-Seid}, as a version of the classical Krylov-Bogoliubov technique reset with weak
 topologies. This method has been successful to prove existence of invariant measures for other SPDE's, as the stochastic
 nonlinear beam and wave equations \cite{BOndSeid}, the Navier-Stokes equations in unbounded domains
 \cite{BMO, Brz+Ferr_2019}, the
stochastic Landau-Lifshitz-Bloch equation \cite{BGL}, the stochastic damped Euler equation \cite{BesFer}.

The paper is organized as follows.
In Section \ref{s-2} we present notations and main assumptions. In Section \ref{Section3MR} we state our main results. In Section
\ref{CompactnessSection} we collect some compactness results. Section \ref{sec-exis} deals with the existence of
martingale solutions. Pathwise uniqueness is proved in Section \ref{sec-UniquenessSection}.
The two last sections \ref{sec-weakCont} and \ref{s-inv} are concerned with the sequential weak  Feller property  and
the existence of invariant measures. In Section \ref{sect-unique}  we consider the particular case of multiplicative noise,
where there is also uniqueness of the invariant measure.
In Appendix \ref{app_Lapl} we recall some facts about Laplacian-type operators on manifolds and on bounded domains
with Dirichlet/Neumann boundary conditions and derive the needed Strichartz estimates; this is very different from the
setting in $\mathbb R^d$ considered in many papers. In Appendices \ref{App_B} and \ref{App_C} we collect the proofs of
some results.
 In Appendix \ref{sec-Yamada-Watanabe Theorem} we present the infinite dimensional version of the Yamada-Watanabe Theorem. In Section \ref{Technical_Lemma} we prove a technical lemma that we need in Section \ref{CompactnessSection}.

\subsection*{Acknowledgments}
B.F. and M.Z. thank GNAMPA-INdAM for financial support.
The authors thank the Hausdorff Institute for Mathematics in Bonn, where this work was started, for the kind hospitality.
The authors are grateful to Fabian Hornung for numerous useful comments and discussions. They also would like to thank Markus Kunze and Martin Ondrej\'at on discussion about their papers on the Yamada-Watanabe Theorem. Finally the authors would like to thank an anonymous referee for numerous comments and questions which have lead to an improved and clarified presentation.

\section{Mathematical setting and assumptions}  \label{s-2}

In this Section, we fix the notation, explain the assumptions and formulate the framework for our problem.
Let $\left(X,\Sigma,\mu_X\right)$ be a $\sigma$-finite measure space endowed with the metric $\rho$ such that the corresponding Borel $\sigma$-field
$\mathscr{B}(X)$ is contained in $\Sigma$. Let $D \subseteq X$.
Our framework covers the following cases
\begin{footnote}{From now on we will denote by $\mathscr{O}$ a subset of $\mathbb{R}^2$ and by $M$ a two-dimensional manifold. We will use the letter $D$ when we need to deal with the two cases above at the same time.
}\end{footnote}
:
\begin{itemize}
\item $X=\mathbb{R}^2$ with the Euclidean distance $\rho$ and the Lebesgue measure $\mu_X$ and $D=\mathscr{O}$ is a  relatively compact  smooth, i.e. with $C^{\infty}$ boundary, domain of $\mathbb{R}^2$.
\item $D=X\equiv M$ is a two-dimensional compact Riemaniann manifold without boundary with the geodesic distance $\rho$ and the canonical volume measure $\mu_X$ on $X$.
\end{itemize}

By $L^q(D)$, for $q\in [1,\infty]$, we denote the space of equivalence classes of $\mathbb{C}$-valued $q-$Lebesgue integrable functions.
We abbreviate $L^q:=L^q(D)$. For $q\in [1,\infty],$ let  $q':=\frac{q}{q-1}\in [1,\infty]$ be the conjugate exponent. We further abbreviate ${H}:=L^2$.
This is a complex Hilbert space with inner product $(u,v)_H=\int_D u \overline{v} \,{\rm d}\mu_X$.
However,  we often interpret $H$ as  a real Hilbert space with the inner product $\Real(u,v)_H$.
They are different but  in one-to-one correspondence: $(u,v)_H=\Real (u,v)_H+i\Real (u,iv)_H$.
These inner products introduce the same norms and hence both spaces are topologically equivalent.

By $H^{s,q}(D)$ we denote the fractional Sobolev space of regularity $s\in \R$ and integrability $q\in (1,\infty)$. We abbreviate $H^{s,q}:=H^{s,q}(D)$ and we shortly write $H^s:=H^{s,2}$.
For a definition of these spaces see Appendix \ref{app_Lapl}.

In the sequel, given two Banach spaces $E$ and $F$, we denote by $\mathscr{L}(E,F)$ the space of all linear bounded
operators $B: E\to F$ and abbreviate $\mathscr{L}(E):=\mathscr{L}(E,E).$  Furthermore, we write $E\hookrightarrow F$, if
$E$ is continuously embedded in $F$, i.e. $E\subset F$ with natural embedding  $j\in \mathscr{L}(E,F)$.
For a Hilbert space $H$ and a Banach space $E$, $\gamma(H,E)$ denotes the spaces of $\gamma$-radonifying operators from $H$ to $E$. If $E$ is a Hilbert space, this is indeed the space of  Hilbert-Schmidt operators from $H$ to $E$.
The space $C^{1,2}([0,T]\times E,F)$ consists of all functions $\varPhi\colon [0,T]\times E\to F$ such that
$\varPhi(\cdot,x)\in C^1([0,T],F)$ for every $x\in E$ and $\varPhi(t,\cdot)\in C^2(E,F)$ for every $t\in[0,T]$.
Given the Hilbert space $H$, $C_w([0,T];H)$ stands for  the space
of all continuous functions from the interval $[0,T]$ to  the space $H$ endowed with the weak topology.

If functions  $a,b\ge 0$ satisfy the inequality $a \le C(A) b$ with a constant $C(A)>0$ depending on the expression $A$, we write $a \lesssim_A b$; for a generic constant we put no subscript.
 If we have $a \lesssim_A b$ and $b \lesssim_A a$, we write $a \simeq_A  b$.

\subsection{Assumptions on the operator $A$}

\begin{assumption}
\label{assumption-A-space}
The operator $A$ that appears in equation \eqref{eqn-ProblemStratonovich} is a Laplacian-type operator. We consider
 $A$ to be as one of the following:
\begin{itemize}
\item [i)] the negative Laplace-Beltrami operator $-\Delta_g$ on a compact two-dimensional Riemannian manifold
$(M,g)$ without boundary, equipped with a Lipschitz metric $g$; in this case $\mu_X$ is the canonical volume measure;
\item [ii)] the negative Laplacian with Dirichlet boundary conditions $-\Delta_D$ on a smooth,
i.e. $C^{\infty}$, relatively compact  domain $\mathscr{O}$ of $\mathbb{R}^2$;
\item [iii)] the negative Laplacian with Neumann boundary conditions $-\Delta_N$ on a smooth,
i.e. $C^{\infty}$, relatively compact  domain $\mathscr{O}$ of $\mathbb{R}^2$.
\end{itemize}
\end{assumption}

Some classical results, see e.g. \cite{Ouh_2009},  ensure that the operator $A$, in any of the forms given in Assumption \ref{assumption-A-space}, is a non-negative self-adjoint operator on $H$. We denote by $\D(A)$ its domain. We set $\EA:=\D(\sqrtA)$ and note that $V$  is a Hilbert space when equipped with the inner product
		\begin{align*}
		\skp{u}{v}_{\EA}:=\skpH{u}{v}+\skpH{\sqrtA u}{\sqrtA v},\qquad u,v\in \EA.
		\end{align*}
We call it the \emph{energy space} and we call the induced norm $\norm{\cdot}_{\EA}$ the  \emph{energy norm} associated to $A$.
For a characterization of the energy spaces associated to the operators that appear in Assumption \ref{assumption-A-space},
see Remark \ref{B4} and Proposition \ref{propB6}(i). We denote the dual space of $\EA$ by $\EAdual$ and abbreviate the
duality with $\duality{\cdot}{\cdot}$,
where the complex conjugation is taken over the second variable of the duality. Note that $\left(\EA, H, \EAdual\right)$ is a
Gelfand triple, i.e. \begin{equation}\label{eqn-Gelfand triple}
\EA\hookrightarrow H \cong H^\ast \hookrightarrow \EAdual.
\end{equation}

Notice that, thanks to the geometry of the domain $D$, the condition $\alpha \in (1, \infty)$ ensures that the embedding
$
\label{eqn-V to L^ alpha+1} V \subset L^{\alpha+1}$
 is  compact (and hence bounded/continuous). Hence, since $(\LalphaPlusEins)^\ast=\LalphaPlusEinsDual$, we can extend the $\EA-\EAdual$  duality $\duality{\cdot}{\cdot}$
to the couple $\LalphaPlusEins-\LalphaPlusEinsDual$.
\\
Let us point out that we also have the compact (and hence bounded/continuous) embedding
 \begin{equation}\label{eqn-V to H} V \subset H.
 \end{equation}

It can be proved, see e.g. \cite[Lemma 2.3(a)]{Brz+H+W-2019}, that there exists a non-negative self-adjoint operator $\hat{A}$
on $\EAdual$ with $\D(\hat{A})=\EA$ and $\hat{A}=A$ on $\D(A)$.
In most cases where this does not cause ambiguity or confusion, we also use the notation $A$ for $\hat{A}$.

In the following Lemma we introduce the operator $S$ and state some properties of it. $S$ will play the role of an auxiliary operator to cover the different cases we consider in an unified framework.
\begin{lemma}
\label {S_properties}
Given $A$ as in Assumption \ref{assumption-A-space}, there exists an operator $S$ on $H$ such that
\begin{itemize}
\item [i)]$S$ is strictly positive and self-adjoint. $S$ commutes with $A$ and   satisfies
		$\D(S^k)\hookrightarrow \EA$
		for sufficiently large $k.$ Moreover, $S$
			satisfies the upper Gaussian estimate i.e. for all $t >0$ there is a measurable function $p(t, \cdot, \cdot):D \times D \rightarrow \mathbb{R}$ with
	\begin{equation*}
\left(e^{-tS}f\right)(x)= \int_Dp(t,x,y)f(y)\, \mu_X({\rm d}y), \quad t >0, \quad \text{a.e.} \ x \in D,
	\end{equation*}
for all $f \in H$ and with constants $c,C>0$ and $m \ge 2$
\begin{equation}
\label{eqn_Gauss_est}
|p(t,x,y)| \le \frac{C}{\mu_X(B(x, t^{\frac 1m})}
\exp\left(-c \left(\frac{\rho(x,y)^m}{t} \right)^{\frac{1}{m-1}} \right),
\end{equation}
for all $t>0$ and a.e. $(x,y) \in D \times D$.
\item [ii)] $S$ has compact resolvent. In particular, there is an orthonormal basis $\left(h_n\right)_{n\in \N}$ of $H$ and a nondecreasing sequence $\left(\lambda_n\right)_{n\in\N}$ with $\lambda_n>0$ and  $\lambda_n\to \infty$ as $n\to \infty$ such that
		\begin{align}
		\label{eqn_spectral_S}
		S x=\sum_{n=1}^\infty \lambda_n \skpH{x}{h_n} h_n, \quad x\in \D(S)=\left\{x\in H: \sum_{n=1}^\infty \lambda_n^2 \vert \skpH{x}{h_n}\vert^2<\infty\right\}.
		\end{align}
\end{itemize}
\end{lemma}
\begin{proof}
For $A=-\Delta_g$ we choose $S:=I-\Delta_g$, for $A=-\Delta_D$ we choose $S=A=-\Delta_D$, for $A=-\Delta_N$ we fix $\varepsilon>0$ and choose $S=\varepsilon I-\Delta_N$. For these choices of $S$ all the  statements of the Lemma are verified: for a proof see \cite[Sections 3.2 - 3.3, Remark 2.2(b) and Lemma 2.3(c)]{Brz+H+W-2019} and the therein references.
\end{proof}

\begin{remark}
The operator $S$ plays a crucial role in the construction of our Galerkin approximations.
The Gaussian estimate \eqref{eqn_Gauss_est} is used in the proof of Proposition \ref{PaleyLittlewoodLemma} where spectral multiplier theorems for $S$ are employed (for further details one can consult \cite[Chapter 7]{Ouh_2009}).
\end{remark}

\subsection{Assumptions on the nonlinear term $F$}
We continue with the assumptions on the nonlinear term $F$ of our problem. We deal with power-type defocusing
nonlinearities.

\begin{assumption}
\label{assumption-F_def}
Assume that $\alpha \in (1, \infty)$ and set
\begin{equation*}
F(u):=|u|^{\alpha-1}u.
\end{equation*}
\end{assumption}

The function $F$ satisfies a set of properties that we summarize in the following Lemma (for a proof see e.g. \cite[Proposition 3.25 and Remark 3.16]{Cazenave} and \cite[Proposition 3.1]{Brz+H+W-2019}).
It is important to recall that the embedding $V \subset L^{\alpha+1}$ is  compact (and hence bounded/continuous).
Therefore we have
\begin{equation}\label{eqn-Gelfand triple}
\EA \embed   \LalphaPlusEins \embed L^2 \equiv (L^2)^\ast \embed \LalphaPlusEinsDual \embed \EAdual,
\end{equation}
where the first and the last embeddings are compact, while all other embeddings are simply continuous.

\begin{lemma}\label{lem-nonlinearAssumptions}
	Let $\alpha \in (1, \infty)$.
	\begin{itemize}
	\item[i)] The map $F:L^{\alpha +1} \to \LalphaPlusEinsDual$ satisfies,   for any $ u,v\in \LalphaPlusEins$
		\begin{align}\label{eqn_nonlinearityEstimate}
		\norm{F(u)}_\LalphaPlusEinsDual = \norm{u}_\LalphaPlusEins^\alpha,
		\end{align}
		\begin{align}\label{eqn_nonlinearityLocallyLipschitz}
                 \norm{  F(u)-F(v)}_{\LalphaPlusEinsDual}
                &\lesssim
                \left(\norm{u}_\LalphaPlusEins+\norm{v}_\LalphaPlusEins\right)^{\alpha-1} \norm{u-v}_\LalphaPlusEins.
                \end{align}	
      	Moreover, $F: \EA \to \EAdual$, $F(0)=0$ and
                \begin{align}\label{eqn_nonlinearityComplex}
		\Real \duality{\im u}{F(u)}=0, \quad u\in \LalphaPlusEins 	,	\end{align}
		\begin{align}\label{ineq-dissipativity}
		 \duality{F(u)}{u} = \norm{u}_\LalphaPlusEins^{\alpha+1} ,\quad u\in \LalphaPlusEins.
		\end{align}
	\item[ii)] The map $F: \LalphaPlusEins\to \LalphaPlusEinsDual$ is continuously 
	 Fr\'{e}chet differentiable with
		\begin{align}\label{eqn_deriveNonlinearBound}
		\Vert F^{\prime}[u]\Vert_{L^{\alpha+1}\to L^\frac{\alpha+1}{\alpha}} \le \alpha \norm{u}
		 _\LalphaPlusEins^{\alpha-1}, \quad u\in \LalphaPlusEins.
		\end{align}
	\item[iii)]The map $F$ is defocusing, that is it admits the real non-negative antiderivative
	\begin{footnote}
	{
	We recall that, if there exists a Fr\'echet differentiable map $\hat F:L^{\alpha+1} \rightarrow \mathbb{R}$ with $\hat F'[u]h=\text{Re}\langle F(u),h\rangle$, for every $u,h \in L^{\alpha+1}$, $\hat F$ is called the antiderivative of $F$.
	}
	\end{footnote}$\hat{F}: \LalphaPlusEins\to \R$ given by
\begin{align}
\label{eqn_boundantiderivative}\hat F(u)=\frac 1{\alpha+1}\|u\|_{L^{\alpha+1}}^{\alpha+1}.
\end{align}			
\end{itemize}
\end{lemma}

\begin{remark}
\label{remnew}
It follows from \eqref{eqn_nonlinearityEstimate} and  \eqref{eqn_boundantiderivative}  that
		\begin{align}\label{ineq-F}
		\norm{F(u)}_\LalphaPlusEinsDual
		=\norm{u}_\LalphaPlusEins^\alpha
		\lesssim[\Fhat(u)]^{\frac{\alpha}{\alpha+1}}, \; \quad u\in \LalphaPlusEins.
		\end{align}
Moreover, it follows from \eqref{eqn_deriveNonlinearBound} and  \eqref{eqn_boundantiderivative} that
		\begin{align}\label{ineq-Fp}
		\|F'[u]\|_{L^{\alpha+1}\rightarrow L^{\frac{\alpha+1}{\alpha}}} \le\alpha\|u\|^{\alpha-1}_{L^{\alpha+1}}
		= \alpha (\alpha+1)^{\frac{\alpha-1}{\alpha+1}} [\hat F(u)]^{\frac{\alpha-1}{\alpha+1}} .
		\end{align}
\end{remark}


\subsection{Assumptions on the stochastic terms}

For a probability space $\left(\Omega, \F, \Prob\right)$ and a measurable space $(E,\mathscr{E})$, the law of a random variable $\xi: \Omega \to E$ will be denoted by $\law{\Prob}{\xi}$.  

\begin{assumption}\label{assumption-stochastic}
	We assume the following.
	\begin{itemize}
		\item[i)] Let $(\Omega,\F,\Prob,\Filtration)$, where $\Filtration=\bigl(\mathscr{F}_t\bigr)_{t\geq 0}$,  be a  filtered probability space satisfying the usual conditions, $\Yb$ and $\Yg$ two separable real Hilbert spaces, with orthonormal bases  $(f_m)_{m\in\N}$ and $(e_m)_{m \in \mathbb{N}}$ respectively, and $W$ and
		$\newW$ two independent, $\Yb$  respectively
		$\Yg$, canonical cylindrical $\Filtration$-Wiener processes.
		 		\item[ii)] 		Let $B: {H} \to \gamma(\Yb,{H})$ be a linear
		operator and set $B_m u:=B(u)f_m$ for $u\in {H}$ and $m\in \N.$ Additionally, we assume that
		$B_m\in\mathscr{L}(H)$ is self-adjoint for every $m\in\N$ and the following stronger assumption, needed to make sense of the Stratonovich correction terms,
is satisfied
		\begin{align}
  \label{noiseBoundsH}
		\sumM \norm{B_m}_{\mathscr{L}({H})}^2<\infty.
		\end{align}
Moreover we assume that $B_m\in \mathscr{L}(\EA)$ and  $B_m\in \mathscr{L}(\LalphaPlusEins)$ for all  $m\in\N$  and
		\begin{align}\label{noiseBoundsEnergy}
		&\sumM\norm{B_m}_{\mathscr{L}(\EA)}^2<\infty,\\
\label{noiseBounds-L^alpha}
		&\sumM \norm{B_m}_{\mathscr{L}(L^{\alpha+1})}^2<\infty.
		\end{align}	
\item[(iii)]
Let $G:{H} \to \gamma(\Yg,{H})$ be Lipschitz continuous,  i.e.
 \begin{equation}
 \label{Lipschitz_G}
\exists \ L_G >0 : \quad  \|G(u_1)-G(u_2)\|_{\gamma(\Yg,H)} \le L_G\|u_1-u_2\|_H \qquad \forall u_1, u_2 \in H.
\end{equation}
Moreover the following  ``restrictions'' of $G$, i.e.
\begin{align*}
G: \EA \to \gamma(\Yg,\EA) \mbox{ and } G: \LalphaPlusEins \to \gamma(\Yg,\LalphaPlusEins)\end{align*}
are measurable, see \cite[Section 2]{Brz+Rana_2021} for a reasonably thorough discussion of this issue,  and of at most  linear growth, i.e.
  for some non negative constants $C_2, \tilde C_2, C_3, \tilde C_3$, the following inequalities hold
 \begin{equation}\label{crescitaGEA}
   \|G(u)\|_{\gamma(\Yg,\EA)} \le C_2+\tilde C_2\|u\|_{\EA}\qquad \forall u \in \EA,
\end{equation}
and
 \begin{equation}\label{crescitaGL}
\|G(u)\|_{\gamma(\Yg, \LalphaPlusEins )} \le C_3+\tilde C_3\|u\|_{ \LalphaPlusEins }\qquad \forall u \in  \LalphaPlusEins.
\end{equation}
 \newline
Finally, we also assume the following weak continuity assumption of the diffusion coefficient $G$:
for every $m \in \mathbb{N}$ the map
\[ H \ni \varphi \mapsto G(\varphi)e_m \in  H \]
extends uniquely to a continuous  map from $V^\ast$ to $V^\ast$, i.e.

\begin{equation}
 \label{cont_V^*}
 \text{the map} \ V^\ast \ni \varphi \mapsto G(\varphi)e_m \in V^\ast \ \text{is continuous, }  m \in \mathbb{N}.
 \end{equation}
\end{itemize}

\end{assumption}

\begin{remark}\label{Rem-linear growth} It is well known that the Lipschitz assumption \eqref{Lipschitz_G} implies the following linear growth condition. There exist positive constants $C_1, \tilde C_1$ such that
\begin{equation}\label{eqn-G crescita}
\|G(u)\|_{\gamma(\Yg,H)} \le C_1+\tilde C_1\|u\|_H\qquad \forall u \in H.
\end{equation}
 We mention this obvious fact since  we explicitly use this estimate \eqref{eqn-G crescita} in  many  computations in the following sections.
\end{remark}

\begin{remark}
By assumption  \ref{assumption-stochastic} i), we can represent the Wiener processes as
\[
W(t)=\sumM f_m \beta_m(t) \qquad\qquad \newW(t)=\sumM e_m {\bm\beta}_m(t)
\]
for two sequences of independent standard real Wiener processes $\{\beta_m\}_m$ and $\{\bm\beta_m\}_m$.
\end{remark}
\begin{remark}\label{rem-HS}
	The estimates $\eqref{noiseBoundsH}$, $\eqref{noiseBoundsEnergy}$ and \eqref{noiseBounds-L^alpha} imply
	\begin{align*}
	B\in \mathscr{L}({H},\gamma(\Yb,{H})),\quad B\in \mathscr{L}(\EA,\gamma(\Yb,\EA)),\quad B\in \mathscr{L}(\LalphaPlusEins,\gamma(\Yb,\LalphaPlusEins)).
	\end{align*}
\end{remark}

\begin{remark}
The property \eqref{cont_V^*} will be exploited in the proof of Lemma \ref{convquadvari} given in Appendix C. The corresponding property for $B$ is not needed since the analogue of Lemma \ref{convquadvari} for $B$ can be proved exploiting the selfadjointness of the operators $B_m$, $m \in \mathbb{N}$. For more details see \cite[Lemma 6.3, step 4]{Brz+H+W-2019}.
\end{remark}

\begin{example}
Examples of operator $B$  satisfying the required properties can be found  in \cite[Section 3.5]{Brz+H+W-2019}.
The self-adjointness of $B$ is crucial there, see \cite[Remark 3.7]{Brz+H+W-2019}.
\\
Concerning the second operator $G$,
in the case it is linear (and examples of such are the same as for $B$) we do not require it to be self-adjoint.
An example of a nonlinear operator $G$ can be constructed as done in \cite[Example 2.3]{BF17} and \cite[Section 2.3]{FZ18}. For any $m \in \mathbb{N}$, let $G(u)e_m:=c_m\sigma(u)h_m$, with  $c_m \in \mathbb{R}$ such that \[\sum_{m=1}^{\infty} c_m^2\|h_m\|^2_V < \infty\]
and, for a fixed and given $k \in V$,
\[
\sigma: V^\ast \ni u \mapsto \frac{\langle u,k\rangle^2}{1+\langle u, k\rangle^2}  \in  \mathbb{R}.\]
 It can be easily verified that this operator satisfies Assumption \ref{assumption-stochastic}(iii).
\end{example}

\section{Statement of the main results}
\label{Section3MR}
This Section is devoted to the statements of our main results. The first result concerns the existence of a unique strong solution to \eqref{eqn-ProblemStratonovich} for a random initial data. The second result concerns the existence of an  invariant measure.

We rewrite equation \eqref{eqn-ProblemStratonovich} in the It\^o form. We have, see e.g. \cite{BE-loops,TwN},
\begin{align*}
-\im B u(t) \circ\df W(t)&=-\im B u(t) \df W(t)+\frac12  \sumM -\im B'[u]\left(-\im B(u(t))f_m\right)f_m \df t\\
&=-\im Bu(t)\,{\rm d}W(t)- \frac12\sumM B(Bu(t)f_m)f_m\, {\rm d}t
\\
&= -\im B u(t) \df W(t)-\frac 12 \sumM B_m^2 u(t) \df t.
\end{align*}
Hence,  equation $\eqref{eqn-ProblemStratonovich}$ will be understood in the following It\^o form 	
\begin{equation}\label{Problem}
\begin{aligned}
\df u(t)&= -\left[\im A u(t)+\im  F(u(t))+\beta u(t)-b(u(t)) \right] \df t
\\
&\,\qquad\qquad\qquad-\im B u(t) \, \df W(t) -\im G(u(t)) \,\df \newW(t)
,\hspace{0,3 cm} t>0,
\end{aligned}
\end{equation}
where
\begin{align}
\label{Strat_cor}
b(u) := -\frac{1}{2} \sumM B_m^2 u,\qquad u\in{H},
\end{align}
is the Stratonovich correction term.
Notice that from assumptions \eqref{noiseBoundsH} and  \eqref{noiseBoundsEnergy}
we infer that
\begin{align}
\label{eqn-b operator}
b \in \mathscr{L}(H) \cap \mathscr{L}(\EA)  \cap \mathscr{L}(\LalphaPlusEins),
\end{align}
i.e.  $b$ is a linear bounded operator in $H$ as well as in $\EA$ and $\LalphaPlusEins$.

We recall that the deterministic unforced nonlinear Schr\"odinger equation, i.e. equation \eqref{Problem} with $\beta=0$, $G=0$ and $B=0$, as a consequence of its Hamiltonian structure, has two invariant quantities: the mass $\|u\|^2_H$ and the energy $\mathcal{E}(u)$, which is defined as
\begin{align}
	\label{eqn-energy_def}
	\energy (u):= \frac{1}{2} \Vert A^{\frac{1}{2}} u \Vert_{H}^2+\Fhat(u),\qquad u\in \EA.
	\end{align}
Note that $\Fhat(u)$, hence $\energy (u)$ too, is  well defined for $u\in \EA$
thanks to  the  embedding $\EA \hookrightarrow L^{\alpha+1}$, see \eqref{eqn-Gelfand triple},  and the form of $\hat F$, see \eqref{eqn_boundantiderivative}.
\\
In general, in the presence of stochastic forcing, these quantities are no longer conserved. But, when the noise is of purely Stratonovich form, in the form we consider here,  and if there is no dissipation, i.e. $\beta=0$ and $G=0$,  one has conservation of mass but not conservation of energy. This is the case studied in \cite{Brz+H+W-2019} and it is a particular case of our framework.
 In the more general setting we consider in this work, neither the mass or energy are preserved. Nevertheless, as quite classical in the stochastic case, we can still use these functionals to prove the existence of solutions with values in the energy space $\EA$.

The following definition is given under
 Assumptions \ref{assumption-A-space}, \ref{assumption-F_def} and \ref{assumption-stochastic}.
\begin{definition}\label{def-martingale solution}
	Let  $\mu$  be a Borel probability measure on the energy space $\EA$
	with
	\begin{equation}\label{integrability-initial-data}
	\int \left( \Vert x \Vert_{H}^2+ \energy (x)\right) \,{\rm d}\mu(x)<\infty.
	\end{equation}
A \emph{martingale solution} of the equation $\eqref{Problem}$ with the initial data  $\mu$ is a system
\begin{equation}\label{eqn-mart sol}
\bigl(\tilde{\Omega},\tilde{\F},\tilde{\Prob},\tilde{W},\tilde{\newW},\tilde{\Filtration},u\bigr)
\end{equation}
 consisting of
	\begin{itemize}
\item a filtered probability space $\bigl(\tilde{\Omega},\tilde{\F},\tilde{\Prob},\tilde{\Filtration}\bigr)$,   satisfying the usual conditions, i.e.
 the filtration $\tilde{\Filtration}=\bigl(\tilde{\F}_t\bigr)_{t\in [0,\infty)}$ is  right-continuous and such that all $\tilde{\Prob}$-null, i.e. $\tilde{\Prob}$-negligible,  sets of $\mathscr{F}$ are elements of $\tilde{\F}_0$;
\item  two independent  $\Yb$-cylindrical, resp. $\Yg$-cylindrical,   Wiener  processes $\tilde{W}$ and  $\tilde{\newW}$ on  $\bigl(\tilde{\Omega},\tilde{\F},\tilde{\Prob}\bigr);$
\item
 a $H$-valued  continuous and  $\tilde{\Filtration}$-adapted process  $u$ with $\tilde{\Prob}$-almost all paths in
	 $C_w([0,\infty),\EA)$,  fulfilling the initial condition  \[\tilde{\mathbb{P}}(u(0))=\mu\] and such that,
	 for every  $T>0$,
 \begin{equation}\label{eqn-ass-u}
 \tilde \E \left(\Vert u\Vert_{L^\infty(0,T;H)}^2+ \Vert  \energy (u)\Vert_{L^\infty(0,T)}\right)<\infty
 \end{equation}
and for every $t \geq 0$
  the equality
	\begin{align}\label{eqn-ItoFormSolution}
	u(t)=  u(0)- \int_0^t \left[\im A u(s)+\im F(u(s))+\beta  u(s) -b(u(s))\right] \df s     \nonumber
	\\
	- \im \int_0^t B u(s) \df \tilde{W}(s) -\im \int_0^t G(u(s)) \,\df \tilde{\newW}(s),
	\end{align}
	holds $\tilde{\mathbb{P}}$-almost surely in $\EAdual$.
\end{itemize}
\end{definition}
\begin{remark}\label{rem-def of solution}
Let us notice that the four deterministic Bochner $\EA^\ast$-valued integrals that appear in \eqref{eqn-ItoFormSolution} make sense.
First, we notice that, since by \eqref{eqn-V to H} the embedding $V \hookrightarrow H$ is compact, the (weak)-continuity in $V$ implies the (strong)-continuity in $H$,  hence $u $ has $\tilde{\mathbb{P}}$-a.s. paths in $C([0,T],H)$.
In addition, since by  \eqref{eqn-Gelfand triple}, the embedding $H \hookrightarrow V^\ast$ is  continuous, the (strong)-continuity in $H$  implies the (strong)-continuity in $V^\ast$. Therefore $u $ has $\tilde{\mathbb{P}}$-a.s. paths in $C([0,T],V^*)$.
\\
Second, we have that $u \in L^{\infty}(0,T;V)$ and $u$ is Bochner integrable in $V$, in view of the following argument:
\begin{itemize}
\item[(i)] if  $u \in C_w([0,T],\EA)$, since $[0,T]$ is compact, the range of $u$ is a compact subset of  $\EA_w$. Since by the Banach-Steinhaus Theorem, compact sets in weak topology are strongly, i.e. norm bounded, we infer that  the range of $u$ is a (norm) bounded  subset of  $\EA$;
\item[(ii)] if  $u \in C_w([0,T],\EA)$, then $u$ is $\mathscr{B}([0,T])/\mathscr{B}(\EA_w))$-measurable. On the other hand, see argument after (2.8) in \cite{Brz+Ferr_2019},  $\mathscr{B}(\EA_w)=\mathscr{B}(\EA)$, see also   \cite[Theorem 7.19]{Zizler_2003} and \cite{Edgar_1979} for more general claims. Therefore,
    function $u:[0,T] \to \EA$ is $\mathscr{B}([0,T])/\mathscr{B}(\EA)$-measurable.
\item[(iii)]    It follows from items (i) and (ii),  that
if  $u \in C([0,T],\EA_w)$ then $u:[0,T] \to \EA$ is  measurable and bounded. Hence in particular,
(the equivalence class of) $u$ belongs to
$L^\infty(0,T;\EA)$ and $u$ is Bochner integrable in $V$, see \cite[section II.2, p.50]{Diestel+Uhl_1977}.
    \end{itemize}
A special attention should be paid to the  $\EA^\ast$-valued Bochner integral  $\int_0^T \im F(u(s)) \,{\rm d}s $.
This integral exists because if $ x\in C([0,T];\EA_w)$ then by the compactness of the embedding $\EA\hookrightarrow\LalphaPlusEins$ we infer that
$ x\in C([0,T];\LalphaPlusEins)$ and therefore, by part (i) of Lemma \ref{lem-nonlinearAssumptions},
$F \circ x \in C([0,T];\LalphaPlusEinsDual)$. Thus, the integral   $\int_0^T  F(x(s)) \,{\rm d}s $  exists in $\LalphaPlusEinsDual$ in the Riemann sense. Thus, by
\eqref{eqn-Gelfand triple}, we infer  that the integral   $\int_0^T  F(x(s)) \,{\rm d}s $  exists in $\EAdual$ in the Riemann, and not only Bochner,  sense.

For what concerns the stochastic integrals that appear in \eqref{eqn-ItoFormSolution}, let us emphasize that  in order for them to be well defined
it would be enough to require, instead of \eqref{eqn-ass-u},  that
\begin{equation*}
\tilde{\mathbb{E}} \left[\int_0^T \| u(t) \|_H^2 \, {\rm d}t\right] <\infty, \mbox{ for every } T>0.
\end{equation*}
\end{remark}

\begin{definition}\label{def-strong solution}
Assume that 	
\begin{equation}\label{eqn-mart sol}
\bigl({\Omega},{\F},{\Prob},{W},{\newW},{\Filtration}\bigr)
\end{equation}
is a system  consisting of
	\begin{itemize}
\item a filtered probability space $\bigl({\Omega},{\F},{\Prob},{\Filtration}\bigr)$,   satisfying the usual conditions, i.e.
 the filtration ${\Filtration}=\bigl({\F}_t\bigr)_{t\in [0,\infty)}$ is  right-continuous and such that all ${\Prob}$-null, i.e. ${\Prob}$-negligible,  sets of $\mathscr{F}$ are elements of ${\F}_0$;
\item  two independent  $\Yb$-cylindrical, resp. $\Yg$-cylindrical,   Wiener  processes ${W}$ and  ${\newW}$ on  $\bigl({\Omega},{\F},{\Prob}\bigr).$
\end{itemize}
Let  $u_0:\Omega \to \EA$  be a ${\F}_0/ \mathscr{B}(\EA)$  measurable function such that  with	\begin{equation}\label{integrability-initial-data-u_0}	{\mathbb{E}} \left( \Vert u_0 \Vert_{H}^2+ \energy (u_0)\right) <\infty.	
\end{equation}
A \emph{strong  solution} of the equation $\eqref{Problem}$ with the initial data $u_0$ is
$H$-valued  continuous and  ${\Filtration}$-adapted process  $u$ with ${\Prob}$-almost all paths in
	 $C_w([0,\infty),\EA)$ and such that,
	 for every  $T>0$, condition \eqref{eqn-ass-u} holds and,
and for every $t \geq 0$
  the equality
  	\begin{align*}
	u(t)=  u_0- \int_0^t \left[\im A u(s)+\im F(u(s))+\beta  u(s) -b(u(s))\right] \df s     \nonumber
	\\
	- \im \int_0^t B u(s) \df {W}(s) -\im \int_0^t G(u(s)) \,\df {\newW}(s),
	\end{align*}
	holds ${\mathbb{P}}$-almost surely in $\EAdual$.
\end{definition}

{The following is a summary of the first of our main results. For a more detailed statements see Theorems \ref{thm-existence} and \ref{thm-pathwise uniqueness}.
\begin{theorem}\label{mainTh}
Fix $\newr\in [1,\infty)$.
Under the Assumptions \ref{assumption-A-space}, \ref{assumption-F_def} and \ref{assumption-stochastic}, for  every  Borel probability measure  $\mu$  on $\EA$ whose $r(\alpha+1)$-th moment   is  finite, i.e.
\begin{equation}\label{eqn-mu-r(alpha+1)}
\int\Vert x\Vert_V^{r(\alpha +1)}\,{\rm d}\mu(x)<\infty
\end{equation}
the following assertion hold true.
 \begin{itemize}
  \item [i)]
 There exists a martingale solution  to equation \eqref{eqn-ProblemStratonovich} with the initial data  $\mu$ such that, for every $T>0$,
 \begin{equation}\label{stima_func}
\tilde \E\left[ \sup_{t \in [0,T]} \| u(t)\|_H^{2r}+\sup_{t \in [0,T]} \mathcal{E}(u(t))^{r}  \right]<\infty.
\end{equation}
  In particular
 	\begin{equation}\label{add-regularity}
		\tilde \E\left[ \sup_{t \in [0,T]} \| u(t)\|_V^{2r}  \right]<\infty.
	\end{equation}
 \item [ii)] In addition, if $r\ge 2$, then the solutions fulfilling \eqref{add-regularity}  are pathwise unique.
\end{itemize}
\end{theorem}
The above result implies the existence of a unique strong solution, see Theorem \ref{thm-strong existence}.
\begin{remark}
\label{bound_rem}
Let us notice that inequality \eqref{add-regularity} is an immediate consequence of the previous estimate \eqref{stima_func}
since 
\[\|u\|^{2r}_V \lesssim_r \|u\|^{2r}_H +\energy(u)^r.\]
\end{remark}
\begin{remark}
\label{rem_new_id}
\begin{itemize}
\item[ i)] When $r=1$,  the  assumption on the initial data in Theorem \ref{mainTh}  is
\begin{equation}\label{eqn-mu-r+1}
\int\Vert x\Vert_V^{\alpha+1}\,{\rm d}\mu(x)<\infty
\end{equation}
 which
 is stronger than the assumption \eqref{integrability-initial-data} appearing in the
 Definition \ref{def-martingale solution}. In fact
 \[
\int \left( \Vert x \Vert_{H}^2+ \energy (x)\right) \,{\rm d}\mu(x)
\lesssim
\int \left( \Vert x \Vert_{V}^2+  \Vert x \Vert_{V}^{\alpha+1}\right) \, {\rm d}\mu(x)
\lesssim 1+ \int   \Vert x \Vert_{V}^{\alpha+1}\, {\rm d}\mu(x).
\]
 We need assumption \eqref{eqn-mu-r+1}
   because of our construction of  a martingale solution by means of the
  finite-dimensional Galerkin approximation. Indeed
 in the finite-dimensional Galerkin approximation we will need   uniform estimates of the power-type  nonlinearity, which  hold in the Hilbert spaces  $H$ and $\EA$
but not in the Lebesgue space $\LalphaPlusEins$, see \eqref{PnHeinsKontraktiv}.
\\
 Similarly, to gain the additional $L^r(\widetilde{\Omega})$-integrability, with $r>1$, of the solution process  the condition on the initial datum has to be strengthened requiring its $r(\alpha+1)$-th moment to be bounded, see \eqref{stima-finale-dimo-7-5}.
\item [ii)] On the other hand, if the initial datum $u_0 \in V$ is deterministic, Theorem \ref{mainTh} ensures the existence of a unique strong solution fulfilling \eqref{stima_func} and \eqref{add-regularity} for any finite $r \ge 1$.
\end{itemize}
\end{remark}

To study the existence of an invariant measure for equation \eqref{Problem} we work with deterministic initial data $u_0 \in V$. We are thus in the situation described in Remark \ref{rem_new_id}(ii) and we deal with the unique strong solution to \eqref{Problem} fulfilling  \eqref{stima_func} and \eqref{add-regularity}.
Given the (non-random) initial datum $u_0 \in V$, we denote by  $\{u(t;u_0)\}_{t\ge 0}$
this unique strong solution. We define the family of operators  $\{P_t\}_{t\ge 0}$ by
\begin{equation}\label{def-P_t}
P_t\phi(u_0)=\mathbb E[\phi(u(t);u_0)]
\end{equation}
and prove that this is a Markov semigroup, see Section \ref{sec-weakCont}, which is sequential weak Feller in $V$.
Then we say that a Borel probability measure $\pi$ on  $V$ is an invariant measure for equation \eqref{eqn-ProblemStratonovich} iff
\[
\int_V P_t \phi\ d\pi=\int_V \phi\  d\pi
\]
for  all $t\ge 0$ and all  bounded functions $\phi:V\to\mathbb R$ which are
 sequentially continuous with respect to the
weak topology on $V$.

The following is our second main result.
\begin{theorem}\label{mainTh2}
Under the Assumptions \ref{assumption-A-space}, \ref{assumption-F_def} and \ref{assumption-stochastic}
there exists at least one invariant measure for equation \eqref{eqn-ProblemStratonovich} provided
      \begin{equation}\label{condizione-beta}
\beta>\max\left(\tilde{C}_1^2+\tilde{C}_2^2+\|B\|^2_{\mathscr{L}(V,\gamma(Y_1,V))},
                     \frac{\alpha +1}{2}\|B\|^2_{\mathscr{L}(L^{\alpha+1},\gamma(Y_1,L^{\alpha+1}))} +\alpha\tilde{C}^2_3  \right).
      \end{equation}
\end{theorem}

\begin{remark}
\label{general}
$\quad$
\begin{itemize}
\item [i)] Statement (i)  of Theorem \ref{mainTh} holds in a more general setting, see Remark \ref{rem-FR1}. Since our focus is on the existence of an invariant measure, we restrict our analysis to the less general setting.
This is done also in \cite{Brz+H+W-2019} to get the uniqueness of solutions.
\item[ii)] For a discussion concerning the regularity assumptions we impose on the domain $\mathscr{O}$ in the case of Assumption \ref{assumption-A-space}(ii)-(iii), see Remark \ref{FR2}.
\item [iii)]
The condition \eqref{condizione-beta} on $\beta$ does not depend on the coefficients
$C_1$, $C_2$ and $C_3$ characterizing the second noise term.
When $B=0$ and $\tilde{C}_1=\tilde{C}_2=\tilde{C}_3=0$, that is when we consider a multiplicative noise
$G(u) \,\df \tilde{\newW}$ with bounded covariance, the condition \eqref{condizione-beta} reduces to $\beta>0$.
Therefore, in the particular case of additive noise  (i.e., $B=0$ and $G$ independent of $u$), we recover the same condition
$\beta>0$ as  in  the previous papers \cite{Kim_2006} and \cite{Ekren_2017}.
\item [iv)] In condition \eqref{condizione-beta} there is the constant $\tilde{C}_1$ (which somehow measures the
intensity of the noise $G$ in the $H$-norm)  but not the analogue for the noise $B$,  that is the term
$\|B\|^2_{\mathscr{L}(H,\gamma(Y_1,H))}$. This asymmetry is due to the fact that the Stratonovich noise $B$
 (in the absence of damping) preserves the $H$-norm, whereas the noise $G$ does not.
 In other words, the correction term $b$ of the Stratonovich noise cancels with the term
 $\|B\|^2_{\mathscr{L}(H,\gamma(Y_1,H))}$ acting as a sort of damping term which perfectly balance the intensity
 of the noise $B$ in the $H$-norm.

A similar reasoning can be done to explain also why we have the different constants $\alpha$ and $ \frac{\alpha +1}{2}$
multiplying the terms $\tilde{C}^2_3$ and $\|B\|^2_{\mathscr{L}(L^{\alpha+1},\gamma(Y_1,L^{\alpha+1}))}$ respectively.
We have $\frac{\alpha+1}{2}<\alpha$: the correction term $b$ provides part of the dissipation in the $L^{\alpha +1}$ norm.
\end{itemize}
\end{remark}

For the purely multiplicative noise,  the  uniqueness  of an invariant measure will be given in Corollary \ref{cordelta0}.

\section{Compactness and tightness criteria}	\label{CompactnessSection}

This Section is devoted to recalling the compactness results, which will be used  in Section \ref{sec-exis} to obtain a
martingale solution as  limit of
the Faedo-Galerkin approximation and in Section \ref{sec-weakCont} to prove the continuous  dependence of the solutions on the initial data.

Let $\alpha>1$ and $A$ be chosen according to Assumption $\ref{assumption-A-space}$.
We consider  the  Banach spaces ${C([0,T];\EAdual)}$ and ${L^{\alpha+1}(0,T;L^{\alpha+1})}$,
and the  locally convex space $C_w([0,T];\EA)$.
So we define the space
\begin{align}\label{eqn-Z_T}
Z_T&= {C([0,T];\EAdual)}\cap{L^{\alpha+1}(0,T;L^{\alpha+1})}\cap C_w([0,T];\EA)
\end{align}
with the topology $\mathscr{T}_T$ given by the supremum of the corresponding topologies. By $\mathscr{B}(Z_T)$ we denote the associated  Borel
$\sigma$-field, i.e. the $\sigma$-field generated by the open sets in the locally convex topology $\mathscr{T}_T$ of $Z_T$.
We also define a corresponding space of functions defined on the whole half-line $[0,\infty)$ with three
locally convex spaces $C([0,\infty);\EAdual)$, $L^{\alpha+1}_{\mathrm{loc}}([0,\infty);\LalphaPlusEins)$ and
$C_w([0,\infty);\EA)$:
\begin{align}
\label{eqn-Z_infty}
Z_\infty&= {C([0,\infty);\EAdual)} \cap {L^{\alpha+1}_{\mathrm{loc}}([0,\infty);\LalphaPlusEins)}
\cap C_w([0,\infty);\EA) =\bigcap_{T \in \mathbb{N}} Z_T.
\end{align}
with the topology $\mathscr T_{\infty}$ defined analogously.

In the next Proposition, we give a criterion for compactness in  $Z_\infty $.

\begin{proposition}\label{CompactnessDeterministic}
Let $(r_N)_{N=1}^\infty$ be a sequence of positive numbers	and  $K$ be a subset of $Z_\infty$  such that for every $N\in \N$, 
	\begin{enumerate}
		\item[a)] $
		\displaystyle\sup_{u\in K} \Vert u\Vert_{L^\infty(0,N;\EA)}\le r_N ;
		$
		\item[b)] $K$ is equicontinuous in ${C([0,N];\EAdual)},$ i.e.
		\begin{align*}
		\lim_{\delta \to 0} \sup_{u\in K} \sup_{s,t \in [0,N]:\vert t-s\vert\le \delta} \Vert u(t)-u(s)\Vert_{\EAdual}=0.
		\end{align*}
	\end{enumerate}
	Then, $K$ is relatively compact in $Z_\infty$.
\end{proposition}
\begin{proof}The proof is a minor modification of the proof of \cite[Proposition 4.2]{Brz+H+W-2019}. Here one uses Lemma \ref{convergenceStetigBall} which is the restatement of \cite[Lemma 4.1]{Brz+H+W-2019} on the time interval $[0, \infty)$.
\end{proof}

Now we want to obtain a criterion for tightness in $Z_\infty$.
Therefore, we introduce the Aldous condition, working in a probability space
$(\Omega, \mathcal{F}, \mathbb{P})$ with filtration $\mathbb{F}:=\{\mathcal{F}_t\}_{t \in [0,\infty)}$ satisfying the usual conditions.
\begin{definition}\label{DefinitionAldous}
	We say that a sequence  $(X_n)_{n\in\N}$ of continuous $\mathbb{F}$-adapted stochastic processes taking values in a Banach space $E$ 
satisfies  the Aldous condition $[A]$ if and only if 
for all $T>0$,  $\varepsilon>0$ and $\eta>0$ there is $\delta>0$ such that for every sequence $(\tau_n)_{n\in\N}$ of $\mathbb{F}$-valued stopping times with $\tau_n \le T$, one has
	\begin{align}
	\sup_{n\in\N} \sup_{0<\theta \le \delta} \Prob \left\{ \Vert X_n((\tau_n+\theta)\land T)-X_n(\tau_n)\Vert_E\ge \eta \right\}\le \varepsilon.
	\end{align}
\end{definition}

The following Lemma, which generalises  \cite[Lemma A.7]{Motyl_2013}, gives us a useful consequence of the Aldous condition $[A].$

\begin{lemma} \label{AldousLemma}			
	Let $(X_n)_{n\in\N}$ be a sequence of continuous $\mathbb{F}$-adapted stochastic processes in a Banach space $E,$ which satisfies the Aldous condition $[A].$ Then, for every $\varepsilon>0$ there exists a measurable subset $A_\varepsilon \subset C([0,\infty),E)$ such that
	\begin{align*}
	\inf_{n \in \mathbb{N}}\law{\Prob}{X_n}(A_\varepsilon)\ge 1-\varepsilon,\qquad
	\end{align*}
and, for every $N \in \N$, 
	\begin{align*}
	\lim_{\delta\to 0} \sup_{u\in A_\varepsilon} \sup_{s,t \in [0,N]: \vert t-s\vert\le \delta} \Vert u(t)-u(s)\Vert_E=0.
	\end{align*}
\end{lemma}
\begin{proof}
The proof of \cite[Lemma A.7]{Motyl_2013} can be easily adapted to the present situation.
\end{proof}

The deterministic compactness result in Proposition \ref{CompactnessDeterministic} and Lemma \ref{AldousLemma} can be used to get the following criterion for tightness in $Z_\infty$.

\begin{proposition}\label{TightnessCriterion}
	Let $(X_n)_{n\in\N}$ be a sequence of continuous adapted $\EAdual$-valued processes satisfying the Aldous condition $[A]$ in $\EAdual$ and 
	\begin{align*}
	\sup_{n\in\N} \E \left[\Vert X_n\Vert_{L^\infty(0,T;\EA)}^2\right] <\infty, \mbox{ for every $T>0$}.
	\end{align*}
	Then the sequence $\left(\law{\Prob}{X_n}\right)_{n\in\N}$ is tight in $Z_\infty,$ i.e. for every $\varepsilon>0$ there is a compact set $K_\varepsilon\subset Z_\infty$ with
	\begin{align*}
	\inf_{n \in \mathbb{N}}\law{\Prob}{X_n}(K_\varepsilon)\ge 1- \varepsilon.
	\end{align*}
\end{proposition}

\begin{proof}
	Let us choose and fix $\varepsilon>0.$  Let us set $c:=\sum_{N=1}^\infty\frac{1}{N^2}$ and let us define a sequence $(r_N)_{N=1}^\infty$ by 
\[
r_N:= \left(\frac{2c}{\varepsilon} \sup_{n\in\N} \E \left[ \Vert X_n\Vert_{L^\infty(0,N;\EA)}^2\right]\right)^{\frac{1}{2}}.
\]
Set
\begin{equation*}
B_N:=\{ \Vert X_n\Vert_{L^\infty(0,N;\EA)}\le N r_N\}.
\end{equation*}
Then, by the Chebyshev inequality  we obtain, for every $N \in \N$, 
	\begin{align}
\mathbb{P}(B_N^c)\le \frac{1}{N^2r_N^2}\E \left[\Vert X_n\Vert_{L^\infty(0,N;\EA)}^2\right]\le \frac{\varepsilon}{2cN^2}.
	\end{align}
Set
\begin{align}\label{eqn-B set}
B:= \left\{u\in Z_\infty: \Vert u\Vert_{L^\infty(0,N;\EA)}\le Nr_N, \; \mbox{ for every } N\in \N\right\}= \bigcap_{N \in \mathbb{N}}B_N.
\end{align}
We have 
\begin{equation*}
\mathbb{P}(B^c) \le \sum_{N \in \mathbb{N}}\mathbb{P}(B_N^c)\le \frac{\varepsilon}{2}.
\end{equation*}
By Lemma \ref{AldousLemma}, we  can use the Aldous condition $[A]$ to get a Borel subset $A$ of ${C([0,\infty);\EAdual)}$ such that
	\begin{align}
	\law{\Prob}{X_n}\left(A\right)&\ge 1-\frac{\varepsilon}{2},\quad n\in\N, 
\\
 \lim_{\delta\to 0} \sup_{u\in A} &\sup_{s,t \in [0,N]: \vert t-s\vert\le \delta} \Vert u(t)-u(s)\Vert_\EAdual=0 \mbox{ for every } N\in \N.
	\end{align}
	We define $K:= \overline{A\cap B}$. By Proposition \ref{CompactnessDeterministic} this set $K$ is compact in $Z_\infty$. Moreover 
	for all $n\in \N$ we have 
	\begin{align}
	\law{\Prob}{X_n}(K)\ge \law{\Prob}{X_n}\left(A\cap B\right)\ge \law{\Prob}{X_n}\left(A\right)-\law{\Prob}{X_n}\left( B^c\right)\ge 1-\frac{\varepsilon}{2}-\frac{\varepsilon}{2}=1-\varepsilon.
	\end{align}
This completes the proof.
\end{proof}

In metric spaces, one can apply Prokhorov Theorem (see \cite{Parthasarathy_1967}, Theorem II.6.7) and Skorohod Theorem (see \cite[Theorem 6.7]{Billingsley_1999}) to obtain convergence from tightness. Since the space $Z_\infty$ is a locally convex space, we use the  following generalization to nonmetric spaces.

\begin{proposition}[Skorohod-Jakubowski]\label{SkohorodJakubowski}
	Let $\mathcal{X}$ be a topological space such that there is a sequence of continuous functions $f_m: \mathcal{X}\to \C$ that separates points of $\mathcal{X}.$ Let $\mathcal{A}$ be the $\sigma$-algebra  generated by $\left(f_m\right)_m.$ Then, we have the following assertions:
	\begin{enumerate}
		\item[a)] Every compact set $K\subset \mathcal{X}$ is metrizable.
		\item[b)] Let $\left(\mu_n\right)_{n\in\N}$ be a tight sequence of probability measures on $\left(\mathcal{X}, \mathcal{A}\right).$
		Then, there are a subsequence $\left(\mu_{n_k}\right)_{k\in\N},$ random variables $X_k$ (for $k\in\N$) and   $X$  on a common probability space $(\tilde{\Omega},\tilde{\Filtration},\tildeProb)$ with $\law{\tildeProb}{X_k}=\mu_{n_k}$ for $k\in\N,$ and $X_k \to X$ $\tildeProb$-almost surely for $k\to \infty.$
	\end{enumerate}
\end{proposition}
We stated Proposition \ref{SkohorodJakubowski} in the form of \cite{Brz+Ondr_2011}; see also \cite{Jak98}, where it was first used to construct martingale solutions for stochastic geometric wave equations.
We apply 
this result to get the final result of this Section.
\begin{corollary}\label{corollaryEstimatesToASconvergence}
	Let $(X_n)_{n\in\N}$ be a sequence of adapted $\EAdual$-valued processes satisfying the Aldous condition $[A]$ in $\EAdual$ and
	\begin{align*}
	\sup_{n\in\N} \E \left[\Vert X_n\Vert_{L^\infty(0,T;\EA)}^2\right] <\infty, \mbox{ for every $T>0$}.
	\end{align*}
	Then, there are a subsequence $(X_{n_k})_{k\in\N}$ and random variables $\tilde{X}_k,$ $\tilde{X}$ for $k\in\N$ on another probability space $(\tilde{\Omega},\tilde{\Filtration},\tildeProb)$ with $\law{\tildeProb}{\tilde{X}_k}=\law{\Prob}{X_{n_k}}$ for $k\in\N,$  and $\tilde{X}_k \to \tilde{X}$ $\tildeProb$-almost surely in $Z_\infty$ for $k\to \infty.$
\end{corollary}
\begin{proof}
This proof is also a minor modification of the proof of \cite[Corollary 4.7]{Brz+H+W-2019}. 
\\
Let us recall that $Z_\infty$ is a locally convex space. Therefore, the assertion follows by an application of Propositions \ref{TightnessCriterion} and \ref{SkohorodJakubowski} if for each of the spaces in the definition of $Z_\infty$ we find a sequence $f_m: Z_\infty \to \R$ of continuous functions separating points which generates the Borel $\sigma$-field.  The separable Fr{\'e}chet  spaces  ${C([0,\infty);\EAdual)}$ and ${L^{\alpha+1}_{\loc}(0,\infty;L^{\alpha+1})}$ have this property. \\
	Let $\left\{h_m: m\in \N \right\}$ be a dense subset of $\EAdual.$ Then, we define the countable set \\$F:=\left\{f_{m,t}: m\in\N, t \in [0, \infty)\cap \mathbb{Q} \right\}$ of functionals on $C_w([0,\infty);\EA)$ by
	\begin{align*}
	f_{m,t}(u):= \duality{u(t)}{h_m}, \;\; u \in C_w([0,\infty);\EA), 
	\end{align*}
	for $m\in\N,$ $t\in [0,\infty)\cap \mathbb{Q}$ and $u\in C_w([0,\infty);\EA).$ \\
	The set $F$ separates points, since for $u,v \in C_w([0,\infty);\EA)$ with $f_{m,t}(u)=f_{m,t}(v)$ for all $m\in\ N$ and $t\in [0,\infty)\cap \mathbb{Q},$ we get $\duality{u}{h_m}=\duality{v}{h_m}$ on $[0,\infty)$ for all $m\in\N$ by continuous continuation and therefore $u=v$ on $[0,\infty).$ \\
	Furthermore, the density of $\{h_m: m\in\N\}$ and the definition of the locally convex topology yield that $\left(f_{m,t}\right)_{m\in\N, t\in [0,\infty)\cap \mathbb{Q}}$ generate the Borel $\sigma$-algebra on $C_w([0,\infty);\EA).$
\end{proof}

\section{Existence of a martingale solution}\label{sec-exis}

In this Section we prove the existence of at least one martingale solution, see Definition \ref{def-martingale solution}.
In this way we prove part i) of Theorem \ref{mainTh}; moreover we  provide an estimate over the time interval $[0,\infty)$.

\subsection{Statement of the existence result}  \label{subsect-statement}	

Keeping in mind the definition of the energy functional $\mathcal{E}$ given in \eqref{eqn-energy_def}, we state the main result of this Section.

\begin{theorem}\label{thm-existence}
Fix $\newr\in [1,\infty)$  and let $\mu$ be a Borel probability measure on $\EA$ whose
$r(\alpha+1)$-th moment   is  finite.
Under Assumptions \ref{assumption-A-space},  \ref{assumption-F_def}, \ref{assumption-stochastic},
there exists a martingale solution  $\left(\tilde{\Omega},\tilde{\F},\tilde{\Prob},\tilde{W},\tilde{\newW}, \tilde{\Filtration},u\right)$
 of equation \eqref{eqn-ProblemStratonovich} with the initial data $\mu$
  which  satisfies
	\begin{equation}\label{eqn-stimaEnergy-full}
		                 \tilde\E \Big[ \sup_{t \in [0,T]} \| u(t)\|_H^{2r}+\sup_{t \in [0,T]}  \energy(u(t))^\newr\Big] <\infty, \;\;\mbox{ for every }T>0.
		                  \end{equation}
 Hence, in particular,
	 \begin{align}\label{eqn-propertySolution}
	\tilde{\mathbb E} \sup_{t \in [0,T]} \|u(t)\|^{2r}_{\EA}<\infty,
\;\;\mbox{ for every }T>0.
	\end{align}
Moreover, if $\beta$ satisfies condition \eqref{condizione-beta}, then
\begin{equation}\label{Ebound}
\sup_{t\ge 0} \tilde{\mathbb E} \|u(t)\|^2_{\EA}<\infty.
\end{equation}
\end{theorem}

The proof of the existence part  is based on a technique already used in  \cite{BM} and \cite{Brz+H+W-2019}.
We present the basic steps: the estimates on the Galerkin approximation and its  convergence
to a process which is a martingale solution. To be more precise, in Proposition \ref{MassEstimateGalerkinSolution} we prove
that there exists a unique global solution of the approximated problem and we obtain the a priori estimates in the space $H$.
In Proposition \ref{EstimatesGalerkinSolution} we obtain the a priori estimates for the energy functional and prove the Aldous condition.
These results lead to Corollary \ref{cor-aprioriEA} which together with the Aldous condition implies the tightness.
In Section \ref{Section_Convergence} we prove the convergence of the Galerkin approximations to the martingale solution of our problem.
The estimate \eqref{Ebound} will be proved first for the Galerkin approximation
in Proposition \ref{EstimatesGalerkinSolution} and then it will hold for the limit too.

More general assumptions can be considered only to prove the existence of a martingale solutions;
see Remark \ref{rem-FR1} at the end of this Section.

\subsection{The Galerkin Approximation and a priori estimates}  \label{sectionGalerkin}	

In this Section we introduce the Galerkin approximation. We prove the well-posedness of the approximated equation
and the uniform estimates for the solutions, that are sufficient to apply Corollary \ref{corollaryEstimatesToASconvergence}.

Let us recall that the operator $S$ was introduced in Lemma \ref{S_properties}. By the functional calculus we define the operators
$P_n:H \rightarrow H$ by $P_n:=\pmb{1}_{(0,2^{n+1})}(S)$ for $n \in \mathbb{N}_0$. Since $S$ has the representation
\eqref{eqn_spectral_S}, we observe that  $P_n$ is the orthogonal projection from $H$ to
$H_n:=\operatorname{span}\left\{h_m: m\in\N, \lambda_m< 2^{n+1}\right\}$ and
\begin{align*}
P_n x = \sum_{\lambda_m< 2^{n+1}} \skpH{x}{h_m}h_m, \qquad x\in {H}.
\end{align*}
Note that we have $h_m\in \bigcap_{k\in\N}\D(S^k)$ for $m\in\N$. Since by Lemma \ref{eqn_spectral_S}(i)
$\D(S^k)\hookrightarrow \EA$
for some $k\in\N$, we infer that $H_n$ is a closed subspace of $\EA$ for $n\in\N$. In particular, $H_n$ is
a closed subspace of $\EAdual$.
The fact that the operators $S$ and $A$ commute implies that $P_n$ and $A^\frac{1}{2}$ commute.
Thus we get
\begin{equation*}
\norm{P_n x}_{\EA}^2=\norm{P_n x}_{H}^2+\norm{\sqrtA P_n x}_{H}^2
=\norm{P_n x}_{H}^2+\norm{ P_n \sqrtA x}_{H}^2
\le \norm{x}_{\EA}^2, \quad x \in V.
\end{equation*}
Moreover,
\begin{equation*}
\norm{P_n x}_{H}
\le \norm{x}_{H}, \quad x \in H\qquad \qquad \text{and} \qquad
\norm{P_n x}_{\EAdual}\le \norm{x}_\EAdual, \quad   x \in \EAdual,
\end{equation*}
 and, recalling \eqref{eqn-Gelfand triple}, \eqref{eqn_boundantiderivative} and \eqref{eqn-energy_def},
 \begin{multline}\label{PnHeinsKontraktiv}
\energy(P_n x)
=\tfrac 12  \norm{\sqrtA P_n x}_{H}^2 +\tfrac 1{\alpha+1}\norm{P_n x}_{\LalphaPlusEins}^{\alpha+1}
\lesssim_\alpha \norm{ P_n \sqrtA x}_{H}^2+\norm{P_n x}_{\EA}^{\alpha+1}
\\\le \norm{x}_{\EA}^2+\norm{x}_{\EA}^{\alpha+1}\lesssim 1+\norm{x}_{\EA}^{\alpha+1}, \quad x \in V.
\end{multline}
 We also have
\[\begin{split}
&\lim_{n\to\infty} \|P_nx-x\|_V=0
\\
&\lim_{n\to\infty} \|P_nx-x\|_{L^{\alpha+1}}\le C \lim_{n\to\infty} \|P_nx-x\|_V= 0.
\end{split}\]
By density, we can extend $P_n$ to an operator $P_n: \EAdual\to H_n$ with $\norm{P_n}_{\EAdual\to \EAdual}\le 1$ and
\begin{align}\label{PnInEAdual}
\duality{v}{P_n v}\in \R, \qquad \duality{v}{P_n w}=\skpH{P_n v}{w}, \qquad v\in \EAdual, \quad w\in \EA.
\end{align}

Unfortunately, the operators $P_n$, $n\in\N$,
are not uniformly bounded from $\LalphaPlusEins$ to $\LalphaPlusEins$. This property is crucial in the proof of the a priori estimates of the stochastic terms. To overcome this deficit, in the following Proposition we construct the sequence $\left(S_n\right)_{n\in\N}$ which enjoys the needed properties.

\begin{proposition}\label{PaleyLittlewoodLemma}
	There exists a sequence $\left(S_n\right)_{n\in\N}$ of self-adjoint operators $S_n: H \to H_n$ for $n\in\N$ such that
	$S_n \psi \to \psi$ in $\EA$ for $n\to \infty$ and $\psi \in \EA$  and the uniform norm estimates
	\begin{align}\label{SnUniformlyBounded}
	\sup_{n\in\N}\norm{S_n}_{{\mathscr{L}(H)}}\le 1, \quad \sup_{n\in\N} \norm{S_n}_{\mathscr{L}(\EA)}\le 1,
	\quad \sup_{n\in\N} \norm{S_n}_{\mathscr{L}(L^{\alpha+1})}<\infty
	\end{align}
	hold.
\end{proposition}

\begin{remark}
Somehow, $S_n$ represents a smoothed version of the indicator function $p_n:=\pmb{1}_{(0,2^{n+1})}$ used to
define the operator $P_n$. This allows to use the spectral multiplier theorems to prove the uniform
$L^{\alpha +1}$- boundedness of the sequence $(S_n)_{n \in \mathbb{N}_0}$.
In \cite{Brz+H+W-2019} the same result is proved by means of the abstract Littlewood-Paley theory rather than
spectral multipliers theorems. Our  proof follows the lines of  \cite[Proposition 10]{BHM-2020} and \cite[Lemma 3.14 and Remark 4.15]{Hornung-F_PhD}
with the difference that here we use the classical estimate from Ouhabaz \cite{Ouh_2009} instead of the results from Kunstmann and Uhl \cite{KU_2015}.
\end{remark}

\begin{proof}[Proof of Proposition \ref{PaleyLittlewoodLemma}]
We take a function $\rho \in C^{\infty}_c(0, \infty)$ with supp$(\rho)\in \left[ \frac 12 , 2\right]$ and $\sum_{m \in \mathbb{Z}}\rho(2^{-m}t)=1$ for all $t >0$. For a fixed $n \in \mathbb{N}_0$ we introduce the function
\begin{equation*}
s_n: (0, \infty) \rightarrow \mathbb{C}, \qquad \quad s_n(\lambda):=\sum_{m=-\infty}^{n}\rho\left(2^{-m}\lambda\right)
\end{equation*}
and we see that
\begin{equation*}
s_n(\lambda)=
\begin{cases}
\ 1 &\lambda \in (0, 2^n)
\\
\rho(2^{-n}\lambda) & \lambda \in [2^n, 2^{n+1})
\\
0 & \lambda \ge 2^{n+1}.
\end{cases}
\end{equation*}

We define $S_n:=s_n(S)$ via the functional calculus for self-adjoint operators.
In particular, by Lemma \ref{S_properties}(ii), we have the representation
\begin{equation}
\label{Sn_rep}
S_n x= \sum_{\lambda_m<2^n}(x,h_m)_Hh_m+ \sum_{\lambda_m\in[2^n,2^{n+1})}
\rho(2^{-n}\lambda_m)(x,h_m)_Hh_m, \qquad x \in H,
\end{equation}
from which immediately follows that the range of $S_n$ is contained in $H_n$.
 Since $s_n$ is real-valued and bounded by $1$, the operator $S_n$ is self-adjoint with $\|S_n\|_{\mathscr{L}(H)}\le1$. Moreover, since by Lemma \ref{S_properties}(i), $S_n$ and $A$ commute, we obtain $\|S_n\|_{\mathscr{L}(V)}\le1$ and $S_n \psi \rightarrow \psi$ for all $\psi \in V$, by the convergence property of the functional calculus.

 Finally, the uniform estimate in $L^{\alpha +1}$ is a consequence of the spectral multiplier Theorem \cite[Theorem 7.23]{Ouh_2009} and the Marcinkiewicz interpolation Theorem. It is sufficient to show, see \cite[equation (7.69)]{Ouh_2009}, that $s_n$ satisfies
 \begin{equation*}
 \sup_{\lambda>0}|\lambda ^ks_n^{(k)}(\lambda)|< \infty, \qquad \quad \ k=0,1,2. \end{equation*}
 We have
 \begin{equation*}
  \sup_{\lambda>0}|\lambda^ks_n^{(k)}(\lambda)|
  = \sup_{\lambda \in [2^n, 2^{n+1})}|\lambda^ks_n^{(k)}(\lambda)|
  =\sup_{\lambda \in [2^n, 2^{n+1})}|\lambda^k\frac{{\rm d}^k}{{\rm d}\lambda^k}\rho(2^{-n}\lambda)|
  \le
  2^k\sup_{\lambda >0}|\rho^{(k)}(\lambda)|< \infty,
     \end{equation*}
 for all $k \in \mathbb{N}$. This completes the proof of Proposition \ref{PaleyLittlewoodLemma}.
\end{proof}

Let $u_0 $ be an $\F_0$-measurable $\EA$-valued  random variable  such that    $ \law{\mathbb{P}}{u(0)}=\mu$ on $\mathscr{B}(V)$.
Then, since
$2r$-th moment of the Borel probability measure  $\mu$ on $\EA$ is  finite, we infer that the $2r$-th moment of $u_0$ is finite.

Using the operators $P_n$ and $S_n,$ $n\in\N,$ we approximate our original problem $\eqref{eqn-ProblemStratonovich}$ by the stochastic differential equation in $H_n$ given by	
\begin{equation*}
\left\{
\begin{aligned}
\df u_n(t)&=- \left[\im A u_n(t)+\im P_n F(u_n(t)) +\beta u_n(t) \right] \df t-\im  S_n B(S_n u_n(t)) \circ \df W(t)
\\
&\qquad\qquad -\im  S_n G(S_n u_n(t)) \,\df \newW(t)
\\
u_n(0)&=P_n u_0.
\end{aligned}\right.
\end{equation*}
With the Stratonovich correction term
\begin{align*}
b_n := -\frac{1}{2} \sumM \left(S_n B_m S_n\right)^2,
\end{align*}
the approximated problem can be written in the following  It\^o form 	
\begin{equation}\label{galerkinEquation}
\left\{
\begin{aligned}
\df u_n(t)&= -\left[\im A u_n(t)+\im P_n F(u_n(t))+\beta u_n(t)- b_n (u_n(t)) \right] \df t
\\
&\qquad\qquad -\im  S_n B (S_n u_n(t)) \df W(t)  -\im  S_n G(S_n u_n(t)) \,\df \newW(t) \\
u_n(0)&=P_n u_0.
\end{aligned}\right.
\end{equation}
By the well known theory of finite dimensional stochastic differential equations with locally Lipschitz coefficients,
we get a local well-posedness result for $\eqref{galerkinEquation}.$

\begin{proposition}\label{localSolutionGalerkin}
Suppose  Assumptions $\ref{assumption-A-space}$, $\ref{assumption-F_def}$, $\ref{assumption-stochastic}$ hold. Assume that $\newr \in [1,\infty)$. If
 $u_0 $  is an $\F_0$-measurable $H$-valued  random variable  with finite  $2\newr$-th  moment,
then for each $n\in\N$ there is a unique local solution $u_n$ of $\eqref{galerkinEquation}$ with continuous paths in $H_n$
and maximal existence time $\tau_n,$ which is a blow-up time in the sense that
\[\displaystyle\limsup_{ t \nearrow \tau_n(\omega)}\norm{u_n(t,\omega)}_{H_n}=\infty\]
for almost all $\omega\in\Omega$ with  $\tau_n(\omega)<\infty.$
\end{proposition}

Now we introduce a technical lemma used in many instances for a priori estimates, see \cite[Lemma 5.6]{Brz+H+W-2019}.
\begin{lemma}\label{LemmaYOmegaNachLzweiZeit}
	Let $r\in [1,\infty),$ $\varepsilon>0,$ $T>0$ and $f\in L^r(\Omega,L^\infty(0,T)).$ Then,
	\begin{align*}
	\norm{f}_{L^r(\Omega,L^2(0,t))}\le \varepsilon \norm{f}_{L^r(\Omega,L^\infty(0,t))} +\frac{1}{4\varepsilon}\int_0^t \norm{f}_{L^r(\Omega,L^\infty(0,s))} \df s,\qquad t\in[0,T].
	\end{align*}
\end{lemma}

We prove a priori estimates in the space $H$ so to get global existence. We work on any  bounded time interval.

\begin{proposition}\label{MassEstimateGalerkinSolution}
Suppose Assumptions $\ref{assumption-A-space}$, $\ref{assumption-F_def}$  and  $\ref{assumption-stochastic}$
hold. Assume that $\newr \in [1,\infty)$. If
 $u_0 $ is an $\F_0$-measurable $H$-valued  random variable  with
  $2\newr$-th moment finite,   then for each $n\in\N$  there exists a unique global solution $u_n$ of $\eqref{galerkinEquation}$
with continuous paths in $H_n$. 	Moreover, for every finite $T>0$,
		\begin{equation}
	\label{esHnew}
	\sup_{n \in \mathbb{N}}\mathbb{E} \left[\sup_{t \in [0,T]}\|u_n(t)\|_H^{2\newr} \right] <\infty.
\end{equation}
\end{proposition}
\begin{proof}
\emph{Step 1:} We fix $n\in\N$ and take the unique maximal  solution $(u_n,\tau_n)$ from Proposition $\ref{localSolutionGalerkin}$.
We prove that the solution is global appealing to the Khasmiskii's test for non explosion, see \cite{Kah}.
Let us introduce a sequence $\{\tau_{n,k}\}_{k\in\N}$ of stopping times defined by
	\begin{align*}
	\tau_{n,k}:=\inf \left\{t\ge0: \norm{u_n(t)}_{H_n}\ge k\right\},\qquad k\in\N.
	\end{align*}
In order to prove that $\tau_n=+ \infty$ $\mathbb{P}$-a.s. it is sufficient to find a Liapunov function $\mathscr{V}:H \rightarrow \mathbb{R}$ satisfying
\begin{align}
 &\mathscr{V} \ge 0  \qquad \text{on} \quad H, \label{K1}\\
 &a_k:= \inf \bigl\{ \mathscr{V}(v): \|v\|_{H}\ge k \bigr\} \rightarrow \infty, \qquad \text{as} \quad k \rightarrow \infty, \label{K2}\\
 &\mathbb{E}[\mathscr{V}(u_n(0))]< \infty  \label{K3}
\end{align}
such that
\begin{equation}
\label{K4}
\mathbb{E}[\mathscr{V}(u_n(t\wedge \tau_{n,k}))] \le \mathbb{E}[\mathscr{V}(u_n(0))]+ C\int_0^t (1+\mathbb{E}[\mathscr{V}(u_n(s\wedge \tau_{n,k}))])\, {\rm d}s
\end{equation}
for a constant $C< \infty$ and all $t\ge 0$ and $k \in \mathbb{N}$.
The idea is the following: once such a function $\mathscr{V}$  is found, by the Gronwall's lemma we infer
\begin{equation*}
\mathbb{E}[\mathscr{V}(u_n(t \wedge \tau_{n,k})] \le e^{Ct} (1+ \mathbb{E}[\mathscr{V}(u_n(0))]), \quad t \ge 0,
\end{equation*}
which implies
\begin{equation*}
\mathbb{P}(\tau_{n,k}<t) \le \frac{1}{a_k}\mathbb{E}\left[\mathbbm{1}_{\{\tau_{n,k}<t\}}\mathscr{V}(u_n(t \wedge \tau_{n,k})) \right]
\le \frac{1}{a_k}e^{Ct}\left(1+\mathbb{E}[\mathscr{V}(u_n(0))]\right).
\end{equation*}
Passing to the limit, we get
\begin{equation*}
\lim_{k \rightarrow \infty}\mathbb{P}(\tau_{n,k}<t)\le e^{Ct}(1+\mathbb{E}[\mathscr{V}(u(0))])\lim_{k \rightarrow \infty}\frac{1}{a_k}=0,
\end{equation*}
for every fixed $t\ge 0$. Therefore $\mathbb{P}(\tau_n<t)=\displaystyle\lim_{k \rightarrow \infty}\mathbb{P}(\tau_{n,k}<t)=0$ for every fixed $t \ge 0$, which means $\mathbb{P}(\tau_n=+\infty)=1$.
\newline
Set
\begin{equation}
\label{V_Kas}
\mathscr{V}(v):=\|v\|^2_H, \qquad v \in H_n.
\end{equation}
The function $\mathscr{V} \in C^2(H)$, $\mathscr{V}$ is uniformly continuous on bounded sets and satisfies \eqref{K1} and \eqref{K2}.
Moreover, $\mathbb{E}[\mathscr{V}(P_nu_0)]< \infty$ is equivalent to $\mathbb{E}\left[\|P_nu_0\|^2_H\right]< \infty$.
\newline
In order to derive inequality \eqref{K4} we use the It\^o formula.
The function $\mathscr{V}: H_n \to \R$  defined  in \eqref{V_Kas} is  twice continuously Fr\'{e}chet-differentiable with
	\begin{align*}
	\mathscr{V}'[v]h_1&= 2 \Real \skpH{v}{ h_1}, \qquad
	\mathscr{V}^{\prime\prime}[v] \left[h_1,h_2\right]= 2 \Real \skpH{ h_1}{h_2},
	\end{align*}
	for $v, h_1, h_2\in H_n.$
We look for estimates for $\norm{u_n}_{H}^2$. The It\^o formula for this real process involves only real quantities,
expressed by means of the real part of the complex scalar product in $H$.

For a fixed $v\in H_n$ and $m\in\N$, we have some basic relationships:
	\begin{align*}
	\Real \skpH{ v}{ -\im A v}&=\Real \left[\im \Vert \sqrtA v\Vert_{H}^2\right]=0,		\\
	\Real \skpH{ v}{ -\im P_n F\left( v\right)}&=\Real \duality{\im v}{   F\left( v\right)}=0, \\
	2\Real \skpH{ v}{b_n(v)}&= -\sumM\Real  \skpH{ v}{\left(S_n B_m S_n\right)^2 v}=-\sumM  \Vert{ S_n B_m S_n v}\Vert_{H}^2,
	\\
\Real \skpH{v}{ -\im S_n B(S_n v)f_m}&
= \Real \left[\im\skpH{ v}{S_n   B_m S_n v}\right]=0,	
	\end{align*}
where we used $\eqref{PnInEAdual}$ and \eqref{eqn_nonlinearityComplex} for the second term and the fact that the operator $S_n B_mS_n $ is self-adjoint for the third and four terms.
		
Therefore, by the It\^o formula we get, for $t\ge 0$,
\begin{align*}
	\norm{u_n(t\wedge \tau_{n,k})}_{H}^2=&\norm{P_n u_0}_{H}^2-2 \beta\int_0^{t\wedge \tau_{n,k}}  \|u_n(s)\|_H^2
 \df s
 +\int_0^{t\wedge \tau_{n,k}}
	\Vert  S_n G\bigl(S_n u_n(s)\bigr) \Vert_{\gamma(Y_2,H)}^2\df  s
\\
&
+2 \int_0^{t\wedge \tau_{n,k}} \Real \skpH{u_n(s)}{ -\im S_n G( S_n u_n(s))  \df \textbf{W}(s)}.
	\end{align*}
To estimate the term in the RHS of the above equation we introduce the  stochastic process
\begin{equation*}
\alpha_k(t):=
\pmb{1}_{(t<\tau_{n,k})}.
\end{equation*}
and we notice that
\begin{equation}
\label{trick1}
\int_0^{t \wedge \tau_{n,k}}\|u_n(s)\|^2_H\, {\rm d}s
= \int_0^t \|\alpha_k(s)u_n(s)\|^2_H\,{\rm d}s
\le \int_0^t \|u_n(s \wedge \tau_{n,k})\|^2_H\, {\rm d}s.
\end{equation}	
Therefore, by Proposition \ref{PaleyLittlewoodLemma} and  \eqref{eqn-G crescita} we get
\begin{equation*}
\Vert  S_n G\bigl(S_n u_n(s)\bigr) \Vert_{\gamma(Y_2,H)}^2 \le 2C_1^2+ 2\tilde{C}_1^2\|u_n(s)\|^2_H.
\end{equation*}
So we obtain
\begin{align*}
	\norm{u_n(t\wedge \tau_{n,k})}_{H}^2\le
	&\norm{P_n u_0}_{H}^2+2\int_0^t (C_1^2+(-\beta+\tilde{C}_1^2)
	 \|u_n(s\wedge \tau_{n,k})\|_H^2)\df  s
\\
&
+2 \int_0^{t\wedge \tau_{n,k}} \Real \skpH{u_n(s)}{ -\im S_n G( S_n u_n(s))  \df \textbf{W}(s)}, \qquad t\ge 0.
\end{align*}
Taking now the expected value on both sides, we obtain
\[
	\mathbb{E}\left[\norm{u_n(t\wedge \tau_{n,k})}_{H}^2\right]
	\le
	\mathbb{E}\left[\norm{P_n u_0}_{H}^2\right]+2\int_0^t \Big(C_1^2+(-\beta+\tilde{C}_1^2)
	 \mathbb{E}\left[\|u_n(s\wedge \tau_{n,k})\|_H^2\right]\Big)\df  s .
\]
This proves \eqref{K4} and so we conclude the proof of the global existence of the solution.
\newline
\emph{Step 2:}
        We now prove estimate \eqref{esHnew}.
Let us fix $T>0$. We start from equality
\begin{equation}\label{stima-2q-Galerkin}
\begin{split}
	\norm{u_n(t)}_{H}^2=&\norm{P_n u_0}_{H}^2-2 \beta\int_0^{t}  \|u_n(s)\|_H^2
 \df s
 +\int_0^t
	\Vert  S_n G\bigl(S_n u_n(s)\bigr) \Vert_{\gamma(Y_2,H)}^2\df  s
\\
&
+2 \int_0^t \Real \skpH{u_n(s)}{ -\im S_n G( S_n u_n(s))  \df \textbf{W}(s)}, \mbox{ for $t\in [0,T]$},
	\end{split}\end{equation}
 and we apply the $L^{\newr}(\Omega, L^{\infty}(0,T))$-norm to this identity.
From Proposition \ref{PaleyLittlewoodLemma} and  \eqref{eqn-G crescita}  we immediately get
\begin{equation*}
\left \Vert \int_0^\cdot
	\Vert  S_n G\bigl(S_n u_n(s)\bigr) \Vert_{\gamma(Y_2,H)}^2\df  s \right \Vert_{L^{\newr}(\Omega, L^{\infty}(0,T))}	\le 2C_1^2T + 2\tilde C_1^2 \left\Vert\int_0^T
	\|u_n(s)\|_H^2\df  s \right\Vert_{L^{\newr}(\Omega)}.\
	\end{equation*}
The Minkowski inequality yields
\begin{align*}
	\left \Vert \int_0^T\|u_n(s)\|^2_H\, {\rm d}s\right \Vert_{L^{\newr}(\Omega)}
&\le\left \Vert \int_0^{T}\sup_{r \in[0,s]}\|u_n(r)\|^2_H\, {\rm d}s\right \Vert_{L^{\newr}(\Omega)}
\\
&\le \int_0^T\left \Vert \sup_{r \in[0,s]}\|u_n(r)\|^2_H\right \Vert_{L^{\newr}(\Omega)} \, {\rm d}s
= \int_0^T\left\Vert \|u_n\|^2_H \right\Vert_{L^{\newr}(\Omega,L^{\infty}(0,s))}\, {\rm d}s.
\end{align*}
By means of the Burkholder-Davis-Gundy and the Young inequalities, Lemma \ref{LemmaYOmegaNachLzweiZeit} and \eqref{eqn-G crescita} we obtain,  for some constant $C$ (depending on ${\newr}$)
\[\begin{split}
	  \Big\Vert\int_0^{\cdot} \Real & \skpH{u_n(s)}{ -\im S_n G(S_n u_n(s)) \df \textbf{W}(s)} \Big\Vert_{L^{\newr}(\Omega, L^{\infty}(0,T))}	
	  \\
           &\le C	 \left\Vert \left(\int_0^T\sum_{m=1}^{\infty} \skpH{u_n(s)}{ -\im S_n G(S_n u_n(s))e_m}^2 ds\right)^{1/2}\right\Vert_{L^{\newr}(\Omega)}	
	 \\
           &  \le C \left\Vert\|G(S_nu_n)\|_{\gamma(Y_2,H)}\|u_n\|_H \right\Vert_{L^{\newr}(\Omega, L^2(0,T))}	
           \\
            &   \le C \left\Vert   (C_1+\tilde C_1 \|u_n\|_H) \|u_n\|_H \right\Vert_{L^{\newr}(\Omega, L^2(0,T))}	
	 \\
           &\lesssim  \left\Vert   1+ \|u_n\|_H^2    \right\Vert_{L^{\newr}(\Omega, L^2(0,T))}	
	 \\
	  &\le1+T + \left\Vert\|u_n\|^2_H\right\Vert_{L^{\newr}(\Omega, L^2(0,T))}
	  \\
	  &
	 \le 1+T + \varepsilon \left\Vert\|u_n\|^2_H\right\Vert_{L^{\newr}(\Omega, L^{\infty}(0,T))}
	        + \frac{1}{4 \varepsilon} \int_0^T \left\Vert\|u_n\|^2_H \right\Vert_{L^{\newr}(\Omega, L^{\infty}(0,s))}\,{\rm d}s,
	 \end{split}\]	
for any $\varepsilon>0$.
	
Collecting the above estimates, we obtain for any $n$
\begin{align*}
\left\Vert \|u_n\|_H^{2}\right\Vert_{L^{\newr}(\Omega, L^{\infty}(0,T))}	
\le \left\Vert\|u_0\|^2_H\right\Vert_{L^r(\Omega)}+
 C+CT+
+ \varepsilon C \left\Vert\|u_n\|^2_H\right\Vert_{L^{\newr}(\Omega, L^{\infty}(0,T))}
\\
+\left(2|\beta|+ \frac{C}{4\varepsilon}\right)
\int_0^T\left\Vert\|u_n\|^2_H \right\Vert_{L^{\newr}(\Omega, L^{\infty}(0,s))}\,{\rm d}s.
\end{align*}
If we choose $\varepsilon>0$ small enough, we can apply the Gronwall's lemma and we get that there exists a positive constant $C$, independent of $n$ (but depending on $T$ and other parameters)  such that
\begin{equation*}
\left\Vert \|u_n\|_H^2\right\Vert_{L^{\newr}(\Omega, L^{\infty}(0,T))}\le C,\qedhere
\end{equation*}
for any $n \in \mathbb{N}$.
	\end{proof}

The  next goal  is to find uniform energy estimates for the global solutions of the equation \eqref{galerkinEquation}.
Recall that the nonlinearity $F$ has a real antiderivative denoted by $\Fhat$
and the energy functional $\energy$ is defined in formula \eqref{eqn-energy_def}.

The next Proposition is the key step to show that we can apply Corollary \ref{corollaryEstimatesToASconvergence} to the sequence of solutions $(u_n)_{n\in\N}$ of the equation \eqref{galerkinEquation}.	

\begin{proposition}\label{EstimatesGalerkinSolution}
Fix $\newr\in [1,\infty)$ and let
   $u_0 $  be an $\F_0$-measurable $\EA$-valued  random variable whose $r(\alpha+1)$-th moment is finite.
	 Under Assumptions $\ref{assumption-A-space}$, $\ref{assumption-F_def}$
	 and  $\ref{assumption-stochastic}$ the following assertions hold.
	\begin{enumerate}
		\item[a)]   For every finite $T>0$,
				\begin{equation}
				\label{eqn-stimaEnergy}
		                 \sup_{n\in\N}\E \Big[ \sup_{t \in [0,T]} \|u_n(t)\|^{2r}_H+ \sup_{t \in [0,T]} \energy(u_n(t))^{\newr} \Big] <\infty.
		                  \end{equation}
		\item[b)] The sequence $(u_n)_{n\in\N}$ satisfies the Aldous condition $[A]$ in $\EAdual.$
		\item[c)] If in addition
		                 $\beta$ satisfies condition \eqref{condizione-beta}, then
                                 \begin{equation}\label{Galerkin-Ebound}
                                   \sup_{n \in \mathbb{N}} \sup_{t\ge 0} \mathbb E \|u_n(t)\|^2_{\EA}<\infty.
                                  \end{equation}
	\end{enumerate}
\end{proposition}
\begin{proof}[\emph{ad a):}]
Let us fix $T>0$. Thanks to  \eqref{esHnew} it is enough to prove the estimate for the energy.
By Lemma $\ref{lem-nonlinearAssumptions}$, the restriction of the energy $\energy: H_n \to \R$ is twice
continuously Fr\'{e}chet-differentiable with
	\begin{align*}
	\energy'[v]h_1=&\Real \duality{Av+F(v)}{ h_1}; \\
	\energy^{\prime\prime}[v] \left[h_1,h_2\right]=& \Real \skpH{\sqrtA h_1}{\sqrtA h_2}+\Real \duality{F^{\prime}[v]h_2}{h_1}
	\end{align*}
	for $v, h_1, h_2\in H_n.$
We look for estimates on $\energy(u_n)$.
		
	Notice that
\begin{align*}
	\Real \left[\duality{A v}{ -\im P_n F(v)}+\duality{F(v)}{ -\im A v}\right]
	&=\Real \left[-\duality{A v}{ \im F(v)}+\overline{\duality{  A v}{\im F(v)}}\right]=0,
	\end{align*}
	\begin{align*}
	\Real \skpH{A v}{ -\im A v}=\Real \left[\im \norm{A v}_{H}^2\right]=0,
	\end{align*}
	for all $v\in H_n$,	
and we can use $\eqref{PnInEAdual}$ for
	\begin{align*}
	\Real \duality{F(v)}{ -\im P_n F(v)}=\Real \left[\im \duality{F(v)}{  P_n F(v)}\right]=0.
	\end{align*}
Therefore, the It\^o formula leads to the identity
\begin{equation}\label{ItoEnergyWithoutExponent}
   \begin{split}
	\energy\left(u_n(t)\right) &=	\energy\left(P_n u_0\right)	
	+\int_0^t \Real \duality{A u_n(s)+F(u_n(s))}{ b_n(u_n(s))-\beta u_n(s)} \df s	
	\\
	&+\int_0^t \Real \duality{A u_n(s)+F(u_n(s))}{ -\im \left(S_n B S_n u_n(s)\right)\df W(s)}
	\\
	&+\int_0^t \Real \duality{A u_n(s)+F(u_n(s))}{ -\im S_n G \left(S_n u_n(s)\right)\df \newW(s)}
	 \\
	&+\tfrac{1}{2} \int_0^t  \Vert \sqrtA S_n B S_n u_n(s)\Vert_{\gamma(Y_1,H)}^2\df s
	+\tfrac{1}{2} \int_0^t  \Vert \sqrtA S_n G S_n u_n(s)\Vert^2_{\gamma{(Y_2,H)}}\df s
	\\
	&+\tfrac{1}{2}\int_0^t \sumM \Real \duality{F^{\prime}[u_n(s)] \left(S_n B S_n u_n(s)\right)f_m}{ (S_n B S_n u_n(s))f_m} \df s
	\\
	&+\tfrac{1}{2}\int_0^t \sumM \Real \duality{F^{\prime}[u_n(s)] \left(S_n G( S_n u_n(s))e_m\right)}{ S_n G(S_n u_n(s))e_m} \df s,
\end{split}\end{equation}	
almost surely for all $t\in [0,T].$

Let us introduce the short notation
\begin{equation}\label{def-z-somma}
Z(u):= \|u\|^2_H + 2\energy(u)= \|u\|^2_\EA+2\Fhat(u), \quad u \in \EA.
\end{equation}
We will estimate the various terms that appear in the RHS of
\eqref{ItoEnergyWithoutExponent} in the $L^r(\Omega, L^{\infty}(0,T))$-norm.
To slightly simplify the proof we will neglect here the dissipation term by assuming $\beta \ge 0$.
The case $\beta<0$ can be treated as in the proof of Proposition \ref{EstimatesGalerkinSolution} (c), see Appendix \ref{App_B}.
We set $I_{1,n}(u)(t)=2\int_0^t  \Real \duality{A u(s)+F(u(s))}{ b_n(u(s))} \df s$. We have
\[
\left \Vert I_{1,n}(u_n) \right \Vert_{L^{\newr}(\Omega,L^{\infty}(0,T))}
\le 	\left\Vert \int_0^T\left| \sumM  \Real \duality{A u_n(s)+F(u_n(s))}{(S_nB_mS_n)^2u_n(s)} \right|\df s\right \Vert_{L^{\newr}(\Omega)}.
\]
Using the bounds $\eqref{SnUniformlyBounded}$ and Assumptions \ref{assumption-stochastic}(ii), see also Remark \ref{rem-HS}, by means of the Young inequality we get
\[\begin{split}
\left| \sumM  \Real \duality{A u_n}{(S_nB_mS_n)^2u_n} \right|
&\le  \|A^{\frac 12}u_n\|_H  \sumM \|A^{\frac 12}(S_nB_mS_n)^2u_n)\|_H
\\
&\le \|A^{\frac 12}u_n\|_H  \|B\|^2_{\mathscr{L}(\EA, \gamma(Y_1,\EA))}   \|u_n\|_{\EA}
\\
&\le \|B\|^2_{\mathscr{L}(\EA, \gamma(Y_1,\EA))}  \left(
\|A^{\frac 12}u_n\|^2_H+\|u_n\|^2_H\right)
\\&
\lesssim\|B\|^2_{\mathscr{L}(\EA, \gamma(Y_1,\EA))}Z(u_n).
\end{split}\]
On the other hand, $\eqref{SnUniformlyBounded},$  Assumptions \ref{assumption-stochastic}(ii),
$\eqref{eqn_nonlinearityEstimate}$ and $\eqref{eqn_boundantiderivative}$ lead to
\[\begin{split}
 \left| \sumM \Real \duality{F(u_n)}{(S_nB_mS_n)^2u_n}\right|
 &\le \|F(u_n)\|_{L^{\frac{\alpha+1}\alpha}} \sumM \|(S_nB_mS_n)^2u_n\|_{L^{\alpha+1}}
 \\
 &\lesssim
 \hat F(u_n)\|B\|^2_{\mathscr{L}(L^{\alpha+1},\gamma(Y_1,L^{\alpha+1}))}
 \\
 &\lesssim
 \|B\|^2_{\mathscr{L}(L^{\alpha+1},\gamma(Y_1,L^{\alpha+1}))}Z(u_n).
\end{split}\]
Therefore we obtain
\begin{multline}
\label{e1}
\left \Vert I_{1,n}(u_n)  \right \Vert_{L^{\newr}(\Omega,L^{\infty}(0,T))}
\\
\lesssim
\left(\|B\|^2_{\mathscr{L}(\EA, \gamma(Y_1,\EA))} + \|B\|^2_{\mathscr{L}(L^{\alpha+1},\gamma(Y_1,L^{\alpha+1}))}\right)\int_0^T\|Z(u_n)\|_{L^{\newr}(\Omega,L^{\infty}(0,s))}\, {\rm d}s.
\end{multline}

We set $I_{2,n}(u)(t)=2\int_0^t \Real \duality{A u(s)+F(u(s))}{ -\im S_n B \left(S_n u(s)\right)\df W(s)}$.
We employ the Burkholder-Davis-Gundy inequality to get
\[
\left \Vert   I_{2,n}(u_n)     \right \Vert_{L^{\newr}(\Omega,L^{\infty}(0,T))}
\le 2
\left \Vert \left(\sumM|\Real \duality {A u_n+F(u_n)}{ -\im (S_n B S_n u_n)f_m}|^2 \right)^{\frac 12}\right \Vert_{L^{\newr}(\Omega,L^2(0,T))} .
\]
Using the bounds $\eqref{SnUniformlyBounded},$  Assumptions \ref{assumption-stochastic}(ii) and the Young inequality, we estimate
\[\begin{split}
\left(  \sumM  |\Real \duality{A u_n}{ -\im \left( S_n B S_n u_n\right) f_m}|^2 \right)^{\frac 12}
&\le \|A^{\frac 12}u_n\|_H \|B(S_nu_n)\|_{\gamma(Y_1, \EA)}
\\&
\le \|A^{\frac 12}u_n\|_H  \|B\|_{\mathcal L (\EA,\gamma(Y_1, \EA))}   \|u_n\|_{\EA}
\\&
\le \|B\|_{\mathscr{L}(\EA,\gamma(Y_1, \EA))}  \left( \|A^{\frac 12}u_n\|^2_H+\|u_n\|^2_H\right)
\\&
\lesssim \|B\|_{\mathscr{L}(\EA,\gamma(Y_1, \EA))}  Z(u_n).
\end{split}\]
On the other hand $\eqref{SnUniformlyBounded},$  Assumptions \ref{assumption-stochastic}(ii), $\eqref{eqn_nonlinearityEstimate}$ and $\eqref{eqn_boundantiderivative}$ lead to
\[\begin{split}
 \left(\sumM|\Real \duality {F(u_n)}{ -\im \left(S_n B S_n u_n\right))f_m}|^2 \right)^{\frac 12}
 &\lesssim \|B\|_{\mathscr{L}(L^{\alpha+1}, \gamma(Y_1, L^{\alpha+1}))} \hat F(u_n)
 \\
&\lesssim \|B\|_{\mathscr{L}(L^{\alpha+1}, \gamma(Y_1, L^{\alpha+1}))} Z(u_n).
 \end{split}\]
Therefore, the latter three estimates with Lemma $\ref{LemmaYOmegaNachLzweiZeit}$
 yield
\begin{align}
\label{e2}
 \Vert I_{2,n}(u) \Vert_{L^{\newr}(\Omega,L^{\infty}(0,T))}
&
\lesssim
\left(\|B\|_{\mathscr{L}(\EA, \gamma(Y_1,\EA))} + \|B\|_{\mathscr{L}(L^{\alpha+1},\gamma(Y_1,L^{\alpha+1}))}\right)\|Z(u_n)\|_{L^{\newr}(\Omega,L^2(0,T))}
\notag \\&
\lesssim \varepsilon \|Z(u_n)\|_{L^{\newr}(\Omega, L^{\infty}(0,T))}+\frac{1}{4\varepsilon}\int_0^T \|Z(u_n)\|_{L^{\newr}(\Omega, L^{\infty}(0,s))}\, {\rm d}s,
\end{align}
for any $\varepsilon>0$.

We set $I_{3,n}(u)(t)=2\int_0^t   \Real \duality{A u(s)+F(u(s))}{ -\im S_n G \left(S_n u(s)\right)\df \newW(s)} $.
Also for this stochastic integral we employ  the Burkholder-Davis-Gundy's inequality to get
\[
\left \Vert  I_{3,n}(u_n)\right \Vert_{L^{\newr}(\Omega,L^{\infty}(0,T))}
\le 2 \left \Vert \left(\sumM|\Real \duality {A u_n+F(u_n)}{ -\im S_n G \left(S_n u_n\right))e_m}|^2 \right)^{\frac 12}\right \Vert_{L^{\newr}(\Omega,L^2(0,T))}.
\]
Using the bounds $\eqref{SnUniformlyBounded}$ and Assumptions \ref{assumption-stochastic}(iii), exploiting the Young inequality we obtain
\[\begin{split}
 \left(\sumM|\Real \duality{A u_n}{ -\im \left( S_n G (S_n u_n)\right)e_m} |^2 \right)^{\frac 12}
 &\lesssim \|A^{\frac12}u_n\|_H \|G(S_nu_n)\|_{\gamma(Y_2, \EA)}
 \\&
 \lesssim \|A^{\frac12 }u_n\|_H \left(C_2+ \tilde C_2\|u_n\|_{\EA}\right)
\\&
 \lesssim 1+ Z(u_n).
 \end{split}\]
Moreover, $\eqref{SnUniformlyBounded},$  Assumptions \ref{assumption-stochastic}(iii), $\eqref{eqn_nonlinearityEstimate}$,
$\eqref{eqn_boundantiderivative}$, \eqref{ineq-F} and the Young inequality lead to
 \[\begin{split}
 \left(\sumM|\Real \duality {F(u_n)}{ -\im S_n G \left(S_n u_n\right)e_m}|^2 \right)^{\frac 12}
 &\le \|F(u_n)\|_{L^{\frac{\alpha+1}{\alpha}}}\|G(S_nu_n)\|_{\gamma(Y_2,L^{\alpha+1})}
 \\
 &\lesssim \|u_n\|^{\alpha}_{L^{\alpha+1}}\left(C_3+ \tilde C_3\|u_n\|_{L^{\alpha+1}} \right)
 \\
 &\lesssim [\hat F(u_n)]^{\frac{\alpha}{\alpha+1}}+ \hat F(u_n)
  \\
 &\lesssim 1+ \hat F(u_n)
 \\
 &\lesssim  1 +Z(u_n).
  \end{split}\]
 Therefore, collecting the previous estimates, by means of Lemma $\ref{LemmaYOmegaNachLzweiZeit}$ we obtain
\begin{align}
 \label{e3}
 \left \Vert   I_{3,n}(u_n) \right \Vert_{L^{\newr}(\Omega,L^{\infty}(0,T))}
 &
\lesssim
\|1+Z(u_n)\|_{L^{\newr}(\Omega, L^2(0,T))}
\\&
\lesssim 1+ \varepsilon \|Z(u_n)\|_{L^{\newr}(\Omega,L^{\infty}(0,T))}+ \frac{1}{4 \varepsilon}\int_0^T \|Z(u_n)\|_{L^{\newr}(\Omega,L^{\infty}(0,s))} \, {\rm d}s,
\end{align}
for any $\varepsilon>0$.

We set $I_{4,n}(u)(t)= \int_0^t  \Vert \sqrtA S_n B S_n u(s)\Vert_{\gamma(Y_1,H)}^2\df s $.
 Exploiting Remark \ref{rem-HS} and bounds $\eqref{SnUniformlyBounded}$, we easily obtain
\begin{equation}
\label{e4}
  \left\Vert I_{4,n}(u_n) \right\Vert_{L^{\newr}(\Omega, L^{\infty}(0,T))}
 \lesssim \|B\|^2_{\mathscr{L}(\EA,\gamma(Y_1, \EA))}\int_0^T \|Z(u_n)\|_{L^{\newr}(\Omega,L^{\infty}(0,s))}\,{\rm d}s.
\end{equation}

We set $I_{5,n}(u)(t)=\int_0^t \Vert \sqrtA S_n G (S_n u(s))\Vert_{\gamma(Y_2,H)}^2\df s $.
Similarly, estimate \eqref{crescitaGEA} and bounds \eqref{SnUniformlyBounded} yield
\[
\Vert \sqrtA S_n G (S_n u_n)\Vert_{\gamma(Y_2,H)}^2\le (C_2 +\tilde C_2 \|u_n\|_V)^2
 \lesssim 1+ Z(u_n) .
\]
Hence
\begin{equation}
\label{e5}
 \left\Vert  I_{5,n}(u_n)   \right\Vert_{L^{\newr}(\Omega, L^{\infty}(0,T))}
 \lesssim 1+ \int_0^T \|Z(u_n)\|_{L^{\newr}(\Omega,L^{\infty}(0,s))}\,{\rm d}s.
\end{equation}

We set $I_{6,n}(u)(t)=\int_0^t \sumM \Real \duality{F^{\prime}[u(s)] \left(S_n B (S_n u(s))f_m\right)}{ S_n B (S_n u(s))f_m} \df s$.
From \eqref{eqn_deriveNonlinearBound}, Remark \ref{rem-HS}, \eqref{SnUniformlyBounded}  and \eqref{eqn_boundantiderivative} we get
\[\begin{split}
\left| \sumM \Real \duality{F^{\prime}[u_n] \left(S_n B S_n u_n\right)f_m}{ (S_n B S_n u_n) f_m} \right|
& \lesssim
  \|B\|^2_{\mathscr{L}(L^{\alpha+1},\gamma(Y_1,L^{\alpha+1}))} \|u_n\|^{\alpha+1}_{L^{\alpha+1}}
\\&
 \lesssim  \|B\|^2_{\mathscr{L}(L^{\alpha+1},\gamma(Y_1,L^{\alpha+1}))} \hat F(u_n).
\end{split}\]
Hence,
\begin{equation}\label{e6}
\left\Vert   I_{6,n}(u_n)    \right\Vert_{L^{\newr}(\Omega, L^{\infty}(0,T))}
 \lesssim \|B\|^2_{\mathscr{L}(L^{\alpha+1},\gamma(Y_1,L^{\alpha+1}))}
 \int_0^T \|Z(u_n)\|_{L^{\newr}(\Omega,L^{\infty}(0,s))}\,{\rm d}s.
\end{equation}

We set $I_{7,n}(u)(t)=\int_0^t \sumM \Real \duality{F^{\prime}[u(s)] \left(S_n G( S_n u(s))e_m\right)}{ S_n G(S_n u(s))e_m} \df s $.
By means of   \eqref{ineq-Fp}, \eqref{crescitaGL}, \eqref{SnUniformlyBounded} and the Young inequality
 we get
\[\begin{split}
\left|  \sumM \Real \duality{F^{\prime}[u_n] \left(S_n G( S_n u_n)e_m\right)}{ S_n G(S_n u_n)e_m}    \right|
&\le
\|F'[u_n]\|_{L^{\alpha+1} \rightarrow L^{\frac{\alpha+1}{\alpha}}}
\sumM \|S_n G( S_n u_n)e_m\|^2 _{L^{\alpha+1}}
\\
&\lesssim
\hat F(u_n)^{\frac{\alpha-1}{\alpha+1}} \|G( S_n u_n)\|^2 _{\gamma(Y_2,L^{\alpha+1})}
\\
&\le
\hat F(u_n)^{\frac{\alpha-1}{\alpha+1}}  (C_3+\tilde C_3 \|u_n\|_{L^{\alpha+1}})^2
\\
&\lesssim \hat F(u_n)^{\frac{\alpha-1}{\alpha+1}}+\hat F(u_n)
\\
&\lesssim 1+\hat F(u_n) \le 1+ Z(u_n).
\end{split}\]
Therefore,
\begin{equation}
\label{e7}
\left\Vert I_{7,n}(u_n)  \right\Vert_{L^{\newr}(\Omega, L^{\infty}(0,T))}
\lesssim
1+\int_0^T \|Z(u_n)\|_{L^{\newr}(\Omega,L^{\infty}(0,s))}\,{\rm d}s.
\end{equation}

We now go back to \eqref{ItoEnergyWithoutExponent}; using \eqref{e1}-\eqref{e7} we finally obtain
\begin{equation}\label{stima-finale-dimo-7-5}
\begin{split}
\|\energy(u_n)\|_{L^{\newr}(\Omega,L^{\infty}(0,T))}
&\le
\left\Vert\energy(P_nu_0)\right\Vert_{L^r(\Omega)}+C
+C  \int_0^T \|Z(u_n)\|_{L^{\newr}(\Omega,L^{\infty}(0,s))}\, {\rm d}s
+2 \varepsilon \|Z(u_n)\|_{L^{\newr}(\Omega,L^{\infty}(0,T))}
\\
&\le C\left(1+\Vert \|u_0\|_V^{\alpha+1}\Vert_{L^r(\Omega)}\right)
  +C
+
C  \int_0^T 2\|\energy(u_n)\|_{L^{\newr}(\Omega,L^{\infty}(0,s))}\, {\rm d}s
\\
&\quad+4 \varepsilon \|\energy(u_n)\|_{L^{\newr}(\Omega,L^{\infty}(0,T))}
+(CT+2 \varepsilon)   \left\Vert\|u_n\|^2_H\right\Vert_{L^{\newr}(\Omega,L^{\infty}(0,T))}
\end{split}\end{equation}
for some positive constant $C$, independent of $n$. We have estimated the initial data thanks to \eqref{PnHeinsKontraktiv}.
If we choose $\varepsilon$ sufficiently small and bear in mind  the a priori estimate \eqref{esHnew} of Proposition \ref{MassEstimateGalerkinSolution}, we get
\[
\|\energy(u_n)\|_{L^{\newr}(\Omega,L^{\infty}(0,T))} \lesssim 1+  \int_0^T \|\energy(u_n)\|_{L^{\newr}(\Omega,L^{\infty}(0,s))}\, {\rm d}s.
\]
By the Gronwall Lemma we deduce the assertion of Proposition \ref{EstimatesGalerkinSolution}, part a).
\end{proof}
\begin{proof}[\emph{ad b):}]
We choose and fix $T>0$. Let us now prove the Aldous condition.
The proof of this part is an extension of the proof of part (b) in \cite[Proposition 5.7]{Brz+H+W-2019}.
We will provide just the main steps.
We have
	\begin{align*}
	u_n(t)- P_n u_0=&  -\im \int_0^t A u_n(s) \df s-\im \int_0^t P_n F(u_n(s)) \df s+\int_0^t b_n(u_n(s)) \df s
	-\beta \int_0^t u_n(s)\, {\rm d}s	\\
	\hspace{1cm}&- \im \int_0^t S_n B (S_n u_n(s)) \df W(s)
	-\im \int_0^t S_nG(S_n u_n(s))\, {\rm d}\textbf{W}(s)
	\\
	=&:J_1(t)+J_2(t)+J_3(t)+J_4(t)+J_5(t)+J_6(t),
	\end{align*}
	in $H_n$ almost surely for all $t\in [0,T]$ and therefore
	\begin{align*}
	\Vert u_n((\tau_n+\theta)\land T)-u_n(\tau_n)\Vert_\EAdual\le\sum_{k=1}^{6} \Vert J_k((\tau_n+\theta)\land T)-J_k(\tau_n)\Vert_{\EAdual}
	\end{align*}
	for each sequence $\left(\tau_n\right)_{n\in\N}$ of stopping times and $\theta>0.$
	Hence, we get
	\begin{multline}\label{AldousStartingEstimate}
	\Prob \left\{\Vert u_n((\tau_n+\theta)\land T)-u_n(\tau_n)\Vert_\EAdual\ge \eta \right\}\le \sum_{k=1}^6\Prob \left\{\Vert J_k((\tau_n+\theta)\land T)-J_k(\tau_n)\Vert_{\EAdual}\ge \frac{\eta}{6}\right\},
	\end{multline}
	for a fixed $\eta>0$. We aim to apply the Chebyshev inequality and estimate the expected value of each term in the sum. 	
	Proceeding as in the proof of part (b) in \cite[Proposition 5.7]{Brz+H+W-2019} one obtains the following estimates:
	\begin{align*}
	\E\Vert J_1((\tau_n+\theta)\land T)-J_1(\tau_n)\Vert_{\EAdual}&\le  \theta K_1,
	\end{align*}
	\begin{align*}
	\E\Vert J_2((\tau_n+\theta)\land T)&-J_2(\tau_n)\Vert_{\EAdual}
	\le \theta K_2,
	\end{align*}
	\begin{align*}
	\E\Vert J_3((\tau_n+\theta)\land T)-J_3(\tau_n)\Vert_{\EAdual}
	\le \theta K_3,
	\end{align*}

\begin{align*}
	\E\Vert J_5((\tau_n+\theta)\land T)-J_5(\tau_n)\Vert_{\EAdual}^2&
	\le \theta K_5.
	\end{align*}
We use part a) to estimate
\begin{align*}
	\E\Vert J_4((\tau_n+\theta)\land T)&-J_4(\tau_n)\Vert_{\EAdual}
	\le \beta\mathbb{E} \left\Vert\int_{\tau_n}^{(\tau_n+\theta)\wedge T} u_n(s)\, {\rm d}s\right\Vert_{H}
	\\
	&\le \beta\mathbb{E}\int_{\tau_n}^{(\tau_n+\theta)\wedge T} \|u_n(s)\|_{H}\, {\rm d}s
	\lesssim \theta \beta\mathbb{E} \left[\sup_{s \in [0,T]}\|u_n(s)\|_H\right] \le \theta K_4.
		\end{align*}
The It\^o isometry and \eqref{eqn-G crescita} yield
\begin{align*}
\E\Vert J_6((\tau_n+\theta)\land T)&-J_6(\tau_n)\Vert^2_{\EAdual}
	\le \mathbb{E} \left\Vert \int_{\tau_n}^{(\tau_n+\theta)\wedge T}S_nG(S_nu_n(s))\, {\rm d}\newW(s)\right\Vert^2_H
	\\
	& = \mathbb{E}\int_{\tau_n}^{(\tau_n+\theta)\wedge T}\|S_nG(S_nu_n(s))\|^2_{\gamma(Y,H)}\, {\rm d}s
	\\
	&\le \mathbb{E}\int_{\tau_n}^{(\tau_n+\theta)\wedge T}\left(2C_1^2+2\tilde C_1^2\|u_n(s)\|^2_H \right)\, {\rm d}s
	\\
	&\lesssim \mathbb{E}\int_{\tau_n}^{(\tau_n+\theta)\wedge T}\left(1+\|u_n(s)\|^2_H\right)\, {\rm d}s
	\lesssim \theta + \theta \mathbb{E} \left[ \sup_{s \in [0,T]}\|u_n(s)\|_H^2\right] \lesssim \theta K_6.
\end{align*}

	By the Chebyschev inequality, we obtain for any given $\eta>0$
	\begin{align}\label{AldousEins}
	&\hspace{-1truecm}\Prob \left\{\Vert J_k((\tau_n+\theta)\land T)-J_k(\tau_n)\Vert_{\EAdual}\ge \frac{\eta}{6}\right\} \\
&\le \frac{6}{\eta} \E\Vert J_k((\tau_n+\theta)\land T)-J_k(\tau_n)\Vert_{\EAdual}\le \frac{ 6K_k \theta}{\eta}
\mbox{ 	for $k\in \{1,2,3,4\}$}
\nonumber
	\end{align}
 and
	\begin{align}\label{AldousZwei}
&\hspace{-1truecm}	\Prob \left\{\Vert J_k((\tau_n+\theta)\land T)-J_k(\tau_n)\Vert_{\EAdual}\ge \frac{\eta}{6}\right\}\\
&\le \frac{36}{\eta^2} \E\Vert J_k((\tau_n+\theta)\land T)-J_k(\tau_n)\Vert_{\EAdual}^2\le \frac{36 K_k \theta}{\eta^2}
\mbox{ for $k \in \{5,6\}$.}
\nonumber 	\end{align}
	
	Let us fix $\varepsilon>0$ and $\eta>0$. Due to  estimates $\eqref{AldousEins}$ and $\eqref{AldousZwei}$ we can choose $\delta_1,\dots,\delta_{6}>0$ such that
	\begin{align*}
	\Prob \left\{\Vert J_k((\tau_n+\theta)\land T)-J_k(\tau_n)\Vert_{\EAdual}\ge \frac{\eta}{6}\right\}\le \frac{\varepsilon}{6}
	\end{align*}
	for $0<\theta\le \delta_k$ and $k=1,\dots,6.$ With $\delta:= \min \left\{\delta_1,\dots,\delta_{6}\right\},$  using $\eqref{AldousStartingEstimate}$ we get
	\begin{align*}
	\Prob \left\{\Vert J_k((\tau_n+\theta)\land T)-J_k(\tau_n)\Vert_{\EAdual}\ge \eta\right\}\le \varepsilon
	\end{align*}
	for all $n\in\N$, $k=1,...6$ and $0<\theta\le \delta$ and therefore, the Aldous condition $[A]$ holds in $V^\ast.$	

\end{proof}
\begin{proof}[\emph{ad c):}]
This point has some similarities with point (a) proved above.
We prove it in Appendix \ref{App_B}.
\end{proof}

As an immediate consequence of Propositions \ref{MassEstimateGalerkinSolution} and
\ref{EstimatesGalerkinSolution} and the fact that $\|u\|^{2r}_V \lesssim_r \|u\|^{2r}_H +\energy(u)^r$ we obtain
\begin{corollary}
\label{cor-aprioriEA}
Fix $\newr\in [1,\infty)$ and let
 $u_0 $  be  an $\F_0$-measurable $\EA$-valued random variable
  with finite $r(\alpha+1)$-th moment. Then,
under Assumptions $\ref{assumption-A-space}$, $\ref{assumption-F_def}$, $\ref{assumption-stochastic}$,
  the following bound holds
		\begin{align}\label{eqn-5.22}
		\sup_{n\in\N}\E \Big[\sup_{t\in[0,T]} \norm{u_n(t)}_\EA^{2\newr}\Big]<\infty, \;\;\; \mbox{ for every }T>0.
		\end{align}
\end{corollary}

\subsection{Convergence. Proof of the first part of Theorem \ref{thm-existence}}\label{Section_Convergence}
\hfill

In this Section we prove part (i) of Theorem \ref{thm-existence}, that is the existence of a martingale solution of \eqref{eqn-ProblemStratonovich} which satisfies
conditions \eqref{eqn-stimaEnergy-full} and \eqref{eqn-propertySolution}.

We construct a solution to equation \eqref{eqn-ProblemStratonovich} by a suitable limiting process in the Galerkin equation \eqref{galerkinEquation}, exploiting the results of the previous sections.
Proposition \ref{EstimatesGalerkinSolution}(b) and Corollary \ref{cor-aprioriEA} provide the tightness to pass to the limit.
One proceeds as in \cite{Brz+H+W-2019} and \cite{BM}. We will just provide the main steps of the proof and refer to these papers for more details.

Let us recall from Section \ref{CompactnessSection} the definition of the space $Z_\infty$: 
\begin{align}
\label{eqn-Z_infty-2}
Z_\infty&= {C([0,\infty);\EAdual)} \cap {L^{\alpha+1}_{\mathrm{loc}}([0,\infty);\LalphaPlusEins)}
\cap C_w([0,\infty);\EA). 
\end{align}
\dela{
\begin{align*}
Z_T={C([0,T];\EAdual)}\cap {L^{\alpha+1}(0,T;L^{\alpha+1})} \cap C_w([0,T];\EA).
\end{align*}
}

Proposition \ref{EstimatesGalerkinSolution} (b)  and Corollary \ref{cor-aprioriEA}
provide the a priori estimates on the Galerkin approximation sequence; hence, this is tight  in $Z_\infty$ thanks to Proposition \ref{TightnessCriterion}.
Then by means of Corollary \ref{corollaryEstimatesToASconvergence}  we get the convergence.
More precisely,  there exist a subsequence $\left(u_{n_k}\right)_{k\in\N}$, a probability space $\left(\hat{\Omega},\hat{\F},\hat{\Prob}\right)$ and random variables $v_k, v:\hat{\Omega} \rightarrow Z_\infty$ with $\text{Law}_{\hat{\Prob}}(v_k)=\text{Law}_{\mathbb{P}}(u_{n_k})$ such that
\begin{equation}
\label{convergencev_n}
v_k\to v  \ \quad \hat{\Prob}-\text{a.s.} \quad  \text{in $Z_\infty$ for $k\to \infty$}.
\end{equation}
Moreover, arguing as in the proof of \cite[ Proposition 6.1(b)]{Brz+H+W-2019}, we infer that  $v_k \in  C\left([0,\infty),H_k\right)$ $\hat{\Prob}$-a.s. and \begin{equation}
\sup_{k\in\N} \hat{\mathbb{E}} \left[ \sup_{t \in [0,T]}\norm{v_k(t)}_{H}^{2\newr}+ \sup_{t \in [0,T] }\mathcal{E}(v_k(t))^r\right]<\infty,\quad \text{for any $T>0$,}
\end{equation}
from which, keeping in mind that $\|\cdot\|^{2r}_V \lesssim_r \|\cdot\|^{2r}_H +\energy(\cdot)^r$, we also infer
\begin{align}
\label{STAR}
\sup_{k\in\N} \hat{\mathbb{E}} \left[ \sup_{t \in [0,T]}\norm{v_k(t)}_V^{2\newr}\right]<\infty, \quad \text{for any $T>0$,}
.
\end{align}
Let us remark that
we also get
\begin{equation}\label{eqn-5.23}
\hat{\mathbb{E}} \left[ \sup_{t \in [0,T]}\norm{v(t)}_{H}^{2\newr}+ \sup_{t \in [0,T] }\mathcal{E}(v(t))^r\right] \le \liminf_{k} \hat{\mathbb{E}} \left[ \sup_{t \in [0,T]}\norm{v_k(t)}_{H}^{2\newr}+ \sup_{t \in [0,T] }\mathcal{E}(v_k(t))^r\right]<\infty.
\end{equation}
Hence, by Remark \ref{bound_rem}
\begin{equation}\label{confronto-norme}
\hat{\mathbb{E}} \left[ \sup_{t \in [0,T]}\norm{v(t)}_V^{2\newr}\right]
< \infty,\quad \text{for any $T>0$}
.
\end{equation}
The last two inequalities  prove respectively inequalities  \eqref{eqn-stimaEnergy-full} and \eqref{eqn-propertySolution}.

Since each $v_k$ has the same law as $u_{n_k}$, it is a martingale solution to equation \eqref{galerkinEquation}: expanding the arguments in \cite[Lemma 6.3]{Brz+H+W-2019} one can easily prove that each process $N_n: \hat{\Omega} \times [0,\infty) \rightarrow H_n$ defined by
\begin{align}
\label{Nn}
N_n(t)=-v_n(t)+ P_n u_0&+ \int_0^t \left[-\im A v_n(s)-\im P_n F(v_n(s))+b_n(v_n(s))-\beta v_n(s)\right] \df s
\end{align}
for $n\in \N$ and $t\in[0,\infty)$ is  an ${H}$-valued continuous square integrable martingale w.r.t. the filtration $\hat{\F}_{n,t}:=\sigma \left(v_n(s): s\le t\right)$.
As far as its quadratic variation process is concerned, in order to exploit classical results presented in a real  Hilbert space setting, see e.g., \cite{DPZ1},
we work with the real inner product in $H$ (and $H_n$). This means that the  quadratic variation process $\quadVar{N_n}$ is defined through
the property that for any $\psi,\phi \in H_n$ the process
\[
\skpHReal{N_n(t)}{\psi}\skpHReal{N_n(t)}{\phi}-\skpHReal{\quadVar{N_n}_t\psi }{\phi}, \qquad t \in [0,\infty),
\]
is a martingale.
Therefore we find that
	\begin{multline}
	\label{Nnqv}
	\quadVar{N_n}_t\psi
	= \sumM \int_0^t \im S_n B_m S_n v_n(s) \skpHReal{\im S_n B_m S_n v_n(s)}{\psi} \df s
	\\
	+\sumM \int_0^t \im S_n G(S_n v_n(s))e_m \skpHReal{\im S_n G(S_n v_n(s))e_m}{\psi} \df s
	\end{multline}
	for all $\psi \in {H}$ and $t\in[0,\infty)$.
The martingale property can be rephrased as
\begin{align}
\label{1lim}
	\hat{\mathbb{E}} \left[ \skpHReal{N_n(t)-N_n(s)}{\psi} h(v_n|_{[0,s]})\right]=0
	\end{align}
	and
	\begin{align}
	\label{2lim}
	\hat{\mathbb{E}} \Bigg[ \Bigg(&\skpHReal{N_n(t)}{\psi}\skpHReal{N_n(t)}{\varphi}-\skpHReal{N_n(s)}{\psi}\skpHReal{N_n(s)}{\varphi}
	\notag\\
	&\hspace{1 cm}-\sumM \int_s^t  \skpHReal{\im S_n B_m S_n v_n(r)}{\psi} \skpHReal{\im S_n B_m S_n v_n(r)}{\varphi} \df r
	\notag\\
	&\hspace{1 cm}-\sumM \int_s^t  \skpHReal{\im S_n G (S_n v_n(r))e_m}{\psi} \skpHReal{\im S_n G(S_n v_n(r))e_m}{\varphi} \df r
	\Bigg) h(v_n|_{[0,s]})\Bigg]=0,
	\end{align}		
	for all $0<s<t<\infty$,  $\psi, \varphi \in {H}$ and bounded continuous functions $h$ on $C([0,\infty),H_n).$

It is useful at this point to introduce the following notation. Let $\iota: \EA \hookrightarrow {H}$ be the usual embedding,
$\iota^\ast: {H} \rightarrow \EA$ its Hilbert-space-adjoint,
i.e. $\skpH{\iota u}{v}=\skp{u}{\iota^\ast v}_\EA$ for $u\in\EA$ and $v\in{H}.$
Further, we set $L:=\left(\iota^\ast\right)': \EAdual \rightarrow {H}$  as the dual operator of $\iota^\ast$ with respect to the
Gelfand triple $\EA\hookrightarrow H\eqsim H^\ast\hookrightarrow\EAdual.$
	\newline
Let us introduce the process
\begin{align*}
N(t):=-v(t)+ u_0+ \int_0^t \left[-\im A v(s)-\im  F(v(s))+b(v(s))-\beta v(s)\right] \df s, \quad t\in[0,T],
\end{align*}
which has $\EAdual$-valued continuous paths. Moreover, $N(t) \in L^2(\hat\Omega,\EAdual)$.

We now use the martingale property of $N_n$ for $n\in\N$ and a passage to the limit in \eqref{1lim} and \eqref{2lim}
to show that $L N$ is an ${H}$-valued continuous square integrable martingale with respect to the filtration
$\hat{\Filtration}=\left(\hat{\F}_t\right)_{t\in[0,T]},$ where $\hat{\F}_{t}:=\sigma \left(v(s): s\le t\right)$,
with quadratic variation given by
	\begin{align}
	\label{qvN}
	\quadVar{L N}_t\zeta
	=\sumM \int_0^t\left[\im  L B_m v(s) \skpHReal{\im L B_m v(s)}{\zeta} + \im  L G(v(s))e_m
	\skpHReal{\im L G( v(s))e_m}{\zeta}\right]\df s,
		\end{align}
	for all $\zeta \in {H}.$	
We just provide the main steps of the limiting process (for a detailed proof see \cite{Brz+H+W-2019})
and the computations of what is new.

Taking the limit as $n \rightarrow \infty$ in \eqref{1lim}, for $\psi \in \EA$, we obtain
\begin{footnote}{
For the proof see \cite[Lemma 6.2, Lemma 6.4 Steps 1-2]{Brz+H+W-2019}. Here, in addition, we have to consider the
convergence of the damping term but the needed estimates can be obtained rather easily.}
\end{footnote}
\begin{equation}
\label{martH}
\hat{\mathbb{E}} \left[\langle N(t)-N(s),\psi\rangle h(v|_{[0,s]}) \right]=0.
\end{equation}
Then we take the limit as $n\to\infty$  in \eqref{2lim}, for $\phi, \psi \in \EA$, and obtain
	\begin{align}
	\label{martes2}
	\hat{\mathbb{E}}\Bigg[ \Bigg(
	\dualityReal{N(t)}{\psi}&\dualityReal{N(t)}{\varphi}-\dualityReal{N(s)}{\psi}\dualityReal{N(s)}{\varphi}	\\ \notag
	&-\sumM \int_s^t  \dualityReal{ B_m v(r)}{\psi} \dualityReal{B_m v(r)}{\varphi} \df r
	\\\notag
	&-\sumM \int_s^t  \dualityReal{G v(r)e_m}{\psi} \dualityReal{G v(r)e_m}{\varphi} \df r
	\Bigg) h(v|_{[0,s]})\Bigg]=0.
	\end{align}
To prove the convergence of the first three terms in \eqref{2lim}
 one proceeds as in \cite[Lemma 6.4 Steps 3-4]{Brz+H+W-2019}.
The convergence of the four
th term is proved in the following lemma whose proof is postponed
to Appendix \ref{App_C}.
\begin{lemma}
 \label{convquadvari}
 Under Assumption \ref{assumption-stochastic}(iii), for all $0\le s\le t < \infty$, $\psi, \varphi\in \EA$, $h$ a bounded continuous function on ${C([0,\infty);\EAdual)}$, we have
 \begin{multline*}
\lim_{n \rightarrow \infty}\hat{\mathbb{E}}
 \Bigg[
\Bigg(
\sumM \int_s^t  \skpHReal{S_n G (S_n v_n(r))e_m}{\psi} \skpHReal{S_n G (S_n v_n(r))e_m}{\varphi} \df r
\Bigg) h(v|_{[0,s]})
\Bigg]
\\
=
\hat{\mathbb{E}} \Bigg[
\Bigg(
\sumM \dualityReal{ G  v(r)e_m}{\psi} \dualityReal{G  v(r)e_m}{\varphi}  \df r
\Bigg) h(v|_{[0,s]})
\Bigg].
\end{multline*}
\end{lemma}

Now let $\eta, \zeta \in H$. Then $i^\ast\eta$, $i^\ast\zeta \in \EA$ and for every $z$ in $\EAdual$ we have $\skpHReal{Lz}{\eta}=\dualityReal{z}{i^\ast\eta}$.
Thus from \eqref{martH} and \eqref{martes2} we deduce
\begin{equation}
\label{martH2}
\hat{\mathbb{E}} \left[\left( LN(t)-LN(s),\psi\right)_H h(v|_{[0,s]}) \right]=0.
\end{equation}
and
\begin{align}
	\label{martes2a}
	\hat{\mathbb{E}} \Bigg[ \Bigg(\skpHReal{L N(t)}{\eta}&\skpHReal{L N(t)}{\zeta}-\skpHReal{L N(s)}{\eta}\skpHReal{L N(s)}{\zeta}\\ \notag
	&-\sumM \int_s^t  \skpHReal{L  B_m v(r)}{\eta} \skpHReal{L B_m v(r)}{\zeta} \df r
	\\\notag
	&-\sumM \int_s^t  \skpHReal{L  G( v(r)e_m)}{\eta} \skpHReal{L G( v(r)e_m)}{\zeta} \df r
	\Bigg) h(v|_{[0,s]})\Bigg]=0.
	\end{align}
	Hence, from \eqref{martH2} and \eqref{martes2a}, we infer that $L N$ is a continuous, square integrable martingale in ${H}$ with respect to  $\hat{\F}_{t}:=\sigma \left(v(s): s\le t\right)$ and quadratic variation given by \eqref{qvN}.		
			
Therefore, with the usual Martingale Representation Theorem, see  \cite[Theorem 8.2]{DPZ1},
we can conclude that there exist two  cylindrical Wiener processes $\tilde{W}$ and $\tilde{\newW}$ on $Y$
defined on a probability space $
	\left(\tilde\Omega,\tilde\F,\tilde\Prob\right)=\left(\hat{\Omega} \times \hat{\hat{\Omega}}, \hat{\F}\otimes \hat{\tilde{\F}},  \hat{\Prob}\otimes\hat{\hat{\Prob}}\right)$
	with
	\begin{align}
	\label{LN}
	L N(t)
	=\int_0^t \im L B\left( v(s)\right) \df \tilde{W}(s)+\int_0^t \im L G\left( v(s)\right) \df \tilde{\newW}(s)
	\end{align}
	for $t\in [0,\infty).$
	Thanks to 	\eqref{confronto-norme}, the estimates
	\begin{multline*}
	\norm{B v}_{L^2([0,T]\times \Omega,\gamma(Y_1,\EAdual))}^2
	=\E \int_0^T \sumM \norm{B_m v(s)}_\EAdual^2 \df s
	\lesssim\E \int_0^T \sumM \norm{B_m v(s)}_\EA^2 \df s\\
	\le \E \int_0^T \left(\sumM \norm{B_m}_{{\mathscr{L}(\EA)}}^2\right) \norm{v(s)}_\EA^2 \df s
	\lesssim \E \int_0^T  \norm{v(s)}_\EA^2 \df s
	\lesssim \norm{v}_{L^2(\Omega,L^\infty(0,T;\EA))}^2 <+\infty	
	\end{multline*}
	and
	\begin{multline*}
	\norm{G (v)}_{L^2([0,T]\times \Omega,\gamma(Y_2,\EAdual))}^2
	=\E \int_0^T \sumM \norm{G (v(s))e_m}_\EAdual^2 \df s
	\lesssim\E \int_0^T \sumM \norm{G (v(s))e_m}_\EA^2 \df s
	\\
	= \mathbb{E} \int_0^T \|G(v(s))\|^2_{\gamma(Y_2, \EA)}\, {\rm d}s
	\lesssim 1+ \norm{v}_{L^2(\Omega,L^\infty(0,T;\EA))}^2<\infty ,
	\end{multline*}
yield that $LN$ in \eqref{LN} is a continuous martingale in $H$ and using the continuity of the operator $L$, we get
	\begin{align*}
	\int_0^t \im L B\left( v(s)\right) \df \tilde{W}(s)  +\int_0^t \im L B\left( v(s)\right) \df \tilde{\newW}(s)
	=
	L \left(\int_0^t \im  B\left( v(s)\right) \df \tilde{W}(s) +\int_0^t \im G\left(v(s)\right) \df \tilde{\newW}(s) \right)
	\end{align*}
	for all $t\in [0,T].$
	The definition of $N$ and the injectivity of $L$ yield the equality
	\begin{align}\label{vIsSolution}		
	\int_0^t \im B  v(s) \df \tilde{W}(s) +\int_0^t \im G (v(s)) \df \tilde{\newW}(s)=-v(t)+ u_0+ \int_0^t \left[-\im A v(s)-\im  F(v(s))+b(v(s))\right] \df s
	\end{align}
		in $\EAdual$ for $t\in [0,\infty).$

The estimates for properties \eqref{eqn-stimaEnergy-full} and $\eqref{eqn-propertySolution}$ and the weak continuity of the paths of $v$ in $\EA$ have already been shown at the beginning of the proof. Moreover, in view of (the beginning of) Remark \ref{rem-def of solution}, we have that $v$ is a $H$-valued continuous process. Hence, the system
$\left(\tilde{\Omega},\tilde{\F},\tilde{\Prob},\tilde{W},\tilde{\Filtration},v\right)$ is a martingale solution of equation
$\eqref{eqn-ProblemStratonovich}$ with the initial data  $\mu$, that satisfies  \eqref{eqn-stimaEnergy-full} and $\eqref{eqn-propertySolution}$.

It remains to prove Lemma \ref{convquadvari}. This is done in Appendix \ref{App_C}.

\subsection{Proof of the second part of Theorem \ref{thm-existence}}
\label{Section_inequalities}

Inequality \eqref{Ebound} is a consequence of the same inequality for the Galerkin approximation
\eqref{Galerkin-Ebound}, which is inherited by the limit.

\begin{remark}
\label{rem-FR1}
Theorem \ref{thm-existence} holds in a more general setting. It is sufficient for $A$ to satisfy
\cite[Assumption 2.1]{Brz+H+W-2019} and assume $(X, \Sigma, \mu_X)$ to be a $\sigma$-finite measure space with metric $\rho$
satisfying the doubling property and $D$ to be an open bounded subset of $X$ with $\mu_X(D)<\infty$.
Moreover, it is sufficient for the nonlinear term $F$ to satisfy \cite[Assumptions 2.4 and 2.6(i)]{Brz+H+W-2019}. This kind of
assumptions ensure the compactness of the embedding $V \subset H$ which is the crucial ingredient to prove the existence of a solution
by means of a tightness argument, see Section \ref{sec-exis}. In this general framework one can work in any space
dimension provided a suitable condition on $\alpha$ is taken into account, see \cite[Assumption 2.1(iv)]{Brz+H+W-2019}. Moreover, this framework  allows $A$ to be also a fractional power of Laplacian type operators considered so far; for more details see \cite[Section 3.4]{Brz+H+W-2019}.
\newline
One can easily check that the computations that lead to the existence of a martingale solution, as stated in  Theorem \ref{thm-existence}, hold true in this more general setting.
\end{remark}

\begin{remark}
\label{FR2}
 In light of Remark \ref{rem-FR1}, when one works under Assumptions \ref{assumption-A-space}(ii) or (iii), the
 regularity assumptions on the domain $\mathscr{O}$ are as follows.
To ensure the existence of a martingale solution, see  Theorem \ref{thm-existence}, it is sufficient for the domain $\mathscr{O}$
to be a bounded
open subset of $\mathbb{R}^2$  in the case of Assumption \ref{assumption-A-space}(ii) and a bounded open subset  of
$\mathbb{R}^2$ with Lipschitz boundary in the case of Assumption \ref{assumption-A-space}(iii); the same
 holds in dimension $d$ for suitable $\alpha$ depending on $d$.
In fact, to prove the existence of a martingale solution we just exploit the fact that the embedding $V \subset H$ is
continuous and compact. With the above mentioned regularity assumption on the domain, the continuity of the embeddings in the cases (ii)-(iii) follows by Leoni \cite[Theorem 11.23 and
Exercise 11.26]{Leoni}: roughly speaking, the regularity we require on the domain ensures that $\mathscr{O}$ is an
extension domain, see also Leoni \cite[Exercises 12.11 and 12.14]{Leoni}. The compactness of the embedding is instead ensured by the boundedness of $\mathscr{O}$.
\\
We emphasize that to prove  the pathwise uniqueness of solutions and  the existence of  invariant measures, an additional regularity on the domain  is  required. For more details see Remark \ref{FR2bis}.
\end{remark}

\section{Pathwise uniqueness}\label{sec-UniquenessSection}

In this Section we study the pathwise uniqueness for solutions to \eqref{eqn-ProblemStratonovich}.
We work under  Assumptions \ref{assumption-A-space}, \ref{assumption-F_def} and \ref{assumption-stochastic}.
Here it is crucial that the spatial dimension is 2,
whereas the result of existence of martingale solutions can be obtained in a more general setting.
Indeed,  by means of the Stichartz estimates we prove that any  martingale solution fulfilling
\eqref{eqn-propertySolution} enjoyes more regularity.
The deterministic and stochastic Strichartz estimates will be presented in
 Appendix \ref{Str_est_sec}; they are
based on the results of Blair, Smith and Sogge, see \cite{Blair+Sogge_2008}.
  These estimates allow us to work at the same time with the Laplace-Beltrami operator
  on a two-dimensional compact Riemannian manifold $M$ and the realization of the negative Laplace operator with Dirichlet or Neumann boundary
conditions on a smooth relatively compact  domain $\mathscr{O} \subset \mathbb{R}^2$.
This is a first difference with respect to \cite{Brz+H+W-2019} where the authors start from the Strichartz estimates due to Bernicot and Samoyeau, see \cite{Bernicot} and  \cite[Lemmas B.3 and B.4]{Brz+H+W-2019}.
Moreover, differently to \cite{Brz+H+W-2019}, in order to prove the  pathwise uniqueness we cannot work pathwise since our noise is not conservative. We address this issue by appealing to a classical argument contained in \cite{Sch1997}.

As usual, when we write $L^q$, $H^{s,q}$ without specifying the domain, we mean either $L^q(M)$, $H^{s,q}(M)$ or $L^q(\mathscr{O})$, $H^{s,q}(\mathscr{O})$.

\begin{lemma}\label{lem-F-continity}
Let $\vartheta \in (0,1)$. Then $F$ maps the space $V$ into $\D(A^{\frac{\vartheta }{2}})$ and
\begin{align}\label{eqn-F-growth}
  \|F(u) \|_{\D(A^{\frac{\vartheta }{2}})} &\lesssim  \|u \|_{V}^\alpha, \qquad u \in V.
  \end{align}
\end{lemma}
\begin{proof}
For the case of $(M,g)$ a compact manifold without boundary equipped with a Lipschitz metric g and $-A$ equal to the Laplace-Beltrami operator we refer to \cite[Lemma 7.1]{Brz+H+W-2019}.
Hence we have to  prove \eqref{eqn-F-growth} when $-A$ is the Laplace operator with either Dirichlet or Neumann
boundary conditions on a smooth relatively compact  subset $\mathscr{O}$ of $\mathbb{R}^2$.

Let us fix $\vartheta \in (0,1)$ and choose ${\news} \in (1,2)$ such that
\begin{equation}
\label{beta_s}
\vartheta  < \frac{2({\news}-1)}{{\news}}.
\end{equation}
 We start by proving that
\begin{equation}
\label{ineq_1s}
\|F(u)\|_{W^{1,{\news}}(\mathscr{O})} \lesssim \|u\|_{H^1(\mathscr{O})}^{\alpha},\;\; u\in H^1(\mathscr{O}).
\end{equation}
In order to prove \eqref{ineq_1s} we compute the weak derivative of $F(u)$:
\begin{equation*}
\nabla F(u)= \left( \frac{\alpha -1}{2}\right)|u|^{\alpha -3}\left(\bar u\nabla u+u \nabla \bar u\right)u +|u|^{\alpha-1}\nabla u, \qquad  \text{for an arbitrary} \ u \in H^1(\mathscr{O}).
\end{equation*}
The H\"older inequality and the Sobolev embedding $H^1(\mathscr{O}) \subset L^q(\mathscr{O})$ with $q=\frac{2{\news}(\alpha-1)}{2-{\news}}$, see Proposition \ref{B1}(i), yield
\begin{equation*}
\|\nabla F(u)\|_{L^{\news}(\mathscr{O})}
\lesssim \| |u|^{\alpha-1}\|_{L^{\frac{2{\news}}{2-{\news}}}(\mathscr{O})}\|\nabla u\|_{L^2(\mathscr{O})}
\lesssim \|u\|^{\alpha-1}_{L^q(\mathscr{O})}\|\nabla u\|_{L^2(\mathscr{O})} \lesssim \|u\|^{\alpha}_{H^1(\mathscr{O})}.
\end{equation*}
Similarly, the Sobolev embedding $H^1(\mathscr{O})\subset L^{{\news}\alpha}(\mathscr{O})$ yields
\begin{equation*}
\|F(u)\|_{L^{\news}(\mathscr{O})} \simeq \|u\|_{L^{{\news}\alpha}(\mathscr{O})}^{\alpha} \lesssim \|u\|^{\alpha}_{H^1(\mathscr{O})},\;\; u\in H^1(\mathscr{O})
\end{equation*}
and thus \eqref{ineq_1s} immediately follows. By the choice of ${\news}$ in \eqref{beta_s}, Proposition \ref{B1}(iii) ensures that Sobolev embedding $H^{1,{\news}}(\mathscr{O})\subset H^{\vartheta }(\mathscr{O})$ holds and thus
\begin{equation}
\label{general_est}
\|F(u)\|_{H^{\vartheta }(\mathscr{O})} \lesssim \|u\|^{\alpha}_{H^1(\mathscr{O})}.
\end{equation}

Observe now that, when on $\mathscr{O}$ we consider the Neumann boundary conditions, thanks to \eqref{realization_sN}, $H^{\vartheta }=\D(A_N^{\frac{\vartheta }{2}})$ and thus from \eqref{general_est}, \eqref{eqn-F-growth} immediately follows.

In the case of the Dirichlet boundary conditions, from \eqref{realization_sD} and \eqref{general_est}, \eqref{eqn-F-growth}
immediately follows when $\vartheta  \in \left(0, \frac 12\right)$. When $\vartheta  \in \left(\frac 12,1\right)$, \eqref{eqn-F-growth}
is obtained by a density argument: for $\vartheta  \in \left(\frac 12,1\right)$, by \eqref{realization_sD},
$\D(A_D^{\frac{\vartheta }{2}})=H_0^{\vartheta }$, where $H_0^{\vartheta }$ is the closure of $C^{\infty}_0$ w.r.t. the norm
$\|\cdot\|_{H^{\vartheta ,2}}$. One can easily verify that $F$ maps $C^{\infty}_0$ in itself and thus, by \eqref{general_est} we
deduce \eqref{eqn-F-growth} when $\vartheta  \in \left(\frac 12,1\right)$.
\end{proof}

We now reformulate problem \eqref{eqn-ProblemStratonovich} in the mild form to show additional regularity properties of solutions to \eqref{eqn-ProblemStratonovich} that satisfies \eqref{eqn-propertySolution}.

\begin{proposition}\label{prop-mild solution-new}
Fix  $r \in [1, \infty)$
and let $\mu$ be a Borel probability measure on $V$  whose $r(\alpha+1)$-th moment is finite.\\ Let
$\left(\tilde{\Omega},\tilde{\F},\tilde{\Prob},\tilde{W},\tilde{\newW},\tilde{\Filtration},u\right)$ be a martingale solution to
\eqref{eqn-ProblemStratonovich}  with $\law{\tilde{\mathbb{P}}}{u(0)}=\mu$ on $\mathscr{B}(V)$
which satisfies the condition  \eqref{eqn-propertySolution}.
If $2<p,q<\infty$  satisfy  the  following admissibility condition
\begin{equation}\label{eqn-admissibility condition}
\frac 2p+\frac 2q=1
\end{equation}
 and  $\vartheta  \in \left(\frac{4}{3p},1\right)$, then the following hold.
\begin{itemize}
\item For any $T>0$, \begin{equation}
\label{more_regularity}
u \in 
L^{2r/\alpha}(\tilde{ \Omega};Y^\vartheta _T),
\end{equation}
where  $Y^\vartheta_T$ is a Banach space defined as
\begin{align}\label{eqn-Y^theta_T}
Y^\vartheta _T=L^p(0,T;\D(A_q^{\frac{\vartheta }2-\frac{2}{3p}})) \cap C([0,T];\D(A^{\frac{\vartheta }{2}})),
\end{align}
endowed with a norm
\begin{equation*}
\|\cdot\|_{Y_T^{\vartheta }}=\|\cdot\|_{L^p(0,T;\D(A_q^{\frac{\vartheta }2-\frac{2}{3p}}))}+ \|\cdot\|_{L^\infty(0,T;\D(A^{\frac{\vartheta }{2}}))}.
\end{equation*}
In particular,
\begin{equation}
\label{double_star}
u \in C\left([0, \infty);\D(A^{\frac{\vartheta }2})\right), \qquad \tilde{\mathbb{P}}-a.s.
\end{equation}
\item
For every $t \in [0,\infty)$ the equality
 \begin{align}
\label{mild_form}
\im u(t)&= \im e^{-\im tA}u_0+ \int_0^t e^{-\im(t-\tau)A}F(u(\tau))\, {\rm d}\tau
\notag\\
&+\im \int_0^t e^{-\im(t-\tau)A}b(u(\tau))\, {\rm d}\tau - \im \beta   \int_0^t e^{-\im(t-\tau)A}u(\tau)\, {\rm d}\tau
\notag\\
&+\int_0^t e^{-\im(t-\tau)A}B(u(\tau))\, {\rm d}\tilde{W}(\tau)+\int_0^t e^{-\im(t-\tau)A}G(u(\tau))\, {\rm d}\tilde{\newW}(\tau)
\end{align}
is satisfied $\tilde{\Prob}$-a.s.
in $\D(A^{\frac{\vartheta }{2}})$.
\end{itemize}
\end{proposition}


\begin{proof}
   Let $\left(\tilde{\Omega},\tilde{\F},\tilde{\Prob},\tilde{W},\tilde{\newW},\tilde{\Filtration},u\right)$ be
   a martingale solution  to \eqref{eqn-ProblemStratonovich}  given  in Theorem \ref{thm-existence}
  such that it has the regularity
   \eqref{eqn-propertySolution}.
Let us at first show  the equality \eqref{mild_form} makes sense in
$\D(A^{-\frac 32})$.  Let us notice that for  $\sigma=-\frac32$,
 the group $(e^{-\im tA})_{t\ge 0}$ on $L^2$ extends to a $C_0$-group
 $(T_{\sigma}(t))_{t\ge 0}$ on $\D(A^{-\sigma})$ with the generator
 $-\im A_\sigma$,  where  $\D(A_\sigma)=\D(A^{-\sigma+1})$, i.e.  $-\im A_\sigma$ is a suitable extension of  $-\im A$.
 To keep the notation simple we will denote this semigroup  by $(e^{-\im tA})_{t \ge 0}$.\\
Let us choose and fix $t \in (0,\infty)$.
We apply the It\^o formula, see \cite[Theorem 2.4]{BvNVW},  to the process $\im f(\tau, u(t-\tau))$, $\tau \in [0,t]$,  where $f$ is the function defined as
\begin{equation*}
f : [0,t]\times \D(A^{-\frac 12}) \ni  (\tau,x) \mapsto e^{-\im (t-\tau)A}x \in \D(A^{-\frac 32}),
\end{equation*}
where we recall that  $\D(A^{-\frac 12})=\EA^\ast$.
Obviously $f$ is of  $C^{1,2}$-class and since it  follows from the assumptions that  we can apply
Theorem 2.4  of \cite{BvNVW},    we deduce that $\tilde{\mathbb{P}}$-a.s.,
\begin{align*}
\im u(t)
&= \im e^{-\im tA}u(0)+ \int_0^t e^{-\im(t-\tau)A}F(u(\tau))\, {\rm d}\tau
\\
&+\im \int_0^t e^{-\im(t-\tau)A}b(u(\tau))\, {\rm d}\tau - \im \beta  \int_0^t e^{-\im(t-\tau)A}u(\tau)\, {\rm d}\tau
\\
&+\int_0^t e^{-\im(t-\tau)A}B(u(\tau))\, {\rm d}\tilde{W}(\tau)+\int_0^t e^{-\im(t-\tau)A}G(u(\tau))\, {\rm d}\tilde{\newW}(\tau) \;\mbox{ in $\D(A^{-\frac 32})$}.
\end{align*}
We now use the Strichartz estimates from Lemma \ref{Strichartz_lemma} to improve the regularity of the solution.
Let us consider two Strichartz pairs: $(\infty, 2)$ and $(p,q)$ with $p,q \in (2,\infty)$ such that $\frac 2p+\frac 2q=1$. Moreover, we choose and fix a  $\vartheta  \in \left(\frac{4}{3p},1\right)$.

Let $T>0$ and let $Y^\vartheta _T$ be the Banach  space defined by \eqref{eqn-Y^theta_T} for this choice of Strichartz pairs. Notice that  by assumption the exponent $\frac{\vartheta }2-\frac{2}{3p}$ is positive.
By estimating the terms in the RHS of the above expression we will prove \eqref{more_regularity}, \eqref{double_star} and that identity
\eqref{mild_form} holds a.s. in $\D(A^{\frac{\vartheta }{2}})$.
\\
Thanks to the homogeneous Strichartz estimate \eqref{eqn-Strichartz-211} we get
\begin{align}
\label{strong_cond}
\|e^{-\im tA}u(0)\|_{L^{2r/\alpha}(\tilde{\Omega},Y^{\vartheta }_T)}&\lesssim_T \|u(0)\|_{L^{2r/\alpha}(\tilde{\Omega},(\D(A^{\frac{\vartheta }{2}}))} \lesssim \|u(0)\|_{L^{2r/\alpha}(\tilde{\Omega},\EA)}
\notag\\
&\lesssim \|u(0)\|_{L^{2r}(\tilde{\Omega},\EA)}\lesssim\|u(0)\|_{L^{r(\alpha+1)}(\tilde{\Omega},\EA)} < \infty,
\end{align}
where in the second to last inequality we exploited the embedding $V \subset D(A^{\frac{\vartheta }{2}})$ and in the last two inequalities we used the embeddings $L^{r(\alpha+1)}\subset L^{2r}\subset L^{2r/\alpha}$ since $2r>2r/\alpha$ and $r(\alpha+1)>2r$ being $\alpha>1$ by Assumption \ref{assumption-F_def}. Moreover, $\|u(0)\|_{L^{r(\alpha+1)}(\tilde{\Omega},\EA)}  < \infty$ since $\law{\tilde{\mathbb{P}}}{u(0)}=\mu$ and $\mu$ has finite $r(\alpha+1)$-th moment by assumptions.\\
The inhomogeneous Strichartz estimate \eqref{eqn-Strichartz-212}, Lemma \ref{lem-F-continity} and \eqref{eqn-propertySolution} yield
\begin{align*}
\left \Vert\int_0^{\cdot} e^{-\im(\cdot-\tau)A} F(u(\tau))\ {\rm d}\tau\right\Vert_{L^{2r/\alpha}(\tilde{\Omega};Y^{\vartheta }_T)}
&\lesssim_T
\|F(u)\|_{L^{2r/\alpha}(\tilde{\Omega},L^1(0,T;\D(A^{\frac{\vartheta }{2}})))}
\\
&\lesssim \|u\|^{\alpha}_{L^{2r}(\tilde{\Omega},L^{\alpha}(0,T;V))}
\lesssim \|u\|^{\alpha}_{L^{2r}(\tilde{\Omega},L^{\infty}(0,T;V))}
< \infty.
\end{align*}
Similarly, since we assume  $\vartheta <1$, from the inhomogeneous Strichartz estimate \eqref{eqn-Strichartz-212} and \eqref{eqn-propertySolution} 
we infer that
\begin{align*}
\left \Vert   \int_0^{\cdot} e^{-\im(\cdot-\tau)A} u(\tau)\ {\rm d}\tau\right\Vert_{L^{2r/\alpha}(\tilde{\Omega};Y^{\vartheta }_T)}
&\lesssim
 \|u\|_{L^{2r/\alpha}(\tilde{\Omega},L^1(0,T;\D(A^{\frac{\vartheta }{2}})))}
\lesssim \|u\|_{L^{2r}(\tilde{\Omega},L^{\infty}(0,T;V))}< \infty.
\end{align*}
The It\^o correction term can be estimated as follows. Recalling \eqref{Strat_cor}, by the inhomogeneous Strichartz estimate \eqref{eqn-Strichartz-212}, and \eqref{noiseBoundsEnergy} and \eqref{eqn-propertySolution} we infer that
	\begin{align*}
	\left \Vert\int_0^{\cdot} e^{-\im (\cdot-\tau)A}b(u(\tau))\df \tau\right\Vert_{L^{2r/\alpha}(\tilde{\Omega},Y^{\vartheta }_T)}
	&\lesssim \norm{b(u)}_{L^{2r/\alpha}(\tilde{\Omega}, L^1(0,T; \D(A^{\frac{\vartheta }{2}})))}
	\\
&\lesssim \norm{b(u)}_{L^{2r}(\tilde{\Omega}, L^\infty(0,T; V))}
	\lesssim \norm{u}_{L^{2r}(\tilde{\Omega}, L^\infty(0,T; V))}<\infty.
	\end{align*}
For what concerns the stochastic convolution term involving the operator $B$, by means of the stochastic Strichartz estimate \eqref{eqn-Strichartz-213}, and   Assumptions
\eqref{noiseBoundsEnergy}  and \eqref{eqn-propertySolution}, we obtain
	\begin{align*}
	\left\Vert \int_0^{\cdot} e^{-\im (\cdot-\tau)A}B(u(\tau))\df \tilde{W}(\tau)\right\Vert_{L^{2r/\alpha}(\tilde{\Omega},Y_T^{\vartheta })}
	&	
	\lesssim \norm{B(u)}_{L^{2r/\alpha}(\tilde{\Omega}, L^2(0,T; \gamma(Y_1,\D(A^{\frac{\vartheta }{2}}))))}
		\\
\lesssim \norm{B(u)}_{L^{2r/\alpha}(\tilde{\Omega}, L^2(0,T; \gamma(Y_1,V)))}
	&\lesssim \norm{u}_{L^{2r/\alpha}(\tilde{\Omega}, L^2(0,T; V))}
	\lesssim \norm{u}_{L^{2r}(\tilde{\Omega}, L^{\infty}(0,T;V))}<\infty
	\end{align*}
when $\frac{2r}\alpha\ge 1$. For smaller values we  first estimate with the moment of order 1 and then conclude as above. The same will be done for the next stochastic integral.
The stochastic Strichartz estimate \eqref{eqn-Strichartz-213}, and Assumptions \eqref{crescitaGEA} and \eqref{eqn-propertySolution} yield
\begin{align*}
\left\Vert \int_0^{\cdot} e^{-\im (\cdot-\tau)A}G(u(\tau))\df \tilde{\newW}(\tau)\right\Vert_{L^{2r/\alpha}(\tilde{\Omega},Y_T^{\vartheta })}
&\lesssim \norm{G(u)}_{L^{2r/\alpha}(\tilde{\Omega}, L^2(0,T; \gamma(Y_2,\D(A^{\frac{\vartheta }{2}}))))}
\\
&\hspace{-3truecm} \lesssim \norm{G(u)}_{L^{2r/\alpha}(\tilde{\Omega}, L^2(0,T; \gamma(Y_2,V)))}
= \mathbb{E} \left[ \left(\int_0^T \|G(u(\tau))\|^2_{\gamma(Y_2,V)}\,{\rm d}\tau \right)^{r/\alpha} \right]
\\
&\hspace{-2truecm}\lesssim 1 + \|u\|_{L^{2r/\alpha}(\tilde{\Omega}, L^2(0,T;V))}
\lesssim 1 + \|u\|_{L^{2r}(\tilde{\Omega}, L^{\infty}(0,T;V))} < \infty.
\end{align*}
Thus
the mild equation \eqref{mild_form} holds $\tilde{\mathbb{P}}$-a.s. in $\D(A^{\frac{\vartheta }{2}})$ for every $t \in [0,\infty)$.
 Therefore,
 thanks to the pathwise continuity
 of the deterministic and stochastic integrals, we get \eqref{more_regularity} and \eqref{double_star}.
\end{proof}

\begin{remark}
Let us note the following difference between our result and Proposition 7.2 in \cite{Brz+H+W-2019}. Here in \eqref{eqn-propertySolution} we assume the $\EA$-regularity of the solution while in the paper \cite{Brz+H+W-2019} only $H^s$-regularity  was assumed, see assumption (7.6) therein.   We have to make a stronger assumption here because the Strichartz estimates for  a  boundaryless manifold  are stronger
than the Strichartz estimates for a bounded domain with smooth boundary, see Remark \ref{rem-Strichartz comparison}. However, this stronger assumption  is fully sufficient for our purposes.
\end{remark}

We are now ready to prove the pathwise uniqueness of the martingale solutions to \eqref{eqn-ProblemStratonovich} satisfying condition   \eqref{eqn-propertySolution}. We will need the following result which exploits the gain of regularity of solutions proved in Proposition \ref{prop-mild solution-new}.

\begin{lemma}
\label{lem-psi}
Assume $r \in [1, \infty)$ and let $\mu$ be a Borel probability measure on $V$ whose  $r(\alpha+1)$-th moment is finite.\\ Let
$\left(\tilde{\Omega},\tilde{\F},\tilde{\Prob},\tilde{W},\tilde{\newW},\tilde{\Filtration},u\right)$ be a martingale solution to
\eqref{eqn-ProblemStratonovich} with $\law{\tilde{\mathbb{P}}}{u(0)}=\mu$ on $V$
such that \eqref{eqn-propertySolution} holds.
Then the trajectories of the  process $h$ defined by
\begin{equation*}
h(s):= \|u(s)\|_{L^{\infty}}^{\alpha-1}, \qquad s \in [0,\infty),
\end{equation*}
belong to $L_{\mathrm{loc}}^1([0,\infty))$, $\tilde{\Prob}$-a.s.
\end{lemma}

\begin{proof}[Proof of Lemma \ref{lem-psi}] \textbf{Step 1.} Let us assume that $\alpha \in (1,3]$.  We choose $p,q\in (2,\infty)$
satisfying the admissibility condition \eqref{eqn-admissibility condition}. Note that in this case $p> \alpha-1$.
\\
Since $1-\frac{2}{3p}>\frac{4}{3p}$ we can choose  a number   $\vartheta  \in \left(1-\frac{2}{3p},1\right) \cap \left(\frac{4}{3p},1\right)$.
Thanks to Proposition \ref{prop-mild solution-new}, for any $T>0$,
$u \in L^p(0,T;\D(A_q^{\frac{\vartheta }{2}-\frac{2}{3p}}))$ $\tilde{\mathbb{P}}$-a.s.. Moreover, by  Proposition \ref{B1}(ii), Definition \ref{def-sobolev spaces with boundary conditions} and Theorem \ref{dom_D_N},
which ensures that $\D(A_{B_q}^{\frac{\vartheta }{2}-\frac{2}{3p}}) \subset H^{\vartheta -\frac{4}{3p},q}(\mathscr{O})$,
for $B=D,N$,  and Proposition \ref{propB6}(i)-(ii), we infer  that $\D(A_q^{\frac{\vartheta }{2}-\frac{2}{3p}})\embed L^{\infty}$.
In fact, this embedding holds true
because $\vartheta -\frac{4}{3p}-\frac2q=\vartheta -1 +\frac{2}{3p}$ so that  $\vartheta -\frac{4}{3p}-\frac2q>0\iff \vartheta  >1- \frac{2}{3p}$.
Since $p> \alpha-1$, by applying the H\"older inequality in time  we see that the process $h$ satisfies,  $\tilde{\mathbb{P}}$-a.s.
\begin{equation}
\label{eqn-esHold-1}
\|h\|_{L^1(0,T)} \lesssim  \|u\|^{\alpha-1}_{L^{\alpha-1}(0,T;\D(A_q^{\frac{\vartheta }{2}-\frac{2}{3p}}))}\lesssim  \|u\|^{\alpha-1}_{L^p(0,T;\D(A_q^{\frac{\vartheta }{2}-\frac{2}{3p}}))} .
\end{equation}
 \textbf{Step 2.} Let us assume that   $\alpha >3$.  Then we set $p:=\alpha-1$. Since $p>2$ we can find $q>2$ such that the admissibility condition
\eqref{eqn-admissibility condition} holds. As in  Step 1 we  choose a number   $\vartheta  \in \left(1-\frac{2}{3p},1\right) \cap \left(\frac{4}{3p},1\right)$ and observe that
 $\vartheta -\frac{4}{3p}-\frac2q>0$. Therefore, we have  $\D(A_q^{\frac{\vartheta }{2}-\frac{2}{3p}})\embed L^{\infty}$.
Hence,  we get  the following version of  estimate \eqref{eqn-esHold-1}
\begin{equation}
\label{eqn-esHold-2}
\|h\|_{L^1(0,T)} \lesssim   \|u\|^{\alpha-1}_{L^{\alpha-1}(0,T;\D(A_{q}^{\frac{\vartheta }{2}-\frac{2}{3p}}))}=\|u\|^{p}_{L^{p}(0,T;\D(A_{q}^{\frac{\vartheta }{2}-\frac{2}{3p}}))}.
\end{equation}
\textbf{Step 3}
We  conclude that, for any $T>0$, $h\in L^1(0,T)$ $ \tilde{\mathbb{P}}$-a.s. as a consequence of inequalities \eqref{eqn-esHold-1} and \eqref{eqn-esHold-2} combined with the moment estimate \eqref{more_regularity}.
\end{proof}

Now we are ready to prove the pathwise uniqueness. This holds for any martingale solution $u$,
as in Definition \ref{def-martingale solution} and enjoying \eqref{eqn-propertySolution}, since we have proved that, under suitable conditions on the initial distribution, for these solutions there is  the additional regularity
$u\in L_{\mathrm{loc}}^{\alpha-1}([0,\infty);L^\infty)\quad \tilde \Prob-a.s.$
 and  $u\in L^4(\tilde \Omega; L^4(0,T;H))$, for every $T>0$,
at least.

\begin{theorem}\label{thm-pathwise uniqueness}
 Let Assumptions \ref{assumption-A-space},  \ref{assumption-F_def}, \ref{assumption-stochastic} be in force. Assume $r \in [2, \infty)$.
Let $\mu$ be a Borel probability measure on $V$ whose $r(\alpha+1)$-th moment is finite.
 \\
 Let
$\left(\tilde{\Omega},\tilde{\F},\tilde{\Prob},\tilde{W},\tilde{\newW},\tilde{\Filtration},u_i\right)$, $i=1,2$ be two martingale solutions to \eqref{eqn-ProblemStratonovich}
with random initial data of law $\mu$ and both satisfying condition \eqref{eqn-propertySolution}.
Then
\begin{trivlist}
\item[(i)] these solutions to equation \eqref{eqn-ProblemStratonovich} are pathwise unique, i.e.
\begin{equation}\label{eqn-equivalence}
\tilde{\mathbb{P}}\left(u_1(t)=u_2(t) \ \text{for all} \ t \in [0,\infty)\right)=1.
\end{equation}
\item[(ii)] martingale  solutions to equation \eqref{eqn-ProblemStratonovich} are unique in law, i.e.  if
$\left(\widetilde{\Omega_{i}},\widetilde{\F_{i}},\widetilde{\Prob_{i}},\widetilde{W_{i}},\widetilde{\newW_{i}},\widetilde{\Filtration_{i}},u_i\right)$, $i=1,2$ be two martingale solutions to \eqref{eqn-ProblemStratonovich}
with random initial data of law $\mu$ and both satisfying condition \eqref{eqn-propertySolution}, then
\begin{equation}\label{eqn-unique in law}
\law{\widetilde{\Prob}_{1}}{ u_1}= \law{\widetilde{\Prob}_{2}}{u_2} \mbox{ on } Z_\infty.
\end{equation}
\end{trivlist}

\end{theorem}
\begin{proof}
Let us first deal with assertion  (ii). In view of assertion (i), it is a consequence   of \cite[ Theorem  2]{Ondrejat_2004_Uniqueness}, the second assertion; see also Theorem 12.1 therein.

Let us next deal with assertion  (i).
Since the noise is not conservative we can not work pathwise as in \cite{Brz+H+W-2019}. We prove the uniqueness of the solution by means of a rather classical argument, see \cite{Sch1997}.
Take two solutions $\left(\tilde{\Omega},\tilde{\F},\tilde{\Prob},\tilde{W},\tilde{\newW},\tilde{\Filtration},u_i\right)$, $i=1,2$, with the same initial data $\mu$ on $V$ which  satisfy
  \begin{align}
&u_i \in L^{2r}(\tilde{\Omega},L^{\infty}(0,T;V))\; \text{ for } \ i=1,2 \mbox{ and }T>0.
\label{reg_Sch}
\end{align}

Define $v:= u_1-u_2$. This difference satisfies
\begin{equation*}
\begin{cases}
{\rm d}v(t)= -[\im Av(t)+\im\left(F(u_1(t))-F(u_2(t))\right)+\beta v(t)]{\rm d}t +
\\
\qquad \quad \ - \im B(v(t)) \circ {\rm d} \tilde{W}(t)-\im[G(u_1(t))-G(u_2(t))] {\rm d}\tilde{\newW}(t)
\\
v(0)=0.
\end{cases}
\end{equation*}
We use the It\^o formula to compute ${\rm d}\left( e^{-\int_0^t\psi(s)\, {\rm d}s}\|v(t)\|^2_H\right)$, by choosing a process  $\psi$ as
\begin{equation}
\label{u0}
\psi(t):=2\left[\|u_1(t)\|_{L^{\infty}}^{\alpha-1} +\|u_2(t)\|^{\alpha-1}_{L^{\infty}}-\beta\right] +L_G , \qquad t \in [0,\infty),
\end{equation}
with $L_G$ the Lipschitz constant given in \eqref{Lipschitz_G}. Thanks to
Lemma \ref{lem-psi} we have that $\psi \in L^1_{\text{loc}}(0,\infty)$, $\tilde{\mathbb{P}}$-a.s.. For $t \in [0,\infty)$, we have
\begin{align}
\label{u1}
{\rm d}\left( e^{-\int_0^t\psi(s)\, {\rm d}s}\|v(t)\|^2_H \right)= -\psi(t)e^{-\int_0^t\psi(s)\,{\rm d}s}\|v(t)\|^2_H\, {\rm d}t + e^{-\int_0^t\psi(s)\, {\rm d}s}\,{\rm d}\|v(t)\|^2_H,
\end{align}
where
\begin{align}
\label{u2}
{\rm d} \|v(t)\|_H^2
&= 2 \text{Re}\langle v(t), [G(u_1(t))-G(u_2(t))]\, {\rm d}\tilde{\newW}(t)\rangle
\\
&\hspace{-1truecm}+\left[ 2\text{Re}\langle v(t), -\im F(u_1(t))+ \im F(u_2(t))\rangle -2 \beta \|v(t)\|^2_H+\|G(u_1(t))-G(u_2(t))\|^2_{\gamma(Y_2,H)} \right] \, {\rm d}t.
\notag
\end{align}
By means of inequality \eqref{Lipschitz_G} we have the following estimate
\begin{equation}
\label{u3}
\|G(u_1(t))-G(u_2(t))\|^2_{\gamma(Y_2,H)} \le L_G \|v(t)\|^2_H, \;\; t \geq 0.
\end{equation}
The inequality
\begin{equation*}
|F(z_1)-F(z_2)| \lesssim \left(|z_1|^{\alpha-1}+|z_2|^{\alpha-1}\right) |z_1-z_2|, \qquad z_1, z_2 \in \mathbb{C},
\end{equation*}
and the H\"older inequality yield
\begin{equation}
\label{u4}
\text{Re}\langle v(t), -\im F(u_1(t))+ \im F(u_2(t))\rangle
\lesssim \|v(t)\|^2_H \left[\|u_1(t)\|_{L^{\infty}}^{\alpha-1} +\|u_2(t)\|^{\alpha-1}_{L^{\infty}}\right], \;\; t\geq 0.
\end{equation}
By means of 
\eqref{u3} and \eqref{u4} we estimate \eqref{u1} as follows
\begin{align*}
{\rm d}\left( e^{-\int_0^t\psi(s)\, {\rm d}s}\|v(t)\|^2_H \right)
&\lesssim
-\psi(t)e^{-\int_0^t\psi(s)\,{\rm d}s}\|v(t)\|^2_H\, {\rm d}t
\\
&+ e^{-\int_0^t\psi(s)\, {\rm d}s} \left[ 2 \|v(t)\|^2_H \left[\|u_1(t)\|_{L^{\infty}}^{\alpha-1} +\|u_2(t)\|^{\alpha-1}_{L^{\infty}}\right]\, {\rm d}t-2 \beta \|v(t)\|^2_H\, {\rm d}t \right.
\\
&\left.\qquad +L_G\|v(t)\|^2_H\, {\rm d}t+ 2 \text{Re}\langle v(t), [G(u_1(t))-G(u_2(t))]\, {\rm d}\tilde{\newW}(t)\rangle \right]
\\
&= e^{-\int_0^t\psi(s)\, {\rm d}s}\|v(t)\|^2_H
\left[-\psi(t) +2 \left[\|u_1(t)\|_{L^{\infty}}^{\alpha-1} +\|u_2(t)\|^{\alpha-1}_{L^{\infty}}-\beta\right]+L_G\right]\, {\rm d}t
\\
& \qquad + 2 e^{-\int_0^t\psi(s)\, {\rm d}s}\text{Re}\langle v(t), [G(u_1(t))-G(u_2(t))]\, {\rm d}\tilde{\newW}(t)\rangle.
\end{align*}
Therefore, recalling \eqref{u0}, we obtain
\begin{equation}
\label{Schmalfuss}
e^{-\int_0^t\psi(s)\, {\rm d}s}\|v(t)\|^2_H
\lesssim
2\int_0^te^{-\int_0^r\psi(s)\, {\rm d}s} \text{Re}\langle v(r), [G(u_1(r))-G(u_2(r))]\, {\rm d}\tilde{\newW}(r)\rangle.
\end{equation}
Let us observe that the right hand side of \eqref{Schmalfuss} is a square integrable martingale. Indeed, using inequality \eqref{Lipschitz_G} we get
\begin{align*}
& \mathbb{E}\left[ \int_0^te^{-2\int_0^r\psi(s)\, {\rm d}s}\|v(r)\|^2_H\|G(u_1(r))-G(u_2(r))\|^2_{\gamma(Y_2,H)}\,{\rm d}r\right]
\\
& \hspace{1truecm} \lesssim e^{4\beta t} L_G^2\mathbb{E}\left[ \int_0^t\|v(r)\|^4_H{\rm d}r\right] \lesssim \|v\|^4_{L^4(\tilde{\Omega};L^4(0,T;H))}
\lesssim \|v\|^4_{L^{2r}(\tilde{\Omega};L^\infty(0,T;V))},
\end{align*}
which is finite thanks to \eqref{reg_Sch} because $r\geq 2$.
Therefore, by taking the expected value in both sides of \eqref{Schmalfuss}
 we get
\begin{equation*}
\mathbb{E} \left[e^{-\int_0^t\psi(s)\, {\rm d}s}\|v(t)\|^2_H \right]\le 0, \qquad \forall \ t \in [0,\infty).
\end{equation*}
Thus, in particular, for any $t \in [0,\infty)$
\begin{equation*}
e^{-\int_0^t\psi(s)\, {\rm d}s}\|v(t)\|^2_H=0, \qquad \tilde{\mathbb{P}}-a.s.
\end{equation*}
Therefore, if we take a sequence $\{t_k\}_{k=1}^{\infty}$, which is dense in $[0,\infty)$, we have
\begin{equation*}
\tilde{\mathbb{P}}\left( \|v(t_k)\|_H=0 \ \text{for all}\ k \in \mathbb{N}\right)=1.
\end{equation*}
Since by Definition \ref{def-martingale solution} both processes  $u_1$ and $u_2$  are  $H$-valued  continuous,
we deduce that
\begin{equation*}
\tilde{\mathbb{P}}\left(u_1(t)=u_2(t) \ \text{for all} \ t \in [0,\infty)\right)=1
\end{equation*}
and this concludes the proof.
\end{proof}

The pathwise uniqueness and the existence of martingale solutions
imply the existence of  strong solutions, see e.g. \cite[Theorem 2]{Ondrejat_2004_Uniqueness} and \cite[Theorem 5.3 and Corollary 5.4]{Kunze_2013_Yamada}.
In Section \ref{sec-Yamada-Watanabe Theorem} we have formulated a suitable modification of the above two results.
The following result dealing with a generic time interval $[t_0,\infty)$ is thus a direct consequence of Theorems \ref{thm-Yamada-Watanabe}, \ref{thm-existence} and \ref{thm-pathwise uniqueness}.

Before  we formulate this result it convenient to introduce additional notation analogous  to \eqref{eqn-Y^theta_T}, i.e.
	\begin{equation}
Y^\vartheta_{[t_0,\infty)}=L^p_{\mathrm{loc}}([t_0,\infty);\D(A_q^{\frac{\vartheta }2-\frac{2}{3p}})) \cap C([t_0,\infty);\D(A^{\frac \vartheta  2 }))
\end{equation}
and, for $T>0$, 	\begin{equation}
Y^\vartheta_{[t_0,T]}=L^p([t_0,T];\D(A_q^{\frac{\vartheta }2-\frac{2}{3p}})) \cap C([t_0,T];\D(A^{\frac \vartheta  2 })).
\end{equation}

\begin{theorem}\label{thm-strong existence}
Let Assumptions \ref{assumption-A-space}, \ref{assumption-F_def} and \ref{assumption-stochastic} hold. Fix  $\newr\in [2,\infty)$.
Assume that   $t_0\in  [0,\infty)$ and  $u_{t_0}$  is an ${\F}_{t_0}$-measurable  Borel $\EA$-valued random variable  with finite $r(\alpha+1)$-th moment. Then the following assertions are satisfied.
\begin{trivlist}
\item[(1)]  There exists a unique strong solution   $u=(u(t): t\in[ t_0, \infty))$
 to equation \eqref{eqn-ProblemStratonovich}
such that
 \begin{equation}\label{eqn-stimaEnergy-full-SS}
		                 \E \left[  \sup_{t \in [t_0,T]} \|u(t)\|_H^{2r}+\sup_{t \in [t_0,T]} \energy(u(t)) ^{\newr}\right] <\infty, \mbox{ for every $T\geq t_0$},
		                  \end{equation}
(and hence)
  \begin{equation}\label{add-regularity-t0-T}
     \E \left[  \sup_{t \in [t_0,T]} \|u(t)\|_\EA^{2r}\right]<\infty,  \mbox{ for every $T\geq t_0$}.
 \end{equation}
\item[(2)]
If  $\vartheta$ is as in Proposition \ref{prop-mild solution-new}, then this solution has $\mathbb{P}$-a.s. paths in $C([t_0, \infty);H \cap \D(A^{\frac{\vartheta}{2}})) \cap C_w([t_0, \infty);V) 
\cap Y^\vartheta _{[t_0,\infty)}$ and
for every $T\geq t_0$,
 \begin{align}\label{eqn-propertySolution-SS}
	u\in L^{2r/\alpha}( \Omega;Y^\vartheta _{[t_0,T]})
	\end{align}
\item[(3)]
Moreover, for every $t\in [t_0,\infty)$
 the following equality  in $\EAdual$
	\begin{align}\label{eqn-ItoFormSolution-SS}
	u(t)=  u_{t_0}- \int_{t_0}^t \left[\im A u(s)+\im F(u(s))+\beta  u(s) -b(u(s))\right] \df s     \nonumber
	\\
	- \im \int_{t_0}^t B u(s) \df W(s) -\im \int_{t_0}^t G(u(s)) \,\df \newW(s)
	\end{align}
	holds $\mathbb{P}$-almost surely.
\end{trivlist}
\end{theorem}

\begin{remark}
\label{FR2bis} When one works under Assumptions \ref{assumption-A-space}(ii) or (iii), the regularity assumptions on the domain $\mathscr{O}$ that ensure the uniqueness of the solution (and, consequently, the existence of an invariant measure too) are as follows. One needs to require that $\mathscr{O}$ has $C^{\infty}$ (smooth) boundary and  is relatively compact. The relative compactness of the domain is required to apply the needed Strichartz estimates, whereas the smoothness of the boundary is needed both for the Strichartz estimates to hold, and for the definition of the fractional Sobolev spaces (for this point see also Remark \ref{remreg}.) Notice that, to infer just the existence of a martingale solution, less regularity on the domain is required, see Remark \ref{FR2}.
\end{remark}

\section{Sequential weak  Feller property}\label{sec-weakCont}

Consider the
 family of operators  $\{P_t\}_{t\ge 0}$ defined  in \eqref{def-P_t}.
Our aim is  now to prove the sequential weak  Feller property in $V$ at any fixed time $t$ and that
 $\{P_t\}_{t\ge 0}$ is a Markov semigroup.
 These are part of  the ingredients to prove existence of invariant measures. The
sequential weak  Feller property
 relies on an argument of continuous dependence of the solution on the initial data.
 The Markov property depends on the pathwise uniqueness.

Let us now recall the following fundamental definition. A function
$\phi:\EA\to\mathbb R$ is sequentially continuous w.r.t. the weak
topology  on $\EA$ (we write $\phi \in SC(\EA_w)$) if  $\phi(x_k)\to \phi(x)$ when $x_n\rightharpoonup x \text{ in } \EA$.
Using  the subindex  $SC_b(\EA_w)$ we add the property of boundedness.
We recall that the following inclusions hold
\[
  C_b(\EA_w) \subset SC_b(\EA_w) \subset    C_b(\EA_n) .
\]
Here  $\EA_w$  denotes $\EA$ equipped with the weak topology  and  $\EA_n$  denotes $\EA$ equipped with
the strong (norm)  topology.

Let us also notice that because $\EA$ is a separable space,
the weak Borel and the (strong) Borel $\sigma$-fields on it are equal, i.e.
$\mathscr{B}(\EA_n)=\mathscr{B}(\EA_w)$, see, e.g., \cite{Edgar_1977}.

For $x \in V$, by $u(\cdot;x)=\bigl\{ u(t;x): t\geq 0\bigr\}$ we denote the unique strong solution
 with the deterministic initial condition $x$, defined on   the probability space $(\Omega,\F,\Prob)$.
Bearing in mind the  Remark \ref{rem_new_id}(ii) the unique strong solution of Problem \ref{eqn-ProblemStratonovich} with deterministic initial data
 enjoyes  property \eqref{add-regularity-t0-T} for every finite $r \ge 1$.

It is known, see  \cite[Corollary 23]{On2005},
that the transition function is jointly measurable, that is  for any Borel subset $\Gamma$  of $\EA$
the map \[V \times [0,\infty)\ni  (x,t)\mapsto \Prob\{u(t;x)\in\Gamma\}\in \mathbb{R}\] is measurable.

 We define the  family  of operators $P_t$, for any $t>0$,
\begin{equation}\label{eqn-Markov}
(P_t\phi)(x)= {\mathbb E}[\phi( u(t;x))], \;\; x\in V.
\end{equation}
If $\phi:V\to \mathbb R$ is a  bounded and Borel measurable  function,
then the same holds for $P_t\phi$.

We first provide the following result of continuous dependence on the initial data.
For $0s\ge 0$ let   $u(t;s,x) $, $t \geq s$, denote the solution of equation \eqref{eqn-ProblemStratonovich}    when the initial value  at time $s$ is $x$.
According to the previous notation  used so far we have $u(t;x)=u(t;0,x)$, $t\geq 0$.

\begin{theorem}\label{gener-dep-x}
Let Assumptions \ref{assumption-A-space}, \ref{assumption-F_def} and \ref{assumption-stochastic} hold.
Assume $\newr\in [2,\infty)$ and  $t_0\in [0,\infty)$.
Let   $(x_k)_k$ and $x $ be   $\F_{t_0}$-measurable  Borel $\EA$-valued random variables with finite $r(\alpha+1)$-th moments. If
\begin{equation}\label{stime-unif-k-t0}
\sup_k \mathbb E\|x_k\|_\EA^{r(\alpha+1)}<\infty,\qquad
 \mathbb E\|x\|_\EA^{r(\alpha+1)}<\infty
\end{equation}
and
$\mathbb{P}$-a.s. $x_k$ weakly converges to $x$ in $V$, then
\begin{equation}\label{mean-convergence}
\lim_{k\to \infty}\E \phi(u(t;t_0,x_k))=\E \phi(u(t;t_0,x))
\end{equation}
for any  $t \in (t_0,\infty)$ and any $\phi\in SC_b(\EA_w)$.
\end{theorem}
\begin{proof} Let us choose  and fix $ t>t_0 \geq 0$ and $\phi\in SC_b(\EA_w)$.
By  Theorem \ref{thm-strong existence} we know  that for each initial data there exists a unique solution
to equation \eqref{eqn-ProblemStratonovich}. Moreover  we obtain the uniform estimate
\begin{equation}\label{stima-unif-su u-con-dati-x_k}
\sup_{k\in\N} \E \left[\Vert u(\cdot;t_0,x_k)\Vert_{L^\infty(t_0,T;\EA)}^2\right] <\infty, \;\; \mbox{ for every } T\geq t_0.
\end{equation}
and  the Aldous condition as in
inequality \eqref{eqn-5.22} in Corollary \ref{cor-aprioriEA} and Proposition \ref{EstimatesGalerkinSolution} part (b).
The only difference with respect to the Galerkin approximation sequence  is on the initial data,
but assumption \eqref{stime-unif-k-t0}  is a uniform estimate on them leading to \eqref{stima-unif-su u-con-dati-x_k}.
Therefore we deduce  that the sequence $\bigl(\law{\Prob}{u(\cdot;t_0,x_k)}\bigr)_{k=1}^\infty $
 is tight in the space
$Z_{[t_0,\infty)}= C([t_0,\infty);\EAdual) \cap L^{\alpha+1}_{\rm{loc}}([t_0,\infty);\LalphaPlusEins)\cap C_w([t_0,\infty);\EA)$.
 Hence Corollary \ref{corollaryEstimatesToASconvergence}
applies, i.e.  there exists a subsequence $\{u(\cdot;t_0,x_{n_k})\}_k$ such that  on  a new
 probability space $(\tilde{\Omega},\tilde{\Filtration},\tildeProb)$  there exist
$Z_{[t_0,\infty]}$-valued  random variables
$\{\tilde{u}_k\}_{k \in \mathbb{N}}$ and  $\tilde{u}$   with $\law{\tildeProb}{\tilde{u}_k}=\law{\Prob}{u(\cdot;t_0,x_{n_k})}$ for any $k\in\N$  such that
$\tilde{u}_k \to \tilde{u}$ $\tildeProb$-almost surely in $Z_{[t_0,\infty)}$,   as  $k\to \infty$ and,
by repeating the proof of Theorem \ref{thm-existence} given in subsections \ref{Section_Convergence} and \ref{Section_inequalities},
 the system
$\bigl( \tOmega ,  {\tilde{\mathbb{F}}}, \tp ,  \tu  \bigr)$ is also a martingale  solution with the initial of equation  \eqref{eqn-ProblemStratonovich} with the initial  value $x$ at time $t_0$. In particular,  because $\tilde u_k$ converges to    $\tilde u$ in    $C_w([t_0,\infty);\EA)$,  we deduce that
\begin{equation}
\tilde u_k(t) \to  \tilde u(t)  \mbox{ weakly  in }\EA, \;\;  \tilde \Prob-\mbox{a.s.}.
\end{equation}
 Hence, since function $\phi:\EA \to\mathbb R$ is   bounded and  sequentially  weak
continuous, by the Lebesgue Dominated Convergence Theorem we infer  that
   $\tilde \E [\phi(\tilde u_k(t))] \to \tilde \E [\phi(\tilde u(t))]$. Since $\law{\tildeProb}{\tilde{u}_k}=\law{\Prob}{u(\cdot;t_0,x_{n_k})}$ for any $k\in\N$, we infer that
 \begin{equation}
\lim_{k \to \infty}  \E [\phi( u(\cdot;t_0,x_{n_k}))] = \tilde \E [\phi(\tilde u(t))].
\end{equation}
Since the system
$\bigl( \tOmega ,  {\tilde{\mathbb{F}}}, \tp ,  \tu  \bigr)$ is also a martingale solution of equation  \eqref{eqn-ProblemStratonovich} with the initial  value $x$ at time $t_0$ and since by part (ii) in Theorem \ref{thm-pathwise uniqueness} the solution of \eqref{eqn-ProblemStratonovich} is unique in law, i.e.
\[
  \mbox{the processes $u(\cdot;t_0,x)$ and $\tilde u $ have the same law on the space } Z_{[t_0,\infty)},
\]
we infer that
\begin{equation}  \label{E:bw_Feller_3}
     \tilde{\mathbb{E}}[\phi(\tu (t))] = \mathbb{E}[\phi(u (t))]
\end{equation}
Summing up we proved that
\begin{equation}\label{mean-convergence-2}
\lim_{k\to \infty} \E [\phi( u(\cdot;t_0,x_{n_k}))]=\E \phi(u(t;t_0,x))
\end{equation}
Finally, using the standard sub-subsequence argument, we infer that the whole sequence $\E [\phi( u(\cdot;t_0,x_{k}))]$ is convergent and
\eqref{mean-convergence} holds.
This completes the proof of Theorem \ref{gener-dep-x}.
\end{proof}

\dela{
Since by assumptions $( \Omega , \fcal ,  {\mathbb{F}}, \p , W,u )$ is  a martingale  solution of equation \eqref{E:NS'} with the initial data $u_0$
and $\bigl( \tOmega , \tfcal ,  {\tilde{\mathbb{F}}}, \tp , \tW,\tu  \bigr)$ is also a martingale  solution with the initial of equation \eqref{E:NS'} with the initial data $u_0$
 and since the solution of \eqref{E:NS'} is unique in law, we infer that
\[
  \mbox{the processes $u$ and $\tu $ have the same law on the space } {\mathcal{Z}}_{t}.
\]
Hence
\begin{equation}  \label{E:bw_Feller_3}
     \tilde{\mathbb{E}}[\phi(\tu (t))] = \mathbb{E}[\phi(u (t))] = {P}_{t}\phi ({u}_{0}).
\end{equation}
Thus by \eqref{E:bw_Feller_1}, \eqref{E:bw_Feller_2} and \eqref{E:bw_Feller_3}, we infer that
\[
   \lim_{k\to \infty } {P}_{t}\phi ({u}_{0{n}_{k}}) =  {P}_{t}\phi ({u}_{0}).
\]
Using the sub-subsequence argument, we infer that the whole sequence ${({P}_{t}\phi ({u}_{0n}))}_{n\in \mathbb{N} }$ is convergent and
\[
   \lim_{n\to \infty } {P}_{t}\phi ({u}_{0n}) =  {P}_{t}\phi ({u}_{0}),
\]
which completes the proof of Proposition \ref{prop-Feller_bw}.
}

It easily follows the sequential weak Feller property in $V$, that is
\[P_t:SC_b(\EA_w)\to SC_b(\EA_w), \mbox{ for any } t>0\]

\begin{corollary}\label{Feller}
Let Assumptions \ref{assumption-A-space}, \ref{assumption-F_def} and \ref{assumption-stochastic} hold.
For  any $t>0$, if
$\phi \in SC_b(\EA_w)$ then $P_t\phi \in SC_b(\EA_w)$.
\end{corollary}
\begin{proof} Let us fix any  $t>0$ and $\phi \in SC_b(\EA_w)$.
We have to prove that, given a sequence $\left(x_k \right)_k \subset \EA$ which converges weakly in $\EA$ to
$x$,  the sequence $ P_t\phi(x_k)$ converges to
$P_t\phi(x) $.
\dela{We  fix $T>t$, and we will prove the result for any time  in $ (0,T]$.}

By the weak convergence we get the uniform estimate
\[
\sup_k\|x_k\|_\EA<\infty;
\]
hence the sequence of deterministic initial data fulfills the assumptions of Theorem \ref{gener-dep-x} on the time interval $[0,T]$, i.e. we set $t_0=0$. Therefore \eqref{mean-convergence} holds true;  bearing in mind the definition \eqref{eqn-Markov} of the operator $P_t$ we conclude the proof.
\end{proof}

Now we consider the  Markov property.
\begin{proposition}\label{PROP-Markov}
Let Assumptions \ref{assumption-A-space}, \ref{assumption-F_def} and \ref{assumption-stochastic} hold.
For every $\phi\in SC_b(\EA_w )$,
$x\in \EA$  and $t,s>0$ we have
\begin{equation}\label{Markov-reg}
  {\mathbb E}\left[  \phi\left(  u(t+s;x)\right)  |{\mathcal F}_s\right]
  =
  \left(P_{t}\phi\right)  \left(  u(s;x)\right)  \qquad {\mathbb P}\text{-a.s.}
\end{equation}
\end{proposition}
\begin{proof}
The proof is classical when the solution is a continuous process
taking values in a separable Banach space, endowed with the strong topology;
see, e.g. \cite[Theorem 9.14 ]{DPZ1}.
Hence we highlight only the differences when dealing with the weak topology in $\EA$.

By the pathwise uniqueness we know that for all $t,s>0$
\begin{equation}u(t+s;0, x)=u(t+s;s,u(s;0,x)) \qquad \ a.s.
\label{eqn-question 01}
\end{equation}
Set $\eta=u(s;0,x)$;
the  identity
 \eqref{Markov-reg} can be written as
\begin{equation}\label{markov}
{\mathbb E}\left[  \phi\left(u(t+s;s,\eta)\right)  |{\mathcal F}_{s}\right]
  =
  \left(P_{t}\phi\right)  \left( \eta\right) \qquad  \Prob\text{-a.s.}
\
\end{equation}
We notice that
given any  deterministic initial data $x\in \EA$,  Theorem \ref{mainTh}  gives that  $\mathbb E\|u(s;0,x)\|_V^{2r}<\infty$ for any finite $r\ge 1$.
Let us choose $r=\alpha+1$.

Hence it is enough to show that equality \eqref{markov} holds for an arbitrary
$2(\alpha+1)$-integrable  $\mathcal F_s$-measurable random variable $\eta$.

Let us first suppose that  $\eta$ is a simple random variable of the form
$\sum_{j=1}^N x_j \mathbbm 1_{\Gamma_j}$ with $x_j \in \EA$ and a partition
$\Gamma_1,\ldots,\Gamma_N\subset \mathcal F_s$. Then   \eqref{markov}
is proved as usual  by noticing that $u(t+s;s,\eta)=\sum_{j=1}^N u(t+s;s,x_j)\mathbbm 1_{\Gamma_j}$;
  indeed, $\Prob$-a.s. we have the following relationships
\begin{equation}
\begin{split}
&{\mathbb E}\left[  \phi\left(u(t+s;s,\eta)\right)  |{\mathcal F}_{s}\right]
=\sum_{j=1}^N {\mathbb E}\left[ \mathbbm 1_{\Gamma_j} \phi\left(u(t+s;s,x_j)\right)  |{\mathcal F}_{s}\right]
\\&=\sum_{j=1}^N \mathbbm 1_{\Gamma_j}  {\mathbb E}\left[  \phi\left(u(t+s;s,x_j)\right)  |{\mathcal F}_{s}\right]  =\sum_{j=1}^N \mathbbm 1_{\Gamma_j}  {\mathbb E}\left( \phi(u(t+s;s,x_j))\right)
\\& =\sum_{j=1}^N \mathbbm 1_{\Gamma_j}  {\mathbb E}\left( \phi(u(t;0,x_j))\right)
 =\sum_{j=1}^N \mathbbm 1_{\Gamma_j}  (P_t\phi)(x_j)=(P_t\phi)(\eta).
\end{split}
\end{equation}
We used that $ u(t+s;s,x_j)$  is independent of  ${\mathcal F}_{s}$ and  that $u(t+s;s,x_j)$ and $u(t;0,x_j)$ have the same law.

Otherwise, for general $\eta\in L^{2(\alpha+1)}(\Omega)$
 there exists a sequence of simple random variables $\eta_n$  with $\lim_{n\to\infty}\eta_n=\eta$  in $L^{2(\alpha+1)}(\Omega)$ and $\Prob$-a.s.
(in the  strong topology of $\EA$, hence weak too);
moreover
\begin{equation}\label{stime-eta-n}
\sup_n \mathbb E \|\eta_n\|_\EA^{2(\alpha+1)}\le \mathbb E \|\eta\|_\EA^{2(\alpha+1)}<\infty.
\end{equation}
We checked before that
\begin{equation}\label{markov-step}
\mathbb E \left[  \phi\left(u(t+s;s,\eta_n)\right)  |{\mathcal F}_{s}\right]
  =
  \left(P_{t}\phi\right)  \left( \eta_n\right) \qquad  \mathbb P\text{-a.s.}
\
\end{equation}
Thanks to \eqref{stime-eta-n}  we can  proceed as in  Theorem \ref{gener-dep-x} on the time interval $[s,s+t]$ in order to deal with the conditional expectation and  pass to the limit  in the l.h.s. of \eqref{markov-step}. Thus
we have proved that the l.h.s. of \eqref{markov-step} converges to the l.h.s. of
\eqref{Markov-reg} as $n \to \infty$.

 As far as the convergence of the  r.h.s. of \eqref{markov-step}  is concerned, we
 know from Corollary \ref{Feller} that $P_t\phi\in SC_b(\EA_w)$; since $\Prob$-a.s.
   $\eta_n$ converges to $\eta$ weakly in $V$, we obtain that
   $P_t\phi(\eta_n)\to P_t\phi(\eta)$, $\Prob$-a.s. .
\end{proof}

Taking the mathematical expectation in \eqref{Markov-reg},  we deduce that
 the family $\{P_t\}_{t\ge 0}$ is a Markov semigroup, namely $P_{t+s}=P_t P_s$ for any $s,t>0$.

\section{Existence of an invariant measure}\label{s-inv}

Given the  sequential weak Feller  Markov semigroup on the  separable Hilbert space $V$,
 we can define an invariant measure $\pi$ for equation \eqref{eqn-ProblemStratonovich} as a
Borel probability measure  on $V$ such that for any time $t\ge 0$
\begin{equation}\label{def-invmeas}
\int_{\EA} P_t \phi\ d\pi=\int_{\EA} \phi\ d\pi, \qquad \forall  \phi\in SC_b(V_w).
\end{equation}

Let us recall a result of Maslowski-Seidler \cite{Mas-Seid} about  the
existence of an invariant measure. This is a modification of
the Krylov-Bogoliubov technique, usually  presented in the setting of strong topologies, see, e.g.,  \cite{KB} and \cite{DPZ2}.

\begin{theorem}\label{MSprop}
Assume that\\
i) the semigroup $\{P_t\}_{t\ge 0}$ is sequential weak Feller in $\EA$;\\
ii) for any $\varepsilon >0$ there exists $R_\varepsilon>0$ such that
\[
\sup_{T\ge 1} \frac 1T \int_0^T \mathbb P\left(\|u(t;0)\|_{\EA}>R_\varepsilon\right)dt<\varepsilon .
\]
Then there exists  at least one invariant measure for equation \eqref{eqn-ProblemStratonovich}.
\end{theorem}

Hence we get our main result on invariant measures as defined by \eqref{def-invmeas}.

\begin{theorem}\label{th:mis-inv-2d}
Let Assumptions \ref{assumption-A-space}, \ref{assumption-F_def} and \ref{assumption-stochastic} hold.
 If condition
\eqref{condizione-beta} is fulfilled,
then there exists at least one invariant measure  $\pi$ for
equation \eqref{eqn-ProblemStratonovich} and $\pi(\EA)=1$.
\end{theorem}
\begin{proof}  The proof is based on Theorem \ref{MSprop}.
The sequential weak Feller property has been proved before in Corollary \ref{Feller}.
For the tightness
it is enough to recall \eqref{Ebound} and the Chebyshev inequality, so
\[
\mathbb P\left(\|u(t;0)\|_{\EA}>R \right)\le \frac{\E \|u(t;0)\|_\EA^2}{R^2}\le \frac C{R^2} \qquad \text{ for all }t\ge 0
\]
where the constant $C$ is independent of $t$.  Hence we have verified that the two assumptions of Theorem \ref{MSprop} are fulfilled.
\end{proof}

\section{Existence and uniqueness of the invariant measure with purely multiplicative noise}\label{sect-unique}

Assume that the coefficients characterizing the operator $G$  are such that $C_1=0$; this implies that $G(0)=0$, see \eqref{eqn-G crescita}.
Hence the zero process is a solution of equation \eqref{eqn-ProblemStratonovich}, or equivalently $\delta_0$ is an invariant measure.
Let us prove that this is the unique invariant measure if
\[
\beta>\tfrac 12 \tilde C_1^2.
\]
This is our result
\begin{theorem}
\label{cordelta0}
Let Assumptions \ref{assumption-A-space}, \ref{assumption-F_def}  hold and
Assumption \ref{assumption-stochastic} holds with $C_1=0$, that is
\begin{equation}\label{crescitaG-0}
\|G(u)\|_{\gamma(\Yg,H)} \le \tilde C_1\|u\|_H\qquad \forall u \in H.
\end{equation}
If
\begin{equation}
\beta>\tfrac 12 \tilde C_1^2,
\end{equation}
then  there exists a unique invariant measure for equation \eqref{eqn-ProblemStratonovich} given by $\pi=\delta_0$.
\end{theorem}

The proof is based on an auxiliary result
\begin{lemma}
\label{lem1_delta0}
Under the assumptions of Theorem \ref{cordelta0},
 there exists a constant $\lambda>0$ such that, if $u$ is a solution to \eqref{eqn-ProblemStratonovich}, 
 then the process
 $\{e^{\lambda t}\|u(t)\|^2_H\}_{t \ge 0}$ is a non-negative continuous  supermartingale.
\end{lemma}
\begin{proof} Let $u$ be the unique solution to equation \eqref{eqn-ProblemStratonovich} starting from $u_0 \in V$.
We apply the It\^o formula to the process $g(r, u(r))$ with $g(r,x):=e^{\lambda r}\|x\|^2_H$, for $r \in [s,t]$, since we know that the paths are in $  C([0,T];H)$.
We have
\begin{equation*}
{\rm d}\left(e^{\lambda t}\|u(t)\|_H^2\right)= \lambda e^{\lambda t} \|u(t)\|_H^2+e^{\lambda t}{\rm d}\|u(t)\|_H^2.
\end{equation*}
By the same computations done in the proof of Proposition \ref{MassEstimateGalerkinSolution}
we get
\[
d\|u(t)\|_H^2+2\beta\|u(t)\|_H^2 dt = \|G(u(t))\|^2_{\gamma(Y_2,H)} dt + 2 \Real \skpH{u(t)}{ -\im  G(u(t))  \df \textbf{W}(t)}.
\]
Hence, using \eqref{crescitaG-0}
\[
{\rm d}\left(e^{\lambda t}\|u(t)\|_H^2\right)
\le
 (\lambda -2\beta+\tilde C_1 ^2) e^{\lambda t} \|u(t)\|_H^2+ 2 \Real \skpH{u(t)}{ -\im  G(u(t))  \df \textbf{W}(t)}.
 \]
Taking the conditional expected value on both sides, we get
\begin{equation*}
\frac{{\rm d}}{{\rm d}t}\mathbb{E}[e^{\lambda t}  \|u(t)\|_H^2 |\mathscr{F}_s]
\le (\lambda -2\beta+\tilde C_1 ^2)\mathbb{E}[e^{\lambda t} \|u(t)\|_H^2|\mathscr{F}_s], \quad s <t,
\end{equation*}
so that
\begin{equation*}
\mathbb{E}[e^{\lambda t}  \|u(t)\|_H^2 |\mathscr{F}_s]
\le
e^{(\lambda-2\beta+\tilde C_1 ^2)(t-s)}e^{\lambda s} \|u(s)\|_H^2 ,\quad s <t.
\end{equation*}
If we now choose $\lambda>0$ such that $\lambda-2\beta+\tilde C_1 ^2<0$ we obtain
\begin{equation}
\label{eq4.}
\mathbb{E}[e^{\lambda t}\|u(t)\|_H^2 |\mathscr{F}_s] \le e^{\lambda s} \|u(s)\|_H^2  \qquad \forall  \ s<t.
\end{equation}
\end{proof}

 Let us now prove Theorem \ref{cordelta0}.
\begin{proof}[Proof of Theorem \ref{cordelta0}]
 For the (unique) solution of problem \eqref{eqn-ProblemStratonovich}, we put in evidence the initial datum $u_0 \in \EA$
by writing $u(\cdot;u_0)$.
Lemma \ref{lem1_delta0} yields
\begin{equation*}
\label{es1_delta0}
\mathbb{E} \|u(t;u_0)\|^2_H \le e^{-\lambda t} \|u_0\|^2_H, \qquad \forall \ t \ge 0.
\end{equation*}
Proceeding as in  \cite[Proof of Theorem 1.4]{BrzMasSei} one then shows by means of the Borel-Cantelli Lemma
that, for every $\bar \lambda \in (0, \lambda)$,
there exists a $\mathbb{P}$-a.s. finite function $t_0: \Omega \rightarrow [0, \infty]$ such that
\begin{equation*}
 \|u(t;u_0)\|^2_H \le e^{-\bar \lambda t} \|u_0\|^2_H, \qquad \forall \ t \ge t_0, \quad \mathbb{P}-a.s.
\end{equation*}
Hence
 \begin{equation*}
\|u(t;u_0)\|_H \rightarrow 0, \quad \text{as} \ t \rightarrow \infty, \quad \mathbb{P}-a.s.
\end{equation*}

Now take any function $\phi:V\to\mathbb R$ which is continuous with respect to the $H$-norm; write
$\phi \in C(V_H)$.  By the above we have
\[
\lim_{t\to+\infty} \phi(u(t;u_0))=\phi(0)
\]
for any initial data $u_0$. Moreover this function $\phi$ belongs to $SC_b(V_w)$, because
 the embedding
$V\subset H$ is compact, so that any sequence weakly convergent in $V$ is strongly convergent in $H$.

Now let $\pi$ be any invariant measure. Then we have
\[
\int_V P_t \phi\ d\pi=\int_V \phi\  d\pi \qquad \forall \phi \in SC_b(V_w), t\ge 0.
\]
Taking $\phi  \in C_b(V_H)$, by the Dominated Convergence Theorem the left hand side converges to $\phi(0)$
as $t \to +\infty$. This implies that
\begin{equation}\label{pi-inv-0}
\phi(0)=\int_V \phi\  d\pi \qquad \forall \phi \in C_b(V_H).
\end{equation}
We can conclude that $\pi=\delta_0$ if this equality holds for any  $\phi \in SC_b(V_w)$.
In fact, take  
$\phi(u)=e^{i\langle u, h\rangle}$, $h\in \EAdual$, so to get that the integral defines the characteristic function and this is enough to determine the measure.

Now we show by approximation that \eqref{pi-inv-0} holds for any $\phi \in SC_b(V_w)$.
Given $u\in V$ and $\epsilon>0$  define $u_\epsilon=(I+\epsilon A)^{-1}u$ so that
 \begin{equation}\label{pointwise}
\lim_{\epsilon \to 0}\|u_\epsilon-u\|_V=0
\end{equation}
and
\[
\|u_\epsilon\|_V \le C(\epsilon) \|u\|_H
\]
where the constant $C(\epsilon)$ is not bounded as $\epsilon \to 0$.\\
Now take any  $\phi \in SC_b(V_w)\subset C_b(V)$ and define
\[\phi_\epsilon(u)=\phi((I+\epsilon A)^{-1}u).\]
It is clear $\phi_\epsilon(u)\to \phi(u)$ for any $u \in V$.
Moreover  from the previous arguments we have that $\phi_\epsilon\in C_b(V_H)$ and therefore we can write the identity
\[
\int_V \phi_\epsilon\  d\pi =\phi(0).
\]
Now passing in the limit as $\epsilon \to 0$ in the l.h.s., by means of the  dominated convergence theorem
 and \eqref{pointwise} we get
 \begin{equation}
\int_V \phi\  d\pi =\phi(0)\qquad \forall \phi \in SC_b(V_w).
\end{equation}
\end{proof}

\appendix

\section{Laplacian-type operators on manifolds and on bounded domains with Dirichlet/Neumann boundary conditions and Strichartz estimates}
\label{app_Lapl}

In Section \ref{sec-UniquenessSection} we need some results about Sobolev spaces on two-dimensional manifolds and on
bounded domains of $\mathbb{R}^2$ with either Dirichlet or Neumann boundary conditions. We collect them here.
Then we derive the Strichartz estimates employed in Section \ref{sec-UniquenessSection}.

\subsection{Dirichlet and Neumann Laplacians on bounded domains and Sobolev spaces}
\label{LDNappendix}

The present Section is devoted to recall some basic facts about Sobolev spaces on bounded domains of $\mathbb{R}^2$ and their connection with the fractional domains of the realization of the Laplace operator with Dirichlet and Neumann boundary conditions on $L^q$ spaces, $q \in (1, \infty)$. We recall also some Sobolev embedding theorems.

Let $\mathscr{O}$ be a bounded smooth domain of
$\mathbb{R}^2$.
For any $s \in \mathbb{R}$ and $q \in (1, \infty)$, the Sobolev space $H^{s,q}(\mathscr{O})$ is defined as the restriction of
$H^{s,q}(\mathbb{R}^2)$, see \cite[Definition 2.3.1]{Triebel},  to $\mathscr{O}$, see  \cite[Definition 4.2.1(1).]{Triebel}. For
$q=2$ we write $H^s(\mathscr{O}):=H^{s,2}(\mathscr{O})$. When $s$ is a natural number the space $H^{s,q}(\mathscr{O})$
coincides with the Sobolev space $W^{s,q}(\mathscr{O})$, see \cite[Remark 2.3.1($2^*$)]{Triebel}.
We denote by $H^{s,q}_0(\mathscr{O})$ the completition of $C^{\infty}_0(\mathscr{O})$ (set of smooth functions defined
over $\mathscr{O}$ with compact support) in $H^{s,q}(\mathscr{O})$, see \cite[Definition 4.2.1(2)]{Triebel}.

In the following Proposition we list some embedding properties of the Sobolev spaces.
\begin{proposition}
\label{B1}
Let $\mathscr{O}$ be a bounded smooth domain of $\mathbb{R}^2$, then
\begin{enumerate}[label=\roman{*}), ref=(\roman{*})]
\item for $2 \le q < \infty$, the embedding $H^{1}(\mathscr{O})\subset L^q(\mathscr{O}),$ is continuous and compact.
\item for $1<q<\infty$ and $s>\frac 2 q$, $H^{s,q}(\mathscr{O}) \subset L^{\infty}(\mathscr{O})$,
\item for $1 < p \le q <\infty$, $0<t<s<1$ and $s-\frac 2p\ge t-\frac 2q$, $H^{s,p}(\mathscr{O}) \subset H^{t,q}(\mathscr{O})$.
\end{enumerate}
\end{proposition}
\begin{proof}
\begin{enumerate}[label=\roman{*}), ref=(\roman{*})]
\item See \cite[Theorem 11.23 and Exercise 11.26]{Leoni}.
\item See \cite[Theorem 4.6.1.(e)]{Triebel}.
\item See \cite[Remark 2.8.1(2) and Theorem 4.2.2.(1)]{Triebel}.
\end{enumerate}
\end{proof}

\begin{remark}
\label{remreg}
Since we always consider the case $|s|<2$, where $2$ is the dimension of the space, it would be enough to assume
$\mathscr{O}$ bounded and $C^2$: see \cite[Remark 4.2.2 (2)]{Triebel} for the relation between the regularity assumptions
on the domain and the range of the exponent $s$. Smoothness of the boundary is in any case necessary for the Strichartz
estimates we consider, to hold.
\end{remark}

Let us now turn to the characterization of the domains of the Dirichlet and Neumann Laplacian.
Let $-A_D$ and $-A_N$ be, respectively,  the realization of the Laplace operator in $L^2(\mathscr{O})$ with zero Dirichlet
and zero Neumann boundary conditions, with domains
\begin{equation*}
\D(A_D)=\{f \in H^{2}(\mathscr{O}) \ : \gamma_{|\partial \mathscr{O}}f=0\},
\end{equation*}
\begin{equation*}
\D(A_N)=\{f \in H^{2}(\mathscr{O}) \ : \gamma_{|\partial \mathscr{O}}\partial_{\nu}f=0\},
\end{equation*}
where by $\gamma_{|\partial \mathscr{O}}$ we denote the trace operator and by $\nu$ the outward normal unit vector to
$\partial \mathscr{O}$.
It is well known, see e.g. \cite{Temam_IDD}, that both the Dirichlet and the Neumann Laplacian are self-adjoint positive
operators on $L^2(\mathscr{O})$. By means of the functional calculus for self-adjoint operators, see e.g. \cite{Zeidler},
the powers $A^s_D$ and $A^s_N$ of the operators $A_D$ and $A_N$, for every $s \in \mathbb{R}$, are then well defined
and self-adjoint. Thus one can introduce the spaces $\D(A_D^{\frac s2})$ and $\D(A_N^{\frac s2})$, for every
$s \in \mathbb{R}$ in accordance with the spectral theorem.

To derive the needed Strichartz estimates we have to consider the realizations of Dirichlet and Neumann Laplacian on Banach spaces $L^q(\mathscr{O})$, $q \in (1, \infty)$. For this part we mainly refer to \cite{Grubb} and to therein references.
The domains of the realizations of the Dirichlet and Neumann Laplacian in $L^q(\mathscr{O})$, denoted hereafter by $A_{D_q}$ and $A_{N_q}$ respectively, are
\begin{equation*}
\D(A_{D_q})=\{f \in H^{2,q}(\mathscr{O}) \ : \gamma_{|\partial \mathscr{O}}f=0\},
\end{equation*}
\begin{equation*}
\D(A_{N_q})=\{f \in H^{2,q}(\mathscr{O}) \ : \gamma_{|\partial \mathscr{O}}\partial_{\nu}f=0\}.
\end{equation*}

\begin{definition}
\label{def-sobolev spaces with boundary conditions}
Let $s \in (0,2)$ and $q \in (1, \infty).$ Define the spaces
\begin{equation*}
H^{s,q}_D(\mathscr{O}):=\{ f \in H^{s,q}(\mathscr{O}): \gamma_{|\partial \mathscr{O}}f=0 \ \text{if} \ s >\frac 1q \},
\end{equation*}
\begin{equation*}
H^{s,q}_N(\mathscr{O})=\{ f \in H^{s,q}(\mathscr{O}): \gamma_{|\partial \mathscr{O}}\partial_{\nu}f=0 \ \text{if} \ s >1+\frac 1q \}.
\end{equation*}
\end{definition}

\begin{theorem}
\label{dom_D_N}
Let $q \in (1, \infty)$, then
\begin{enumerate}[label=\roman{*}), ref=(\roman{*})]
\item
for $s \in (0,2) \setminus \{\frac 1q\}$, $H_D^{s,q}(\mathscr{O})=\D((A_{D_q})^{\frac s2})$,
\item for $s \in (0,2) \setminus \{1+\frac 1q\}$, $H_N^{s,q}(\mathscr{O})=\D((A_{N_q})^{\frac s2})$.
\end{enumerate}
\end{theorem}
\begin{proof}
See \cite[Theorem 2.2]{Grubb}.
\end{proof}

\begin{remark}
\label{B4}
Theorem \ref{dom_D_N} and \cite[Theorem 4.3.2 (1)]{Triebel}
 yield, in particular,
 \begin{equation}
\label{realization_sD}
\D(A_D^{\frac s2})=
\begin{cases}
H^{s} \ \text{for} \ s \in \left(0,\frac 12\right),
\\
H^{s}_0\ \text{for} \ s \in \left(\frac 12, 1\right].
\end{cases}
\end{equation}
and
\begin{equation}
\label{realization_sN}
\D(A_N^{\frac s2})=H^{s}, \qquad \text{for} \ s \in (0,1].
\end{equation}
\end{remark}

\subsection{Laplace-Beltrami operators on compact Riemannian manifolds and Sobolev spaces}
\label{LBappendix}

In the present Section we recall some results about Sobolev spaces on manifolds and their connection with the fractional domains of the Laplace-Beltrami operator.

We consider $(M,g)$, a compact
 Riemannian manifold without boundary of dimension two.
 By $-A=\Delta_g$ we denote the Laplace-Beltrami operator on $L^2(M)$. Theorem 3.5 in \cite{Strichartz} states that the
 restriction of $\left(e^{-tA}\right)_{t\ge 0}$
to $L^2(M)\cap L^q(M)$ extends to a strongly continuous semigroup on $L^q(M)$, $q \in [1, \infty)$. The infinitesimal
generator of such a semigroup, denoted by $-A_{g,q}=\Delta_{g,q}$, is called the Laplace-Beltrami operator on $L^q(M)$.
 With this extended semigroup one can define the fractional powers of the operator $-A_{g,q}$. For our needs it is  sufficient
 to recall the characterization of the fractional domains of the Laplace-Beltrami operator on $L^q(M)$, in terms of Sobolev
 spaces.

\begin{proposition}
\label{propB6}
Let $(M,g)$ be a compact Riemannian manifold without boundary of dimension two.
Let $s \ge 0$ and $q \in (1, \infty)$. The fractional Sobolev space $H^{s,q}(M)$ defined as
\begin{align*}
		H^{s,q}(M):=\left\{f\in L^q(M): \norm{f}_{H^{s,q}(M)}:=\left(\sum_{i\in I} \norm{(\Psi_i f)\circ \kappa_i^{-1}}_{H^{s,q}(\mathbb{R}^d)}^q\right)^\frac{1}{q}<\infty\right\},
		\end{align*}
where $\mathscr{A}:=\left(U_i,\kappa_i\right)_{i\in I}$ is an atlas of $M$ and $\left(\Psi_i\right)_{i\in I}$ a partition of unity
subordinate to $\mathscr{A}$, has the following properties:
\begin{enumerate}[label=\roman{*}), ref=(\roman{*})]
\item $H^{s,q} = \D((-\Delta_{g,q})^{\frac s2})$.
\item for $s>\frac{2}{q},$ we have $H^{s,q}(M)\hookrightarrow L^{\infty}(M)$,
		\item let $s\ge 0$ and $q\in (1,\infty).$ Suppose $q\in [2,\frac{2}{(1-s)_+})$ or $q=\frac{2}{1-s}$ if $s<1$.
		Then, the embedding $H^{s}(M)\hookrightarrow L^{q}(M)$ is continuous.
		\newline
		If $0<s\le 1$
		as well as $q\in [1,\frac{2}{(1-s)_+})$, the embedding $H^{s}(M)\hookrightarrow L^{q}(M)$ is compact.
		\item For $s,s_0,s_1\ge 0$ and $p,p_0,p_1\in(1,\infty)$ and $\theta\in(0,1)$ with
		\begin{align*}
		s=(1-\theta)s_0+\theta s_1,\qquad \frac{1}{p}=\frac{1-\theta}{p_0}+\frac{\theta}{p_1},
		\end{align*}
		we have
		$\left[H^{s_0,p_0}(M),H^{s_1,p_1}(M)\right]_\theta=H^{s,p}(M)$.
\end{enumerate}
\end{proposition}
\begin{proof}
Statement (i) follows from \cite[Chapter 7]{Triebel_func_spacII} and the results of \cite{Strichartz}. For the other statements we refer to \cite[Proposition B.2]{Brz+H+W-2019}.
\end{proof}

\begin{remark}
\begin{enumerate}[label=\roman{*}), ref=(\roman{*})]
\item It is known, see \cite{Triebel_func_spacII}, that for $k \in \mathbb{N}_0$ and $q\in[1,\infty),$
$H^{k,q}(M)=W^{k,q}(M)$ , where $W^{k,q}(M)$ is the classical Sobolev space defined via covariant derivatives.
\item For $q=2,$ we write $H^s(M):=H^{s,2}(M)$.
\end{enumerate}
\end{remark}

\subsection{Strichartz estimates}
\label{Str_est_sec}

In this Section we derive the Strichartz estimates that we need for the proof of uniqueness in Section \ref{sec-UniquenessSection}.

Throughout this Section the operator $A$ can be either the Laplace-Beltrami operator $-\Delta_g$ on a two-dimensional
compact Riemannian manifold $(M,g)$ without boundary, equipped with a Lipschitz metric $g$, or the  negative Laplace
operator with Dirichlet or Neumann boundary conditions on a smooth relatively compact  domain
 $\mathscr{O} \subset \mathbb{R}^2$.

By $A_q$ we mean the realization of the above mentioned operators on the  $L^q$ space, see Sections \ref{LDNappendix}
and \ref{LBappendix}. As usual, if not specified, by $L^q$ we mean either $L^q(M)$ or  $L^q(\mathscr{O})$ and, for
simplicity we write $A$ instead of $A_2$.

When the operator $A$ is of the type described above, for every $s \ge 0$ and $q \in (1,\infty)$, $(Id+A_q)^{-s}$ defines an
isomorphism from $L^q$ to $\D(A^{s}_{q})$ and it holds that
\begin{equation}
\label{equivalence_norm_A}
\norm{f}_{\D(A^s_q)} \simeq \norm{v}_{L^q}, \quad \text{ for}  	\ f=(I+A_q)^{-s} v.
\end{equation}
In the next Lemma we recall the deterministic homogeneous Strichartz estimate from a recent paper by Blair, Smith and Sogge, see \cite[Theorem 1.1]{Blair+Sogge_2008}, stated here in the form more suitable for our needs.

\begin{lemma}\label{thm-BSS2008}
Let $-A$ be either the Laplace-Beltrami operator on a two-dimensional compact Riemannian manifold $(M,g)$ without
boundary equipped with
a Lipschitz metric $\mathrm{g}$, or the realization of the negative Laplace operator with Dirichlet or Neumann boundary
conditions on a smooth relatively compact  domain $\mathscr{O} \subset \mathbb{R}^2$.
Assume that $(p,q)$ is a Strichartz pair of real numbers, i.e.
$2 \le p,q \le \infty$ and
\begin{align}\label{eqn-Strichartz pair}
\frac{2}{p} + \frac{2}{q} = 1, & & (p,q)\ne(2,\infty).
\end{align}
Then  the following Strichartz estimate holds for every $x \in \D(A^{\frac{2}{3p}})$
\begin{equation}\label{eqn-Strichartz}
\|e^{-i\cdot A}x \|_{L^p([0,T]; L^q)} \lesssim_T
\|x\|_{\D(A^{\frac{2}{3p}})}
\end{equation}
\end{lemma}
Let us notice that when $p=\infty$, then $q=2$ and the inequality  \eqref{eqn-Strichartz} becomes the classical one
\begin{equation}\label{eqn-classical}
\|e^{-i\cdot A}x\|_{L^\infty([0,T]; L^2)} \lesssim_T\|x\|_{L^2}.
\end{equation}

\begin{remark}\label{rem-Strichartz comparison}
In the case where $(M,\mathrm{g})$ is a  boundaryless manifold with $\mathrm{g} \in
C^\infty$, the estimate \eqref{eqn-Strichartz} holds with $s=\frac{1}{2p}$ instead of $s=\frac{2}{3p}$, see  the paper
\cite{Burq+G+T_2004} by Burq, G\'{e}rard and Tzvetkov.  In particular,  the Strichartz estimates for  a  boundaryless manifold  are stronger
than the Strichartz estimates for a bounded domain with smooth boundary.
\end{remark}

From Lemma \ref{thm-BSS2008} we can deduce the following Strichartz estimates for the deterministic and stochastic
convolutions.
\begin{lemma}
\label{Strichartz_lemma}
Assume that $T>0$.
In the situation of Lemma \ref{thm-BSS2008}, we take $\vartheta  \in \left[\frac{4}{3p},1\right]$ and $r \in (1,\infty)$.
\begin{itemize}
\item[i)] We have the homogeneous Strichartz estimate
\begin{equation}\label{eqn-Strichartz-211}
\|e^{-i \cdot A}x\|_{L^p(0,T; \D(A_q^{\frac{\vartheta }2-\frac{2}{3p}}))} \lesssim_T
\|x\|_{\D(A^{\frac{\vartheta }2})}, \qquad \text{for $x\in \D(A^{\frac{\vartheta }2})$},
\end{equation}
 and the inhomogeneous Strichartz estimate
\begin{equation}
\label{eqn-Strichartz-212}
\left \Vert\int_0^{\cdot} e^{-i(\cdot- \tau)A}f(\tau)\, {\rm d}\tau\right\Vert_{L^p(0,T; \D(A_q^{\frac{\vartheta }2-\frac{2}{3p}}))} \lesssim_T \|f\|_{L^1(0,T,\D(A^{\frac{\vartheta }2}))},
\end{equation}
for $f\in L^1(0,T;\D(A^{\frac\vartheta 2}))$.
\item[ii)] Let $(\Omega,\F,\Prob)$ be a probability space, $Y$ be a separable real Hilbert space, $W$ a $Y$-canonical cylindrical Wiener processes adapted to a filtration $\Filtration$ satisfying the usual conditions. We have the stochastic Strichartz estimate
\begin{equation}
\label{eqn-Strichartz-213}
\left \Vert\int_0^{\cdot} e^{-i(\cdot- \tau)A}\xi(\tau)\, {\rm dW }(\tau)\right\Vert_{L^r(\Omega,L^p(0,T; \D(A_q^{\frac{\vartheta }2-\frac{2}{3p}})))}
\lesssim_T \|\xi\|_{L^r(\Omega;L^2(0,T;\gamma(Y,\D(A^{\frac{\vartheta }{2}}))))}
\end{equation}
for all adapted processes $\xi \in L^r(\Omega;L^2(0,T;\gamma(Y,\D(A^{\frac{\vartheta }{2}}))))$.
\end{itemize}
\end{lemma}
\begin{proof}
\begin{itemize}
\item [i)]
Estimate \eqref{eqn-Strichartz-211} follows from \eqref{equivalence_norm_A} and \eqref{eqn-Strichartz} that yield
\begin{align}
\label{Str1}
\|e^{-i\cdot A}x\|_{L^p(0,T; \D(A_q^{\frac{\vartheta }2-\frac{2}{3p}}))}
&\simeq
\|(Id+A_q)^{\frac{\vartheta }2-\frac{2}{3p}}e^{-i\cdot A}x\|_{L^p(0,T;L^q)}
\notag\\
&\simeq
\|e^{-i\cdot A}(Id + A_q)^{\frac{\vartheta }2-\frac{2}{3p}}x\|_{L^p(0,T;L^q)}
\notag\\
&\le
\|(Id+A_q)^{\frac{\vartheta }2-\frac{2}{3p}}x\|_{\D(A^{\frac{2}{3p}})}
\simeq \|x\|_{\D(A^{\frac{\vartheta }{2}})}.
\end{align}
\\
The proof of estimates \eqref{eqn-Strichartz-212} follows the lines of the proof of \cite[Corollary 2.1]{Burq+G+T_2004}, see also the proof of \cite[Lemma 3.2]{BrzezniakStrichartz}.
The l.h.s. in \eqref{eqn-Strichartz-212} reads
\begin{align*}
I&:= \left \Vert \int_0^TF_{\tau}\, {\rm d}\tau\right\Vert_{L^p(0,T; \D(A_q^{\frac{\vartheta }2-\frac{2}{3p}}))} ,
\\
 F_{\tau}(t)&:={\pmb 1}_{[\tau,T]}(t)e^{-i(t-\tau)A}f(\tau), \; t\in [0,T].
\end{align*}
Let us observe that by estimate \eqref{eqn-Strichartz-211} (with $C_T$ being the constant)
\begin{align*}
\|F_\tau\|^p_{L^p(0,T; \D(A_q^{\frac{\vartheta }2-\frac{2}{3p}}))}&= \int_0^T  \|F_\tau(t)\|^p_{ \D(A_q^{\frac{\vartheta }2-\frac{2}{3p}})} \, {\rm d} t
\\
=\int_0^T  \|{\pmb 1}_{[\tau,T]}(t)e^{-i(t-\tau)A}f(\tau) \|^p_{ \D(A_q^{\frac{\vartheta }2-\frac{2}{3p}})} \, {\rm d} t&=
\int_{\tau}^T  \|e^{-i(t-\tau)A}f(\tau) \|^p_{ \D(A_q^{\frac{\vartheta }2-\frac{2}{3p}})} \, {\rm d} t\\
=\int_0^{T-\tau}  \|e^{-is A}f(\tau) \|^p_{ \D(A_q^{\frac{\vartheta }2-\frac{2}{3p}})} \, {\rm d} s &\leq
\int_0^{T}  \|e^{-is A}f(\tau) \|^p_{ \D(A_q^{\frac{\vartheta }2-\frac{2}{3p}})} \, {\rm d} s\\
&\leq C_T^p  \|f(\tau) \|^p_{ \D(A^{\frac{\vartheta }2})}.
\end{align*}
Therefore the  Minkowski inequality
yield
\begin{align*}
I
&\le \int_0^T \|F_\tau\|_{L^p(0,T; \D(A_q^{\frac{\vartheta }2-\frac{2}{3p}}))}\, {\rm d}\tau
\\
&
\leq C_T \int_0^T \|f(\tau)\|_{\D(A^{\frac{\vartheta }{2}})}\, {\rm d}\tau
= C_T\|f\|_{L^1(0,T,\D(A^{\frac{\vartheta }2}))}.
\end{align*}
\item [ii)]
When $r=p$ \cite[Theorem 3.10]{BrzezniakStrichartz} and \eqref{eqn-Strichartz} yield
\begin{equation*}
\left \Vert\int_0^{\cdot} e^{-i(\cdot- \tau)A}\xi(\tau)\, {\rm dW }(\tau)\right\Vert_{L^r(\Omega,L^p(0,T; L^q))} \lesssim_T \|\xi\|_{L^r(\Omega;L^2(0,T;\gamma(Y,\D(A^{\frac{2}{3p}}))))},
\end{equation*}
and reasoning as in \eqref{Str1} one obtains \eqref{eqn-Strichartz-213} with $r=p$.
For the case $r \ne p$ the result follows from \cite[Corollary 2.2]{Hornung-F_PhD}: our parameter $\theta$ is the same parameter $\theta$ that appears in that result, the parameter $\mu$ that appears there is equal to $\frac {4}{3p}$ in our case.
\end{itemize}
\end{proof}

\section{Proof of Proposition \ref{EstimatesGalerkinSolution} (c)}
\label{App_B}

\begin{proof}[Proof of Proposition \ref{EstimatesGalerkinSolution} (c)]
This proof has some similarities with that of Proposition \ref{EstimatesGalerkinSolution} (a).
However, here we look for a uniform estimate on the unbounded time interval $[0,+\infty)$.

We use the auxiliary process $Z(u)$ defined in \eqref{def-z-somma}:
\[
Z(u)= \|u\|^2_H + 2\energy(u)\equiv \|u\|^2_V+2\Fhat(u).
\]
We will prove that
\begin{equation}
\label{Est_y}
\sup_{t\ge 0}\mathbb{E}\left[Z(u_n(t))\right]< \infty,
\end{equation}
from which, estimate \eqref{Ebound} immediately follows.
In order to prove \eqref{Est_y} let us deal separately with the quantities $\mathbb{E}\left[\|u_n\|^2_H\right]$ and $\mathbb{E}\left[\energy(u_n)\right]$.
\newline
Applying the It\^o formula to the squared $H$-norm of $u_n$ (compare with the computations done in the proof of Proposition \ref{MassEstimateGalerkinSolution})	we obtain, almost surely for all $t\ge0$
 \begin{align*}
	\norm{u_n(t)}_{H}^2=&\norm{ P_nu_0}_{H}^2-2 \beta\int_0^{t}  \|u_n(s)\|_H^2
 \df s
 +\int_0^{t}
	\Vert  G\bigl(u_n(s)\bigr) \Vert_{\gamma(Y_2,H)}^2\df  s
+N_n(t),
	\end{align*}
where
$N_n(t)=2 \int_0^{t} \Real \skpH{u_n(s)}{ -\im G( u_n(s))  \df \textbf{W}(s)}$
is a martingale.
Taking the expected values on both sides we obtain \begin{align*}
\mathbb{E}\left[\norm{u_n(t)}_{H}^2 \right]
=\mathbb{E}[\norm{P_n u_0}_{H}^2]-2\beta \int_0^{t} \mathbb{E}\left[ \|u_n(s)\|_H^2\right]
 \df s
 + \int_0^{t}
	\mathbb{E} \left[\Vert  G\bigl(u_n(s)\bigr) \Vert_{\gamma(Y_2,H)}^2\right]\df  s ;
 \end{align*}	
 we write the above equation in the differential form
 \begin{equation*}
 \frac{{\rm d}}{{\rm d} t}\mathbb{E}\left[\norm{u_n(t)}_{H}^2 \right]
 =-2\beta \mathbb{E}\left[\norm{u_n(t)}_{H}^2 \right] + \mathbb{E} \left[\Vert  G\bigl(u_n(t)\bigr) \Vert_{\gamma(Y_2,H)}^2\right].
    \end{equation*}
 From Assumption \ref{assumption-stochastic} (iii) we infer
  \begin{equation}
  \label{Gronwall1}
  \frac{{\rm d}}{{\rm d} t}\mathbb{E}\left[\norm{u_n(t)}_{H}^2 \right]
  \le
  2C_1^2+2\left(\tilde{C}_1^2-\beta \right)\mathbb{E}\left[\norm{u_n(t)}_{H}^2 \right].
  \end{equation}
\newline
We now apply the It\^o formula to the energy functional $\energy$
(compare also with computations done in the proof of in Proposition \ref{EstimatesGalerkinSolution})
and obtain that, almost surely for all $t\ge 0$,
 \[\begin{split}
	\energy\left(u_n(t)\right)=&
		\energy\left(P_n u_0\right)	
	+\int_0^t \Real \duality{A u_n(s)+F(u_n(s))}{ b(u_n(s))-\beta u_n(s)} \df s+ M(t)
	  \\
	&+\frac{1}{2} \int_0^t  \Vert \sqrtA B u_n(s)\Vert_{\gamma(Y_1,H)}^2\df s
	+\frac{1}{2} \int_0^t  \Vert \sqrtA G  (u_n(s))\Vert^2_{\gamma{(Y_2,H)}}\df s
	\\
	&+\frac{1}{2}\int_0^t \sumM \Real \duality{F^{\prime}[u_n(s)] \left(B ( u_n(s))f_m\right)}{ B ( u_n(s))f_m} \df s
	\\
	&+\frac{1}{2}\int_0^t \sumM \Real \duality{F^{\prime}[u_n(s)] \left(G(  u_n(s))e_m\right)}{G( u_n(s))e_m} \df s,
	\end{split}\]	
where
\begin{align*}
M(t)=\int_0^t \Real \duality{A u_n(s)+F(u_n(s))}{& -\im  B \left( u_n(s)\right)\df W(s)}
\\
&	+\int_0^t \Real \duality{A u_n(s)+F(u_n(s))}{ -\im  G \left(u_n(s)\right)\df \newW(s)}
\end{align*}
is the sum of two martingales.
As above, we take the expected value on both sides of the above equality and we write the equation in its differential form as
\[\begin{split}
	\frac{{\rm d}}{{\rm d}t}\mathbb{E}[2\energy\left(u_n(t)\right)]
=&	
	2\mathbb{E}\left[\Real \duality{A u_n(t)+F(u_n(t))}{ b(u_n(t))-\beta u_n(t)}\right]
	  \\
	&+ \mathbb{E}\left[\Vert \sqrtA B u_n(t)\Vert_{\gamma(Y_1,H)}^2\right]	+ \mathbb{E}\left[ \Vert \sqrtA G  u_n(t)\Vert^2_{\gamma{(Y_2,H)}}\right]
	\\
	&+\mathbb{E} \left[\sumM \Real \duality{F^{\prime}[u_n(t)] \left(B ( u_n(t))f_m\right)}{ B ( u_n(t))f_m} \right]
	\\
	&+\mathbb{E} \left[\sumM \Real \duality{F^{\prime}[u_n(t)] \left(G(u_n(t))e_m\right)}{G(u_n(t))e_m} \right].
	\end{split}\]	
	
We now estimate the RHS of the above equality.
Recalling \eqref{Strat_cor} we have
\begin{equation}\begin{split}
\label{a1}
 2\mathbb{E}\left[\Real \duality{A u_n}{ b(u_n)}\right]
&\le
\mathbb{E} \left[\|A^{\frac 12}u_n\|_H \|B\|^2_{\mathscr{L}(V,\gamma(Y_1,V))}\|u_n\|_V\right]
\\
& \le
\|B\|^2_{\mathscr{L}(V,\gamma(Y_1,V))}\mathbb{E}\left[\|u_n\|^2_V\right]
\end{split}\end{equation}
and, thanks to \eqref{eqn_nonlinearityEstimate} and \eqref{eqn_boundantiderivative}, we obtain
\begin{equation}\begin{split}
\label{a2}
 2\mathbb{E}\left[\Real \duality{F(u_n)}{ b(u_n)}\right]
&\le
\mathbb{E} \left[ \|F(u_n)\|_{L^{\frac{\alpha+1}{\alpha}}}\sum_{m=1}^{\infty}\|B^2_m u_n\|_{L^{\alpha+1}}\right]
\\
&\le \|B\|^2_{\mathscr{L}(L^{\alpha+1},\gamma(Y_1,L^{\alpha+1}))}\mathbb{E}\left[\|u_n\|^{\alpha+1}_{L^{\alpha+1}}\right]
\\
&=(\alpha+1) \|B\|^2_{\mathscr{L}(L^{\alpha+1},\gamma(Y_1,L^{\alpha+1}))}\mathbb{E} \left[\hat F(u_n)\right].
\end{split}\end{equation}
We exploit \eqref{ineq-dissipativity} and \eqref{eqn_boundantiderivative} to get
\begin{equation}\begin{split}
\label{a3}
2\mathbb{E}\left[\Real \duality{A u_n+F(u_n)}{-\beta u_n}\right]
&= -2\beta\mathbb{E}\left[ \|A^{\frac 12 } u_n\|_H^2\right] -2\beta  \mathbb{E}\left[\|u_n\|^{\alpha+1}_{L^{\alpha+1}}\right]
\\
&= -2\beta \mathbb{E}\left[\|A^{\frac 12 } u_n\|_H^2\right] -2\beta (\alpha+1) \mathbb{E} [ \hat{F}(u_n) ].
\end{split}\end{equation}
We have
\begin{equation}
\label{a4}
   \mathbb{E}\left[\Vert \sqrtA B u_n\Vert_{\gamma(Y_1,H)}^2\right]	
     \le \|B\|^2_{\mathscr{L}(V,\gamma(Y_1,V))}\mathbb{E}\left[\|u_n\|^2_V\right]
\end{equation}
and, from \eqref{crescitaGEA},
\begin{equation}
\label{a5}
   \mathbb{E}\left[\Vert \sqrtA G (u_n)\Vert_{\gamma(Y_2,H)}^2\right]	
\le2(
C^2_2+\tilde{C}_2^2\mathbb{E}\left[\|u_n\|^2_V\right]).
\end{equation}
From \eqref{eqn_deriveNonlinearBound}, \eqref{eqn_boundantiderivative}  and Remark \ref{rem-HS} we obtain
\begin{multline}
\label{a6}
   \mathbb{E} \left[\sumM \Real \duality{F^{\prime}[u_n] \left((B  u_n)f_m\right)}{ (B  u_n)f_m} \right]
   \le
 \mathbb{E}\left[\|F'[u_n]\|_{L^{\alpha+1}\rightarrow L^{\frac{\alpha+1}{\alpha}}}\|B(u_n)\|^2_{\gamma(Y_1,L^{\alpha+1})}\right]
\\
\le \alpha\|B\|^2_{\mathscr{L}(L^{\alpha+1},\gamma(Y_1,L^{\alpha+1}))}\mathbb{E}\left[ \|u_n\|^{\alpha+1}_{L^{\alpha+1}}\right]
= \alpha(\alpha+1)\|B\|^2_{\mathscr{L}(L^{\alpha+1},\gamma(Y_1,L^{\alpha+1}))}\mathbb{E}\left[ \hat{F}(u_n)\right].
\end{multline}
Finally, from \eqref{eqn_deriveNonlinearBound}, \eqref{eqn_boundantiderivative}, \eqref{ineq-Fp} and \eqref{crescitaGL}, we obtain
\begin{equation}\begin{split}
\label{a7}
 \mathbb{E}   \sumM \Real  &\duality{F^{\prime}[u_n] \left(G(u_n)e_m\right)}{G(u_n)e_m}
    \le  \mathbb{E}\left[\|F'[u_n]\|_{L^{\alpha+1}\rightarrow L^{\frac{\alpha+1}{\alpha}}} \|G(u_n)\|^2_{\gamma(Y_1,L^{\alpha+1})}\right]
\\&
\le\alpha\mathbb{E}\left[\|u_n\|^{\alpha-1}_{L^{\alpha +1}} 2\left(C_3^2+ \tilde{C}_3^2\|u_n\|^2_{L^{\alpha+1}} \right)\right]
\\&
= 2 \alpha C_3^2\left(\alpha+1 \right)^{\frac{\alpha-1}{\alpha+1}}\mathbb{E}\left[\left(\hat{F}(u_n) \right)^{\frac{\alpha-1}{\alpha+1}}\right] +  2 \alpha \tilde{C}^2_3 (\alpha+1) \mathbb{E}\left[\hat{F}(u_n)\right]
\\&
\le
\frac 2{\alpha+1} \left(\frac{\alpha-1}{\varepsilon(\alpha+1)}\right)^{\frac {\alpha-1}2}
(2 \alpha C_3^2)^{\frac{\alpha+1}2}
+ \left(\varepsilon + 2\tilde{C}^2_3 \alpha\right)(\alpha +1) \mathbb{E}\left[\hat{F}(u_n)\right],
\end{split}\end{equation}
where in the last  estimate we exploited the Young inequality
\[
2\alpha C_3^2\left((\alpha+1) \hat{F}(u_n) \right)^{\frac{\alpha-1}{\alpha+1}}
\le
\varepsilon  (\alpha+1) \hat{F}(u_n) + \frac 2{\alpha+1} \left(\frac{\alpha-1}{\varepsilon(\alpha+1)}\right)^{\frac {\alpha-1}2}
(2 \alpha C_3^2)^{\frac{\alpha+1}2},
\]
for any $\varepsilon>0$.

Collecting  estimates \eqref{a1}-\eqref{a7} we get
\begin{multline}
\label{Gronwall2}
     \frac{{\rm d}}{{\rm d}t}\mathbb{E}[2\energy(u_n(t))]
	\le 2C_2^2+ \frac 2{\alpha+1} \left(\frac{\alpha-1}{\varepsilon(\alpha+1)}\right)^{\frac {\alpha-1}2}(2 \alpha C_3^2)^{\frac{\alpha+1}2}
	-2\beta \|A^{\frac 12 } u_n(t)\|_H^2
\\	+2\left(\tilde{C}_2^2+\|B\|^2_{\mathscr{L}(V,\gamma(Y_1,V))}\right)\mathbb{E}\left[ \|u_n(t)\|^2_V\right]
	\\+
(\alpha+1)\left((\alpha+1)\|B\|^2_{\mathscr{L}(L^{\alpha+1},\gamma(Y_1,L^{\alpha+1}))}+  \varepsilon + 2\alpha\tilde{C}^2_3   -2\beta \right)\mathbb{E}\left[ \hat{F}(u_n(t))\right].
\end{multline}
Recalling the definition of $Z(u_n)$, we now take the sum in both sides of inequalities \eqref{Gronwall1} and \eqref{Gronwall2}
\begin{multline}
\label{Gronwall3}
\frac{\rm d}{{\rm d}t}\mathbb{E}[Z(u_n(t))]
 \le
 2 C_1^2+ 2 C_2^2+ \frac 2{\alpha+1} \left(\frac{\alpha-1}{\varepsilon(\alpha+1)}\right)^{\frac {\alpha-1}2}(2 \alpha C_3^2)^{\frac{\alpha+1}2}
\\+	2\left(\tilde{C}_1^2-\beta
	+\tilde{C}_2^2+\|B\|^2_{\mathscr{L}(V,\gamma(Y_1,V))}\right)\mathbb{E}\left[ \|u_n(t)\|^2_V\right]
	\\+
(\alpha+1)\left((\alpha+1)\|B\|^2_{\mathscr{L}(L^{\alpha+1},\gamma(Y_1,L^{\alpha+1}))}+  \varepsilon + 2\alpha \tilde{C}^2_3   -2\beta \right)\mathbb{E}\left[ \hat{F}(u_n(t))\right].\end{multline}
Now we assume \eqref{condizione-beta}; then for a suitable choice of $\varepsilon$
both coefficients in front of the norms are negative.
Therefore there exist two positive constants $C_4$ and $C_5$ independent of $n$ such that
\begin{equation*}
 \frac{\rm d}{{\rm d}t}\mathbb{E}[Z(u_n(t))]\le C_4- C_5\mathbb{E}[Z(u_n(t))] .
\end{equation*}
From Gronwall's inequality we infer
\begin{equation*}
 \mathbb{E}[Z(u_n(t))] \le \mathbb{E}[Z(u_n(0))]e^{-C_5 t}+\frac{C_4}{C_5}\left(1-e^{-C_5 t}\right),
\end{equation*}
so
\begin{equation*}
\sup_{n \in \mathbb{N}}\sup_{t\ge 0}\mathbb{E}\left[Z(u_n(t)) \right] <\infty. \qedhere
\end{equation*}
\end{proof}

\section{Proof of Lemma \ref{convquadvari}}
\label{App_C}

\begin{proof}[Proof of Lemma \ref{convquadvari}]
Let $0\le s\le t< \infty$ and $\psi, \varphi \in \EA$. We define
	\begin{align*}
	g_{n}(t,s)&:=h(v_n|_{[0,s]}) \sumM \int_s^t  \skpHReal{S_n G( S_n v_n(r))e_m}{\psi} \skpHReal{S_n G( S_n v_n(r))e_m}{\varphi}\df r, \\
	g(t,s)&:=h(v|_{[0,s]}) \sumM \int_s^t  \dualityReal{G( v(r))e_m}{\psi} \dualityReal{G( v(r))e_m}{\varphi}\df r .
	\end{align*}
We will prove that the functions $\{g_{n}\}_{n \in \mathbb{N}}$ are uniformly integrable and converge
$\hat{\mathbb{P}}$-a.s. to $g$.
\\
$\bullet$ \textbf{$\hat{\mathbb{P}}$-a.s. convergence.}
  Because of $h(v_n |_{[0,s]}) \to h(v|_{[0,s]})$ $\hat\Prob$-a.s. and the continuity of the inner product
  $L^2([s,t]\times \N),$ the convergence
	\begin{align*}
	\skpHReal{S_n G (S_n v_n)e_m}{\psi} \to \dualityReal{G (v)e_m}{\psi}
	\end{align*}
$\hat\Prob$-a.s. in $L^2([s,t]\times \N)$ already implies
$g_n(t,s)\to g(t,s)$ $\hat\Prob$-a.s. Therefore, it is sufficient to prove
	\begin{equation*}
	\lim_{n\rightarrow \infty}\Vert\skpHReal{S_n G (S_n v_n)e_\cdot}{\psi} - \dualityReal{G (v)e_\cdot}{\psi}\Vert_\LzweiTimeSum=0 \qquad \ \hat{\mathbb{P}}-a.s.
	\end{equation*}
We estimate			
	\begin{align*}
	\Vert&\skpHReal{S_n G (S_n v_n)e_\cdot}{\psi} - \dualityReal{G (v)e_\cdot}{\psi}\Vert_\LzweiTimeSum
	\\
	&\le \norm{\skpHReal{G( S_n v_n)e_\cdot}{\left(S_n-I\right)\psi} }_\LzweiTimeSum
	+
	\norm{\skpHReal{ G(S_nv_n)e_\cdot-G(v_n)e_\cdot}{\psi} }_\LzweiTimeSum \\
	&\hspace{1cm}+\norm{ \dualityReal{G (v_n)e_\cdot-G(v)e_\cdot}{\psi}}_\LzweiTimeSum
	\\
	&=:I_1(n)+I_2(n)+I_3(n).
	\end{align*}	
We work pathwise. By means of \eqref{eqn-G crescita} and \eqref{SnUniformlyBounded} we estimate
	\[\begin{split}
	\norm{I_1(n)}_\LzweiTimeSum
	&=\norm{\skpHReal{G( S_n v_n)e_\cdot}{\left(S_n-I\right)\psi} }_\LzweiTimeSum
	\\
	&\le \left(\int_s^t\|GS_nv_n(r)\|^2_{\gamma(Y_2,H)}\, {\rm d}r \right)^{1/2} \|(S_n-I)\psi\|_{H}
	\\
	&\lesssim \left(1+\sup_{n \in \mathbb{N}}\|v_n\|^2_{L^{\infty}(0,T;H)}\right)^{1/2} \|(S_n-I)\psi\|_{\EA}.
				\end{split}\]
Bearing in mind, from Proposition
\ref{MassEstimateGalerkinSolution}, the boundedness of the sequence $(v_n)_n$ in
$L^{\infty}(0,T;H)$, for any $T>0$, we get  the  convergence to zero as $n \rightarrow \infty$ as a consequence of Proposition
\ref{PaleyLittlewoodLemma}.
\newline
Using \eqref{Lipschitz_G} we estimate

\begin{align*}
\norm{I_2(n)}_\LzweiTimeSum
	&=
\norm{\skpHReal{ G(S_nv_n)e_\cdot-G(v_n)e_\cdot}{\psi} }_\LzweiTimeSum	
\\
& \le \left(\int_s^t \sum_{m=1}^{\infty}\|G(S_nv_n(r))e_m-G(v_n(r))e_m\|^2_H\, {\rm d}r\right)^{\frac 12}\|\psi\|_H
\\
&= \left(\int_s^t \|G(S_nv_n(r))-G(v_n(r))\|^2_{\gamma(Y_2,H)}\, {\rm d}r\right)^{\frac 12}\|\psi\|_H
\\
&\le L_G\left(\int_s^t \|(S_n-I)v_n(r)\|_H^2\,{\rm d}r \right)^{\frac 12}\|\psi\|_H
\\
& \lesssim \|\psi\|_V \|S_n-I\|_{\mathscr{L}(V)} \sup_{n \in \mathbb{N}}\|v_n\|_{L^{\infty}(0,T,V)}.
\end{align*}
Recalling Corollary \ref{cor-aprioriEA}, about the boundedness of the sequence $(v_n)_n$ in $L^{\infty}(0,T;V)$, the
convergence to zero, as $n \rightarrow \infty$, follows again from Proposition \ref{PaleyLittlewoodLemma}.
\newline
The convergence to zero, as $n \rightarrow \infty$, of the last term
\begin{align*}
\norm{I_3(n)}_\LzweiTimeSum
	&=\norm{ \dualityReal{G (v_n)e_\cdot-G(v)e_\cdot}{\psi}}_\LzweiTimeSum
\end{align*}
follows as a consequence of the continuity of the norm $L^2([s,t] \times \mathbb{N})$, Assumption \eqref{cont_V^*} and \eqref{convergencev_n}.

$\bullet$	\textbf{Uniform integrability.} It is sufficient to show that, for some $r>1$,
\begin{equation*}
\sup_{n \ge 1}\hat{\mathbb{E}} \left[|g_{n}(t,s)|^r\right]< \infty, \ \qquad 0 \le s\le t\le T.
\end{equation*}
Let $r>1$, we estimate
	\begin{align}
	\label{g3}
	\hat{\mathbb{E}} \left[\vert g_{n}(t,s)\vert^r \right]
	&\le \hat{\mathbb{E}} \Big[\norm{\dualityReal{S_n G( S_n v_n)e_\cdot}{\psi}}_{L^2([s,t]\times \N)} ^r
	\norm{\dualityReal{S_n G( S_n v_n)e_\cdot}{\varphi}}_{L^2([s,t]\times \N)}^r \vert h(v_n|_{[0,s]})\vert^r\Big]
	\notag\\
	&\le \hat{\mathbb{E}} \left[\Bigg(\int_s^t\|G(S_nv_n(r))\|^2_{\gamma(Y_2,H)}\,{\rm d}r\Bigg)^{\frac r 2}\right] \norm{\psi}_\EA^r  \norm{\varphi}_\EA^r  \norm{h}_\infty^r
	\notag\\
	&\lesssim  \left(1+\sup_{n\ge 1} \hat{\mathbb{E}}\|v_n\|^r_{L^{\infty}(0,T;H)}\right)\norm{\psi}_\EA^r  \norm{\varphi}_\EA^r  \norm{h}_\infty^r,
 		\end{align}
which is finite thanks to Proposition \ref{MassEstimateGalerkinSolution}.
	
Using Vitali's Theorem, we finally obtain
	\begin{align*}
	\lim_{n\to \infty} \hat{\mathbb{E}} \left[g_{n}(t,s)\right]=\hat{\mathbb{E}} \left[g(t,s)\right],\qquad 0\le s\le t\le T,
	\end{align*}
which concludes the proof.	
\end{proof}

\section{Yamada-Watanabe Theorem for Stochastic Evolution Equations}
\label{sec-Yamada-Watanabe Theorem}

The infinite dimensional version of the
Yamada-Watanabe Theorem has a long history. As far as we are aware, the first time Yamada-Watanabe Theorem  was  mentioned in the infinite-dimensional setting
of Stochastic Evolution Equations (SEEs) was a paper by the first named author and G\c{a}tarek in \cite{Brz+Gat_1999} about  stochastic reaction diffusion equations.
In that paper  the classical version
of the  Yamada-Watanabe Theorem from \cite{Ikeda+Watanabe_1989} has been used but no details were provided.
A proper  formulation for mild solutions to SEEs  and a  detailed proof  have  been first given  by Ondrej\'at in \cite{Ondrejat_2004_Uniqueness}. Later on  Kunze \cite{Kunze_2013_Yamada} formulated and proved a similar result in a framework of  weak solutions to SEEs. Let us point out, see also section 4 of \cite{Brz+H+W-2019}, that in case when the pathwise uniqueness holds,
another avenue of proving the existence of strong solutions, not by the Prokhorov-Skorokhod Theorems, is possible. Namely, one can
 use the   Gy\"ongy and Krylov Lemma, see \cite[Lemma 1]{Gyongy+Krylov_1996}, to prove that  the approximations converge in probability and that the limit process is a strong solution.
This approach  has been recently used  by Crisan, Flandoli and  Holm \cite{Crisan+Flandoli+Holm_2019} but  it still  required the use of the Skorokhod
embedding theorem. As it was observed in \cite{Brz+H+W-2022}, it would be of interest to see if  this approach  works for the class of stochastic NLS studied in the present paper.

Returning to the topic of an infinite dimensional version of the Yamada-Watanabe Theorem let us emphasize that the present formulation differs from the formulations
from \cite[Theorem 2]{Ondrejat_2004_Uniqueness} and \cite[Theorem 5.3 and Corollary 5.4]{Kunze_2013_Yamada} since we consider only solutions with a given initial law.

The pathwise uniqueness and the existence of martingale solutions
imply the existence of  strong solutions, see e.g. \cite[Theorem 2]{Ondrejat_2004_Uniqueness} and \cite[Theorem 5.3 and Corollary 5.4]{Kunze_2013_Yamada}.

\begin{theorem}\label{thm-Yamada-Watanabe}
Assume that Assumptions \ref{assumption-A-space}, \ref{assumption-F_def} and \ref{assumption-stochastic} are satisfied.
Assume that ${\newr}\in [1,\infty)$ and that $\mu$ is a Borel probability measure on $\EA$ whose $2{\newr}$-th moment   is  finite. If
\begin{itemize}
\item [i)] there exists a \emph{martingale solutions} to equation \eqref{eqn-ProblemStratonovich} with the initial data $\mu$.\\
\item [ii)]
pathwise-uniqueness of solutions to  equation \eqref{eqn-ProblemStratonovich}
 holds, i.e. if two systems
\begin{equation}\label{eqn-mart sol-YW}
\bigl(\tilde{\Omega},\tilde{\F},\tilde{\Prob},\tilde{W},\tilde{\newW},\tilde{\Filtration},u_1\bigr)
\mbox{ and } \bigl(\tilde{\Omega},\tilde{\F},\tilde{\Prob},\tilde{W},\tilde{\newW},\tilde{\Filtration},u_2\bigr)
\end{equation}
are  \emph{martingale solutions} of the equation \eqref{eqn-ProblemStratonovich} with the initial data $\mu$, i.e. such that
\begin{equation}\label{eqn-initial law-YW}
\law{\mathbb{\tilde{P}}}{u_i(0)}=\mu \  \text{on}\ \mathscr{B}(V), \quad i=1,2,
\end{equation}
then
\[
\tilde{\Prob}\bigl(u_1(t)=u_2(t) \bigr)=1 \;\;\mbox{ for all } t\geq 0,
\]
\end{itemize}
then there exists a strong solution to   equation \eqref{eqn-ProblemStratonovich} with the initial data $\mu$.
\end{theorem}

\newcommand{\stetigBall}{C([0,T],\mathbb{B}_{\EA}^r)}

\newcommand{\strongBall}{C([0,\infty),\mathbb{B}_{\EA}^r)}

\newcommand{\LinftylocEA}{{L^\infty_{\mathrm{loc}}([0,\infty);\EA)}}

\newcommand\LinfEA[1]{{L^\infty(0,#1;\EA)}}
\newcommand\strongball[2]{C([0,#1],\mathbb{B}_{\EA}^{#2})}

\newcommand\strongEAdual[1]{{C([0,#1];\EAdual)}}

\section{A technical Lemma}	\label{Technical_Lemma}

Let us denote by $\mathbb{B}_{\EA}^r$ the ball in $V$ centered in zero with radius $r>0$. The following result provides a criterion for convergence of a sequence in $\strongBall$, where the ball $\mathbb{B}_{\EA}^r$ is equipped with the weak topology. We need this result in the proof of Proposition \ref{CompactnessDeterministic}.
\begin{lemma}\label{convergenceStetigBall}
	Let $(r_N)_{N=1}^\infty$ be a sequence of positive numbers  and $\left(u_n\right)_{n\in\N} \subset \LinftylocEA$ be a sequence with the properties
	\begin{enumerate}
		\item[a)] for every $N \in \mathbb{N}$, $\displaystyle\sup_{n\in\N} \norm{u_n}_{\LinfEA{N}} \le r_N$,
		\item[b)] for every $N \in \mathbb{N}$,  $u_n\to u$ in $\strongEAdual{N}$ for $n\to \infty.$
	\end{enumerate}
	Then, for every $n\in\N$, $u_n,u\in C_w([0,\infty); \EA) $.  Moreover, 
for every $N \in \mathbb{N}$ and  every $n\in\N$, 
 $u_n,u\in \strongball{N}{r_N}$  and 
 \[
 u_n \to u  \mbox{ in }\strongball{N}{r_N} \mbox{ as } n\to \infty.
 \]
\end{lemma}
\begin{proof}This proof is a minor modification of the proof of \cite[Lemma 4.1]{Brz+H+W-2019}.
	The Strauss-Lemma 
(see  \cite[Chapter 3, Lemma 1.4]{Temam_2001})  
and the assumptions guarantee that 	for every  $n\in \N$
	\begin{align*}
	u_n \in C([0,\infty); \EA^\ast) \cap \LinftylocEA \subset C_w([0,\infty); \EA) 
	\end{align*}
 and, for every $N \in \mathbb{N}$ and every  $n\in \N$, 
 \[ \sup_{t\in [0,N]} \norm{u_n(t)}_\EA \le  r_N.
 \]
  Hence, we infer that $u_n\in \strongball{N}{r_N}$ for all $n\in \N$ and $N\in \N$. 
  \\
  Let us now choose and  fix $N\in \N$. Then, for every $h\in \EA$
	\begin{align*}
	\sup_{s\in[0,N]} \left\vert \duality{u_n(s)-u(s)}{h}\right\vert \le \norm{u_n-u}_{C([0,N],\EAdual)} \norm{h}_\EA \to 0,\qquad n\to \infty.
	\end{align*}
Hence by Assumption  a) and the Banach-Alaoglu Theorem we find  a subsequence $\left(u_{n_k}\right)_{k\in\N}$ and $v\in L^\infty(0,N;\EA)$ such that  
\[
u_{n_k} \rightharpoonup^* v \mbox{ in } L^\infty(0,N;\EA).
\]
 Hence,  by the uniqueness of the weak star limit in  $L^\infty(0,N;\EAdual),$ we conclude  $u=v \in L^\infty(0,N;\EA)$ with $\norm{u}_{L^\infty(0,N;\EA)}\le r_N$. \\	
 \\
	Let $\varepsilon>0$ and $h\in \EAdual.$ By the density of $\EA$ in $\EAdual,$ we choose $h_\varepsilon\in \EA$ with $\norm{h-h_\varepsilon}_\EAdual\le \frac{\varepsilon}{4 r}$ and obtain for large $n\in \N$
and all $s\in[0,N]$
	\begin{align*}
	\left\vert \duality{u_n(s)-u(s)}{h}\right\vert&\le
	\left\vert \duality{u_n(s)-u(s)}{h-h_\varepsilon}\right\vert+
	\left\vert \duality{u_n(s)-u(s)}{h_\varepsilon}\right\vert\\
	&\le \norm{u_n(s)-u(s)}_\EA \norm{h-h_\varepsilon}_\EAdual+ \left\vert \duality{u_n(s)-u(s)}{h_\varepsilon}\right\vert\\
	&\le 2 r \frac{\varepsilon}{4 r}+\frac{\varepsilon}{2}=\varepsilon.
	\end{align*}
 This implies that 
	$\sup_{s\in[0,N]} \left\vert \duality{u_n(s)-u(s)}{h}\right\vert \to 0$
	as $n\to \infty$. By the arbitrariness of   $h\in  \EAdual$, we infer that   $u_n \to u$ in $C_w([0,N];\EA).$ Hence by \cite[Lemma A.2]{Brz+H+W-2019} 
we obtain the assertion. The proof of Lemma \ref{convergenceStetigBall} is complete.
\end{proof}

\Addresses

\end{document}